\documentclass{amsart}

\usepackage{graphicx}
\usepackage{geometry}

\usepackage{cleveref}
\usepackage{tikz-cd} 
\usepackage{amssymb}
\usepackage{amsmath}
\usepackage{amsfonts}
\usepackage{enumerate}
\usepackage{caption}
\usepackage{float}
\usepackage{bbold}
\usepackage{graphicx, adjustbox}
\usepackage{subfigure}
\usepackage{epstopdf}
\usepackage{algorithmicx}
\usepackage{algpseudocode}
\ifpdf
\DeclareGraphicsExtensions{.eps,.pdf,.png,.jpg}
\else
\DeclareGraphicsExtensions{.eps}
\fi

\newtheorem{theorem}{Theorem}[section]
\newtheorem{lemma}[theorem]{Lemma}
\newtheorem{proposition}[theorem]{Proposition}
\newtheorem{corollary}[theorem]{Corollary}
\theoremstyle{definition}

\theoremstyle{remark}
\newtheorem{remark}[theorem]{Remark}

\numberwithin{equation}{section}

\newtheorem{AlgDef}{Algorithm}[section]

\usepackage{amsmath}
\usepackage{tikz-cd}

\usepackage{amsopn}

\newcommand{\K}{\mathcal{U}}
\newcommand{\dd}{\mathrm{d}}
\newcommand{\Proj}{\mathcal{S}}

\newcommand{\jhatbold}{\boldsymbol{\hat{\textbf{\j}}}}
\newcommand{\shatbold}{\boldsymbol{\hat{\textbf{s}}}}
\newcommand{\lhatbold}{\boldsymbol{\hat{\textbf{l}}}}
\newcommand{\jhat}{\hat{{\j}}}
\DeclareMathOperator{\spn}{span}

\begin{document}

\title{On the approximation of Koopman spectra for measure preserving transformations}

\author{N. Govindarajan}
\address{Dept. of Mechanical Eng., Univerity of California at Santa Barbara, Santa Barbara, CA 93106}
\email{ngovindarajan@engineering.ucsb.edu}
\thanks{The authors gratefully acknowledge support from the Army Research Office (ARO) through grant W911NF- 11-1-0511 under the direction of program manager Dr. Samuel Stanton.}

\author{R. Mohr}
\address{Dept of Mechanical Eng., Univerity of California at Santa Barbara, Santa Barbara, CA 93106}
\email{mohrrm@engineering.ucsb.edu}

%    Information for third author
\author{S. Chandrasekaran}
\address{Dept. of Electrical and Computer Eng., University of California at Santa Barbara, Santa Barbara, CA 93106}
\email{shiv@ece.ucsb.edu}

%    Information for fourth author
\author{I. Mezi\'{c}}
\address{Dept. of Mechanical Eng., Univerity of California at Santa Barbara, Santa Barbara, CA 93106}
\email{mezic@engineering.ucsb.edu}

%    General info

\date{\today.}


\begin{abstract}
For the class of continuous, measure-preserving automorphisms on compact metric spaces, a procedure is proposed for constructing a sequence of finite-dimensional approximations to the associated Koopman operator on a Hilbert space. These finite-dimensional approximations are obtained from the so-called ``periodic approximation'' of the underlying automorphism and take the form of permutation operators. Results are established on how these discretizations approximate the Koopman operator spectrally. Specificaly, it is shown that both the spectral measure and the spectral projectors of these permutation operators converge weakly to their infinite dimensional counterparts. Based on this result, a numerical method is derived for computing the spectra of volume-preserving maps on the unit $m$-torus. The discretized Koopman operator can be constructed from solving a bipartite matching problem with $\mathcal{O}(\tilde{n}^{3m/2})$ time-complexity, where $\tilde{n}$ denotes the gridsize on each dimension. By exploiting the permutation structure of the discretized Koopman operator, it is further shown that the projections and density functions are computable in $\mathcal{O}( m \tilde{n}^{m} \log \tilde{n})$ operations using the FFT algorithm. Our method is illustrated on several classical examples of automorphisms on the torus that contain either a discrete, continuous, or a mixed spectra. In addition, the spectral properties of the Chirikov standard map are examined using our method.
\end{abstract}

\maketitle

\section{Introduction}
The state-space of a measure-preserving dynamical system may be inspected through an analysis of the spectral properties of the associated Koopman operator on a Hilbert space. Within the applied dynamical systems community, a lot of interest has arisen recently in the computation of spectral properties of the Koopman operator (see the review article \cite{Budisic2012}). On the attractor, these spectral properties are directly related to the statistical properties of a measure-preserving dynamical system. Instead of the individual trajectories, it is exactly these properties that are relevant when it comes down to the comparison of complex dynamics \cite{Mezic2004}. From an engineering and modeling perspective, one would like to exploit this fact to obtain reduced order models. It was pointed out in \cite{Mezic2005} that for measure-preserving dynamical systems, the Koopman operator often contains a mixed spectra. The discrete part of the spectrum correspods to the almost periodic part of the process, whereas the continuous spectrum typically is associated with either chaotic or shear dynamics. In a model reduction application, it would be desirable to use the almost periodic component of the observable dynamics as a reduced order model, while correctly accounting for the remainder contributions with the appropriate noise terms.

A prereqisuite to the development of Koopman-based reduced order models are numerical methods which are capable of approximating the spectral decomposition of observables. A lot of work has already been dedicated on this subject. Apart from taking direct harmonic averages \cite{mezic1999method,levnajic2010ergodic,levnajic2015ergodic,mauroy2012use}, most methods rely mostly on the Dynamic Mode Decomposition (DMD) algorithm \cite{rowley2009spectral,schmid2010dynamic}. Techniques that fall within this class are Extended-DMD \cite{williams2015data,korda2017convergence} as well as Hankel-DMD \cite{arbabi2016ergodic}. Certain spectral convergence results were established for the algorithms. In \cite{arbabi2016ergodic}, it was shown that Hankel-DMD converges for systems with ergodic attractors. On the other hand,  \cite{korda2017convergence} proved weak-convergence of the eigenfunctions for Extended-DMD. Nevertheless, little is known on how these methods deal with continuous spectrum part of the system. 

Numerically, computation of continuous spectra poses a challenge since it involves taking spectral projections of observables onto sets in the complex-plane which are not singleton. One approach is to directly approximate the moments of the spectral measure, and use the Christoffel-Darboux kernel to distinguish the atomic and continuous parts of the spectrum. This route was advocated in \cite{korda2017data}. In this paper, we present an alternative viewpoint to the problem which directly takes aim at a specific discretization of the unitary Koopman operator. Our method relies on the concept of \emph{``periodic approximation'''} which originally was introduced by Halmos \cite{Halmos1944} and later by Lax \cite{Lax1971}. The idea is similar to the so-called ``Ulam approximation'' of the Perron-Frobenius operator \cite{ulam1960collection,li1976finite,ding1996finite}, but instead of approximating the measure-preserving dynamics by Markov chains (see also \cite{dellnitz1999approximation,dellnitz2001algorithms}), one explicitly enforces a bijection on the state-space partition so that the unitary structure of the operator is preserved. In seminal work conducted by Katok and Stepin \cite{Katok1967}, it was already shown how these type of approximations are related to certain qualitative properties of the spectra. In this paper, we show how this technique can be used to obtain an actual approximation of the spectral decomposition of the operator. 

We point out that the spectrum is approximated only in a weak sense. Indeed, the spectral type of a dynamical system is sensitive to arbitrarily small perturbations, making numerical approximation of the actual operator spectra very challenging. Take for example the map:
$$ T_{\epsilon}(x) = (x_1, x_2 + \omega + \epsilon \sin (2 \pi x_1)),$$
defined on the unit torus. $T_{\epsilon}$  has a fully discrete spectrum at $\epsilon=0$, but an absolutely continuous spectrum \cite{mezic2017koopman} when $\epsilon \ne 0$ . Even though the spectral type changed instanteously, the observable dynamics will not differ very much at finite time-scales under arbitrarily small perturbations. This small change in dynamics is reflected by only small perturbations in the spectral measure of the obervable.

\subsection{Contributions} The main contribution of this paper is a rigorous discretization of the Koopman operator  such that its spectral properties are approximated weakly in the limit. It is shown that despite the difficulties of not knowing the smoothness properties of the underlying spectral projectors, control on the error may still be achieved in an average sense. Our discretization gives rise to a convergent numerical method which is capable of computing the spectral decomposition of the unitary Koopman operator. 

We describe how our numerical method can be implemented for volume preserving maps on the $m$-torus. The construction of periodic approximations for such kind of maps is a relatively straightforward task and it has been well studied before in the literature under the phrase of ``lattice maps'' \cite{diamond1993numerical, kloeden1997constructing,earn1992exact,rannou1974numerical}. In our paper, we show that if the map is Lipschitz continuous, a periodic approximation can always be obtained from solving a bipartite matching problem of $\mathcal{O}(\tilde{n}^{3m/2})$ time-complexity, where $\tilde{n}$ denotes the size of the grid in each dimension. We further show that, once a periodic approximation is achieved, the spectral projections and density plots are computable in $\mathcal{O}(m\tilde{n}^{m} \log \tilde{n})$ operations. Here, we heavily exploit the permutation structure of the discretized operator by making use of the FFT algorithm. 
To illustrate our methodology, we will compute the spectra of some canonical examples which either have a discrete, continuous or mixed spectra. In addition, we analyze the spectral properties of the Chirikov Standard map \cite{Chirikov1979} with our method.

\subsection{Paper outline} In order to keep the presentation organized, the paper is split in two parts. In the first part, we cover the theoretical underpinnings of our proposed method by proving spectral convergence for a large class of measure preserving transformations.
In the second part, we proceed by applying this general methodology  to volume preserving maps on the $m$-torus. Numerical examples are provided. 

Section 2 describes the outline of the proposed discretization of the Koopman operator. Section 3 elaborates on the notion of periodic approximation. In \cref{sec:opconvergence} we prove operator convergence. Section 6 proves the main spectral results. In \cref{sec:numericalmethod}, details of the numerical method are presented. In \cref{sec:examples}, the performance of the numerical method is analyzed on some canonical examples of Lebesgue measure-preserving transformations on the torus. In \cref{sec:chirikov} the spectra of the Chirikov standard map is analyzed. Conclusions and future directions of our work are found in \cref{sec:conclusions}. 

\part{Theory: A finite-dimensional approximation of the Koopman operator with convergent spectral properties}

In the first part of the paper, we elaborate on the theory behind the discretization of the Koopman operator using periodic approximations. The theory is discussed under a very general setting. Since the notion of periodic approximation is not  as well known as the Ulam approximation in the applied dynamical systems community, we provide some background in places where the more experienced reader would easily stroll through. The main spectral results which are relevant to the numerical computations are found in \cref{sec:spectralconv}.

\section{Problem formulation} \label{sec:probdescription}
Let $T: X \mapsto X$ be a self-map on the compact, norm-induced metric space $X\subseteq \mathbb{R}^m$. Associate with $X$ the measure space $( X, \mathcal{M}, \mu )$, where $\mathcal{M}$ denotes the Borel sigma-algebra and $\mu$ is an absolutely continuous measure with its support  equaling the state-space, i.e. $\mathrm{supp}\mbox{ }\mu =  X$. The map $T$ is assumed to be an invertible measure-preserving transformation such that $\mu(B) = \mu(T(B)) = \mu(T^{-1}(B))$ for every $B\in\mathcal{M}$. In this paper, we refer to such maps as measure-preserving \emph{automorphisms}.  

The Koopman linearization \cite{Koopman1931} of an automorphisms is performed as follows. Let:
$$ L^2(X, \mathcal{M}, \mu) :=  \left\{ g:X\mapsto \mathbb{C}  \quad | \quad  \left\| g \right\| < \infty \right\}, \qquad \left\| g \right\| := \left(\int_{X} |g(x)|^2 \dd\mu(x) \right)^{\frac{1}{2}}  $$ 
denote the space of square-integrable functions on $X$ with respect to the invariant measure $\mu$. The \emph{Koopman operator} $\K: L^2(X, \mathcal{M}, \mu) \mapsto  L^2(X, \mathcal{M}, \mu)$ is defined as the composition:
\begin{equation}
(\K g)(x) :=    g \circ T(x) \label{eq:KOOPMAN}
\end{equation}
where $g \in L^2(X, \mathcal{M}, \mu)$ is referred to as the \emph{observable}. The Koopman operator is  \emph{unitary} (i.e. $\K^*  = \K^{-1}$), and hence, from the integral form of the spectral theorem \cite{Akhiezer1963,MacLane1971} it follows that the evolution of an observable under \eqref{eq:KOOPMAN} can be decomposed as:
\begin{equation} 
\K^k g = \int_{-\pi}^{\pi} e^{i \theta k} \dd\Proj({\theta}) g, \qquad k\in \mathbb{Z}.  \label{eq:decomposition}
\end{equation}
Here, $\Proj({\theta})$ denotes a self-adjoint, \emph{projection-valued measure} on the Borel sigma-algebra $\mathcal{B}([-\pi,\pi))$ of the circle parameterized by $\theta\in[-\pi,\pi)$. The {projection-valued measure} satifies the following properties:
\begin{enumerate}[(i)]
	\item For every $D \in \mathcal{B}([-\pi,\pi))$, 
	$$ \Proj_D := \int_{D} \dd\Proj({\theta})$$
	is an orthogonal projector on  $L^2(X, \mathcal{M}, \mu)$.
	\item $\Proj_D = 0$ if $D=\emptyset$ and  $\Proj_D = I$ if $D= [-\pi,\pi)$.
	\item If $D_1, D_2 \in \mathcal{B}([-\pi,\pi)) $ and $D_1 \cap D_2 = \emptyset$, then
	$$ \left\langle \Proj_{D_1} g   ,  \Proj_{D_2} h \right\rangle := \int_X \left(\Proj_{D_1} g\right)^*(x)  \left(\Proj_{D_2} h\right)(x)   \dd\mu(x) =  0$$
	for every $g,h\in L^2(X, \mathcal{M}, \mu)$.
	\item If $\left\{ D_k \right\}^{\infty}_{k=1}$ is a sequence of pairwise disjoint sets in $\mathcal{B}([-pi,\pi))$, then
	$$ \lim_{m\rightarrow\infty} \sum_{k=1}^m  \Proj_{D_k} g =  \Proj_{D} g, \qquad D:= \bigcup_{k=1}^{\infty}  D_k   $$
	for every $g\in L^2(X, \mathcal{M}, \mu)$.
\end{enumerate}

Numerically, our objective is to compute two quantities. Firstly, we wish to obtain a finite-dimensional approximation to the \emph{spectral projection}:
	\begin{equation}
	\Proj_D g = \int_{D} \dd\Proj({\theta}) g   \label{eq:target}   
	\end{equation}
	for some given observable $g\in L^2(X,\mathcal{M},\mu)$ and  interval $D \subset [-\pi,\pi)$. Secondly, we wish to obtain an  approximation to the \emph{spectral density function}. Let $\mathcal{D}([-\pi,\pi))$ denote the space of smooth test functions on the circle and $\mathcal{D}^{*}([-\pi,\pi))$ the dual space of distributions.  The {spectral density function} $\rho(\theta;g)\in \mathcal{D}^{*}([-\pi,\pi))$ is defined as the distributional derivative: 
	\begin{equation}
	\int_{-\pi}^{\pi} \varphi'(\theta) c(\theta; g)  \dd\theta =  - \int_{-\pi}^{\pi} \varphi(\theta) \rho(\theta;g)  \dd\theta, \quad \varphi(\theta)\in\mathcal{D}([-\pi,\pi)),    \label{eq:spectraldensity}
	\end{equation}
	of the so-called spectral cumulative function on $[-\pi,\pi)$:
	$$ c(\theta; g) :=  \langle \Proj_{[-\pi,\theta)} g , g \rangle.$$
	
The approximation of the spectral projections and density functions play a critical role in the applications of model reduction \cite{Mezic2004, Mezic2005}. Indeed, on one hand, a low order model is directly obtained from the spectral projections by selecting $D$ to be a finite union of disjoint intervals which contain a large percentage of the spectral mass. On the other hand, the spectral density functions directly tell us where most of the mass is concentrated for a given observable.	

\section{The proposed discretization of the Koopman operator} \label{sec:mainresult}
In this section, we introduce our proposed discretization of the Koopman operator.

\subsection{Why permutation operators?}
We motivate our choice of using permutation operators as a means to approximate \eqref{eq:KOOPMAN} as follows.  Just like the Koopman operator, permutation operators are unitary, and therefore, its spectrum is contained on the unit circle. But the Koopman operator also satisfies the following properties:
\begin{enumerate}[(i)]
	\item $\K (fg) = (\K f) (\K g)$. 
	\item $\K$ is a positive operator, i.e.  $\K g > 0$ whenever $g>0$.
	\item The constant function, i.e. $g(x) = 1$ for every $x\in X$, is an invariant of the operator.
\end{enumerate}
Permutation operators are (finite-dimensional) operators which satisfy the aforementioned properties as well. Overall, they seem to be a natural choice to approximate \eqref{eq:KOOPMAN}. The upcoming results will justify this claim even further.

\subsection{The discretization procedure} \label{sec:periodicapprox} 
The following discretization of the Koopman operator is proposed. Consider \emph{any} sequence of measurable partitions  $\left\{\mathcal{P}_n\right\}^{\infty}_{n=1}$, where  $\mathcal{P}_n := \left\{ p_{n,1}, p_{n,2}, \ldots,  p_{n,q(n)} \right\}$ such that:
\begin{enumerate}[(i)]
	\item Every partition element $p_{n,j}$ is compact, connected, and of equal measure, i.e.
	\begin{equation}\mu(p_{n,j}) = \frac{\mu(X)}{q(n)}, \qquad j\in\left\{1,2,\ldots,q(n)\right\} \label{eq:equalsized} \end{equation}
	where $q:\mathbb{N} \mapsto \mathbb{N}$ is a strictly, monotonically increasing function. Asking for compactness, one gets that the partition elements intersect. By \eqref{eq:equalsized},  it must follow that these intersections are of zero measure.
	\item The diameters of the partition elements are bounded by 
	\begin{equation} \mathrm{diam}(p_{n,j}):= \sup_{x,y\in p_{n,j}} d( x, y) \leq l(n)  \label{eq:partitioncond}
	\end{equation}
	where $l:\mathbb{N} \mapsto \mathbb{R}$ is a positive, monotonic function decaying to zero in the limit.
	\item $\mathcal{P}_n$ is a refinement of $\mathcal{P}_m$ whenever $n>m$. That is, every $p_{m,j} \in \mathcal{P}_m$ is the union of some partition elements in $\mathcal{P}_n$.
\end{enumerate}
\begin{remark}
	Such a construction is possible since the invariant measure $\mu$ is absolutely continuous with respect to the Lebesgue measure and  $\mathrm{supp}\mbox{ }\mu =  X$. If $\mu$ were to have singular components and/or $\mathrm{supp}\mbox{ }\mu \subset X$, then one would run into trouble satisfying conditions (i) and (ii).
\end{remark}
The main idea set forth in this paper is to project observables $g\in L^2(X,\mathcal{M},\mu)$ onto a finite-dimensional subspace of indicator functions,
$$  L^2_n(X, \mathcal{M}, \mu) :=  \left\{ g_n : X\mapsto \mathbb{C}  \quad | \quad  \sum_{j=1}^{q(n)} c_j  \chi_{p_{n,j}}(x), \quad c_j\in\mathbb{C}   \right\},  \qquad  \chi_{p_{n,j}}(x)  = \begin{cases} 1 & x \in p_{n,j}  \\ 0  & x \notin p_{n,j} \end{cases} $$
by means of a smoothing/averaging operation:
\begin{equation} 
(\mathcal{W}_n g)(x) = g_{n}(x) := \sum_{j=1}^{q(n)} g_{n,j}\chi_{p_{n,j}}(x), \qquad g_{n,j} = \frac{q(n)}{\mu(X)}\int_{X} g(x) \chi_{p_{n,j}}(x) \dd\mu(x) \label{eq:averageoperator}
\end{equation}
and then replace \eqref{eq:KOOPMAN} by its discrete analogue $\K_n:  L^2_n(X, \mathcal{M}, \mu) \mapsto L^2_n(X, \mathcal{M}, \mu)$ given by
\begin{equation} 
\left(\K_n g_n\right)(x) := \sum_{j=1}^{q(n)} g_{n,j} \chi_{T^{-1}_n(p_{n,j})}(x)  \label{eq:discreteoperator}
\end{equation}
where $T_n: \mathcal{P}_n \mapsto \mathcal{P}_n$ is a \emph{discrete} map on the partition.

The map $T_n$ is chosen such that it ``mimics'' the dynamics of the continuous map $T$. We impose the condition  that $T_n$ is a bijection. By doing so, we obtain a \emph{periodic approximation} of the dynamics. Since every partition element is of equal measure \eqref{eq:equalsized} and since $T_n$ is a bijection, the resulting discretization is regarded as one which preserves the measure-preserving properties of the original map on the subsigma algebra generated by $\mathcal{P}_n$, i.e.  $\mu(T^{-1} (p_{n,j})) = \mu(p_{n,j}) = \mu(T^{-1}_n (p_{n,j}))$. 

The discrete operators $\left\{\K_n \right\}^{\infty}_{n=1}$ are isomorphic to a sequence of finite-dimensional permutation operators. The spectra for these operators simplify to a pure point spectrum, where the eigenvalues correspond to roots of unity. Let $v_{n,k}= \sum (v_{n,k})_j  \chi_{p_{n,j}} \in L^2_n(X, \mathcal{M}, \mu)$ denote a normalized eigenvector, i.e.
$$ \K_n v_{n,k} = e^{i\theta_{n, k}} v_{n,k}, \qquad  \left\| v_{n,k} \right\| = 1. $$
The spectral decomposition can be expressed as:
\begin{equation} \K_n g_n = \sum^{q(n)}_{k=1}  e^{i\theta_{n,k}} \Proj_{n} (\theta_{n, k}) g_n \label{eq:eigdecomp}
\end{equation}
where $\Proj_{n} (\theta_{n, k}): L^2_n(X,\mathcal{M},\mu) \mapsto  L^2_n(X,\mathcal{M},\mu)$ denotes the rank-1 self-adjoint projector:
\begin{equation}  \Proj_{n} (\theta_{n, k}) g_n =  v_{n,k} \left\langle v_{n,k}, g_n  
\right\rangle  =  v_{n,k} \left(\int_X v^*_{n,k}(x) g_n(x) \dd \mu \right) =  v_{n,k} \left(  \frac{\mu(X)}{q(n)} \sum_{j=1}^{q(n)}   v^*_{n,kj} g_{n,j}\right).
\end{equation}
The discrete analogue to the spectral projection \eqref{eq:target} may then be defined as:
\begin{equation}
\Proj_{n, D} g_n =  \displaystyle \sum_{\theta_{n,k} \in D}  \Proj_{n} (\theta_{n, k}) g_n \label{eq:targetapprx}
\end{equation}
In addition, the discrete analogue to the spectral density function turns out to be:
\begin{equation}   
\rho_n(\theta;g_n) =  \sum_{k=1}^{q(n)} \left\| \Proj_{n} (\theta_{n, k}) g_n \right\|^2 \delta(\theta - \theta_{n,k}). \label{eq:discrete}
\end{equation}

\clearpage
\subsection{Overview} 
An overview of the discretization process is given in \cref{fig:discretization}.
Our goal is to see whether it is possible to construct a sequence of periodic approximation $\left\{T_n: \mathcal{P}_n \mapsto \mathcal{P}_n \right\}^{\infty}_{n=1}$ such that the associated discrete operators \eqref{eq:discreteoperator} and \eqref{eq:targetapprx} converge to their infinite dimensional counterparts in some meaningful sense.

\begin{figure}[h!]
	\begin{center}
		\includegraphics[width=.95\textwidth]{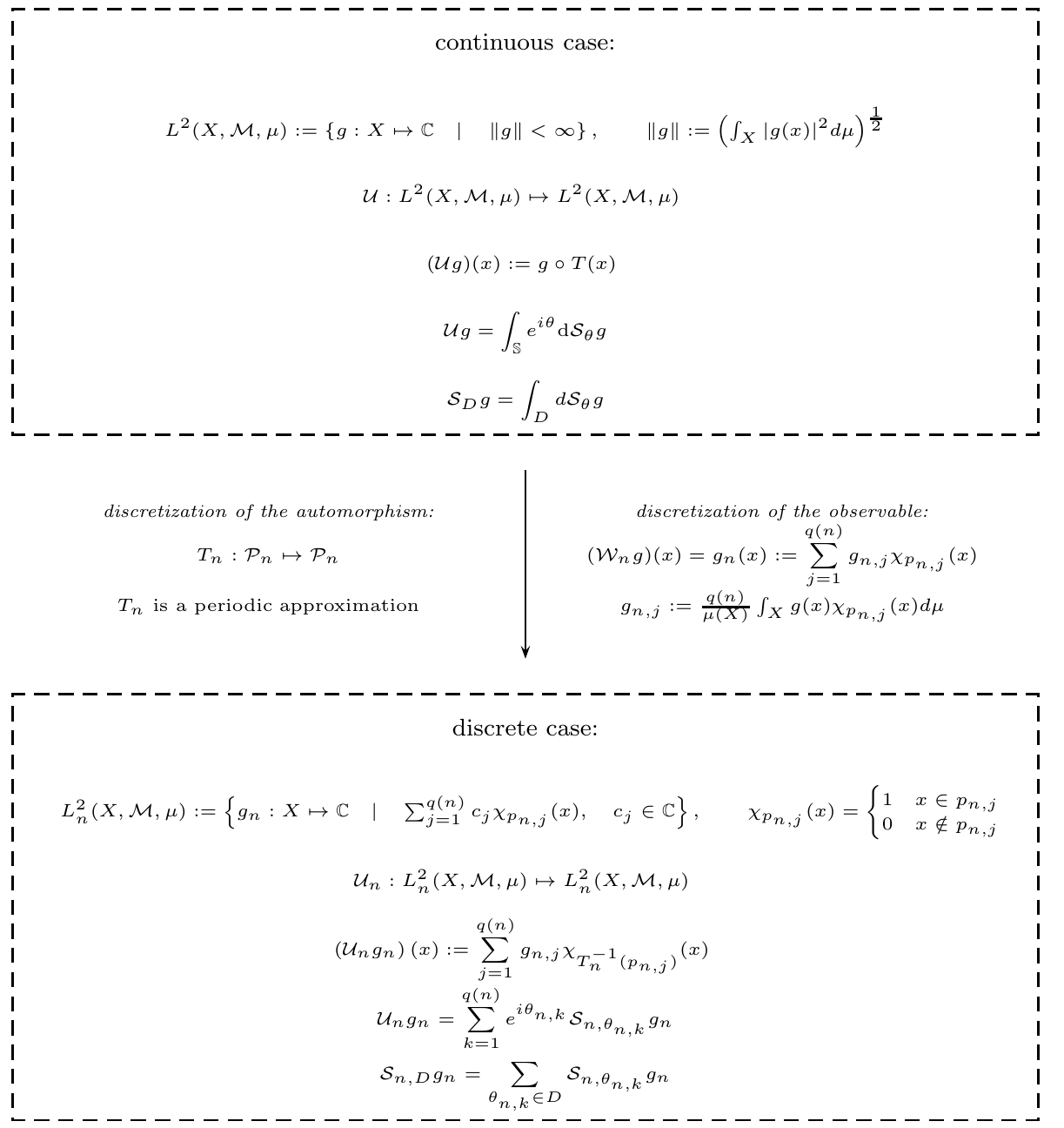} \\
	\end{center} 
	\caption{An overview of the discretization process.} \label{fig:discretization}
\end{figure}

\section{Periodic approximations} \label{sec:periodicapprx}

The notion of periodic approximation of a measure-preserving map is not new \cite{Halmos1944,Halmos1944a,Lax1971, Oxtboby2016,Katok1967,Rokhlin1949,Schwartzbauer1972}. It is well-known that the set of measure-preserving automorphisms, which form a group under composition, are densely filled by periodic transformations under various topologies (see e.g. \cite{Halmos1944}). These properties were useful in proving claims such as whether automorphisms are ``generically'' ergodic or mixing. Katok and Stepin \cite{Katok1967} further showed that the rate at which an automorphism admits a periodic approximation can be directly related to properties of ergodicity, mixing, and entropy. Certain spectral results were also proven in this context. For example, it was shown that if $T$ admits a cyclic approximation by periodic transformations at a certain rate, then the spectrum of \eqref{eq:KOOPMAN} is simple. The motivations of these earlier work was different than ours, given that the focus was more on characterizing global properties of automorphisms, rather than formulating a numerical method to approximate its spectral decomposition. Here, the focus will be on the latter, and with that goal in mind, we will prove the following result.

\begin{theorem}[Existence of a periodic approximation] \label{thm:important}
	Let $T:X\mapsto X$ be a continuous, invertible, measure-preserving automorphism with the invariant measure $\mu$ absolutely continuous w.r.t. the Lebesgue measure and $\mathrm{supp}(\mu) = X$. If $\left\{\mathcal{P}_n\right\}^{\infty}_{n=1}$ is a sequence of measurable partitions on $X$ which are refinements and satisfy the conditions \eqref{eq:equalsized} and \eqref{eq:partitioncond}, then there exists a sequence of bijective maps $\left\{T_n: \mathcal{P}_n \mapsto \mathcal{P}_n \right\}^{\infty}_{n=1}$ that periodically approximates $T$ in an asymptotic sense. More specifically, for every fixed $k\in \mathbb{N}$ and compact set $A\in \mathcal{M}$:
	\begin{equation} 
	\lim_{n \rightarrow \infty} \sum^{k}_{l = -k}  d_H(T^l(A), T^l_n(A_n))  = 0 \label{eq:setconvergence}
	\end{equation}
	where $d_H(A,B) := \max\{  \sup_{a\in A} \inf_{b\in B} d(a,b),  \sup_{b\in B} \inf_{a\in A} d(a,b)   \}$ denotes the Haussdorf metric, and 
	$$
	A_n := \bigcup_{p\in \mathcal{P}_n:\mbox{ } p \cap A \ne \emptyset } p
	$$
	is an over-approximation of $A$ by the partition elements of $\mathcal{P}_n$.
\end{theorem}

The proof of this result is postponed to the end of this subsection and will involve two intermediate steps. From a numerical analysis standpoint, \cref{thm:important} claims that one can make the finite-time set evolution of $T_n$ \emph{numerically indistinguishable} from that of the true map in both the forward and backward direction by choosing $n\in\mathbb{N}$ sufficiently large. In \cref{fig:periodicillustration} the situation is sketched for one forward iteration of the map.

\begin{figure}
	\begin{center}
		\includegraphics[width=.6\textwidth]{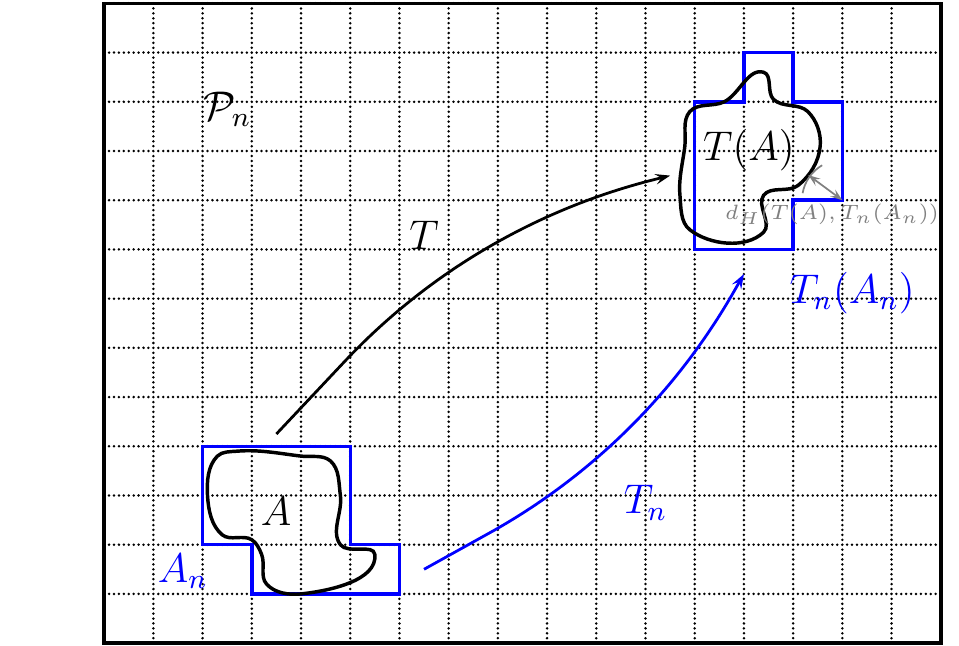} \\
	\end{center} 
	\caption{Shown is a partition $\mathcal{P}_n$ of a compact domain $X$ into square boxes. The set $A$ is (over-)approximated by the partition elements with $A_n$. As a consequence of \cref{thm:important},  the distance between the images: $T(A)$ and $T_n(A_n)$, must converge to zero in the Haussdorf metric as the partition is consecutively refined.} \label{fig:periodicillustration}
\end{figure}

We note that our formulation of the periodic approximation is more along the lines of that proposed by Lax \cite{Lax1971} and not equivalent to the ones proposed by Katok and Stepin \cite{Katok1967}. There, the quality of the approximation was phrased in the measure-theoretic context, where the proximity of $T_n$ to $T$ is described in terms of the one-iteration cost \cite{Katok1967}:
\begin{equation} 
\displaystyle \sum^{q(n)}_{i = 1}  \mu \left( T(p_{n,i})  \Delta  T_n(p_{n,i}) \right)    \label{eq:measuretheorysense} 
\end{equation}
with $\Delta$ denoting the symmetric set difference, i.e. $A \Delta B := (A\backslash B) \cup (B \backslash A)$. It turns out that a specific sequence of maps $\{T_n :\mathcal{P}_{n}\mapsto \mathcal{P}_{n} \}_{n=1}^{\infty}$ may converge in the sense of \cref{thm:important}, while not converging in the sense of \eqref{eq:measuretheorysense}. Relatively simple examples of such sequences may be constructed. Take for example the map $T(x) = (x + \frac{1}{2}) \mod 1$ on the unit-length circle and choose a partition $\mathcal{P}_{n} = \left\{p_{n, 1}, p_{n, 2}, \ldots, p_{n, r^n} \right\}$ with $p_{n, i} = \left[\frac{i-1}{r^n} ,\frac{i}{r^n} \right]$ and $r$ being odd. The mapping:
$$	
T_{n}(p_{n, j}) =  p_{n, j^*}, \qquad j^* = \begin{cases} j + \lfloor \frac{r^n}{2} \rfloor  &  j + \lfloor \frac{r^n}{2} \rfloor \leq r^n \\ j + \lfloor \frac{r^n}{2} \rfloor - r^n   &  j + \lfloor \frac{r^n}{2} \rfloor > r^n    \end{cases},
$$
is the best one can do in terms of the cost \eqref{eq:measuretheorysense}, yet
$$ \sum^{q(n)}_{i = 1}  \mu \left( T(p_{n,i})  \Delta  T_n(p_{n,i}) \right) = 1, \qquad \forall n\in\mathbb{N}. $$
We also note that the convergence results of the measure-theoretic formulation (see \cite{Halmos1944}) were stated for general automorphisms. Here, we restrict ourselves to just continuous automorphisms.

\begin{lemma} \label{lemm:beautifull}
	Let $T:X\mapsto X$ satisfy the hypothesis stated in \cref{thm:important}. Then, for any partition $\mathcal{P}_n$ that satisfies the condition \eqref{eq:equalsized}, there exists a bijection $T_{n}: \mathcal{P}_{n} \mapsto \mathcal{P}_{n}$ with the property:
	\begin{equation}
	T(p_{n, l}) \cap T_n (p_{n, l}) \ne \emptyset  \quad\mbox{and}\quad  T^{-1}(p_{n, l}) \cap T^{-1}_n (p_{n, l})  \ne \emptyset, \qquad\qquad     \forall l\in\left\{1,2,\ldots,q(n)\right\}.  \label{eq:condition}
	\end{equation} 	
\end{lemma}
\begin{proof}
	It suffices to show that there exists a map $T_n: \mathcal{P}_{n} \mapsto \mathcal{P}_{n}$  with the property: $\mu( T(p_{n, l}) \cap T_n (p_{n, l}) ) > 0$ for all $l\in\left\{1,2,\ldots,q(n)\right\}$, since $\mu(A \cap B) > 0$ implies $A \cap B \ne\emptyset$ for any $A,B \in \mathcal{M}$, and:
	$$ \mu( T(p_{n, l}) \cap T_n (p_{n, l}) ) > 0 , \quad\forall l\in\left\{1,2,\ldots,q(n)\right\}    \quad\Rightarrow\quad \mu( T^{-1}(p_{n, l}) \cap T^{-1}_n (p_{n, l}) ) > 0, \quad     \forall l\in\left\{1,2,\ldots,q(n)\right\}.  $$
	Otherwise, let $T^{-1}_n (p_{n, k}) = p_{n, s}$ and $ \mu( T^{-1}(p_{n, k}) \cap T^{-1}_n (p_{n, k}) ) = 0$ for some $k\in\left\{1,2,\ldots,q(n)\right\}$, then $\mu(T(p_{n, s}) \cap T_n(p_{n, s})) = 0$ is a contradiction.
	
	Let $G_n =(\mathcal{P}_{n}, \mathcal{P}^{'}_{n} , E )$ denote a bipartite graph where $\mathcal{P}^{'}_{n}$ is a copy of $\mathcal{P}_{n}$ and $(p_{n,k}, {p}_{n,l}) \in E$ if $\mu( T(p_{n, k}) \cap p_{n, l}) > 0$. In order to generate a bijective map so that $\mu( T(p_{n, l}) \cap T_n (p_{n, l}) ) > 0$ for all $l\in\left\{1,2,\ldots,q(n)\right\}$, we need to uniquely assign every element in $\mathcal{P}_{n}$ with one in $\mathcal{P}_{n}$ using one of the existing edges. We call this new graph $\tilde{G}_n =(\mathcal{P}_{n}, \mathcal{P}^{'}_{n} , \tilde{E})$ (see \cref{fig:matching}). To verify that such an assignment is possible, we need to confirm that $G_n$ admits a perfect matching. 
	
	\begin{figure}[h!]
		\begin{center}
			\includegraphics[width=.6\textwidth]{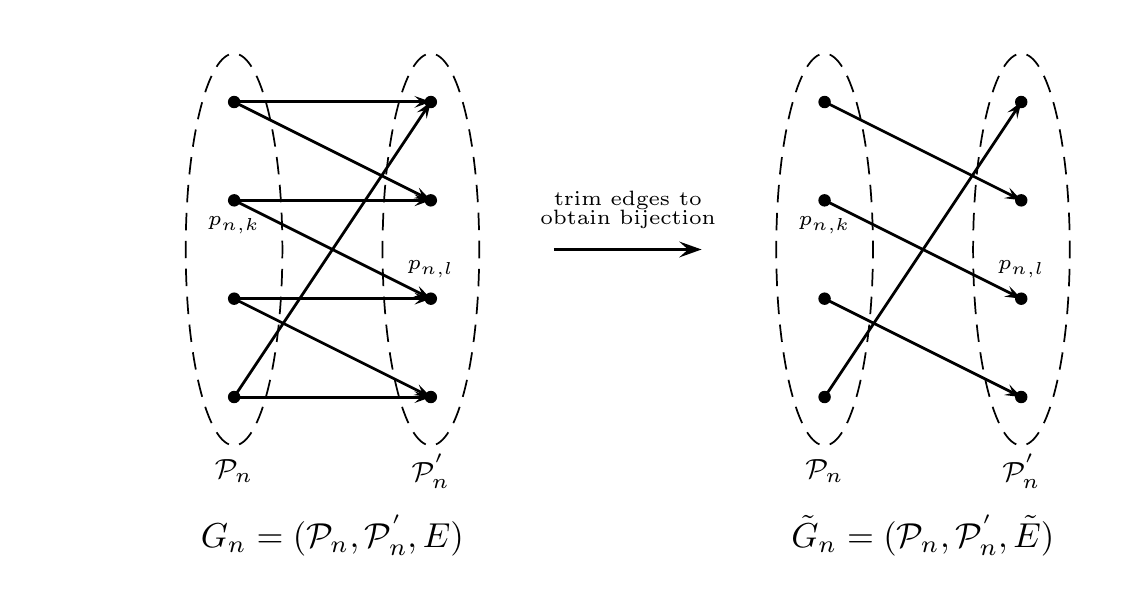} \\
		\end{center}
		\caption{Because of the measure-preserving property, the graph $G_n$ satisfies Hall's marriage conditions, and hence admits a perfect matching.} \label{fig:matching}
	\end{figure}
	To show that this is indeed the case,  let $N_G(B) \subset \mathcal{P}^{'}_{n}$ be the set of all vertices in $\mathcal{P}^{'}_{n}$ adjacent for some $B\subset \mathcal{P}_{n}$. By Hall's marriage theorem (see e.g. \cite{Cameron1994}), $G_n$ has a perfect matching if and only if the cardinality $|B | \leq | N_G(B)|$ for any $B\subset \mathcal{P}_{n}$. Since $T$ is a measure-preserving automorphism and because of condition \eqref{eq:equalsized}, it follows that for any $k\in\left\{1,2,\ldots,q(n)\right\}$, we must have: 
	$$     \sum^{q(n)}_{ l = 1}\mu( T(p_{n, k}) \cap p_{n, l} ) =  \frac{\mu(X)}{q(n)}, \qquad  0 \leq \mu( T(p_{n, k}) \cap p_{n, l} ) \leq \frac{\mu(X)}{q(n)}.$$
	Because of these properties, we deduce that:
	\begin{eqnarray*}
		| N_G(B)|	 & := & \sum_{k\in B} \left( \sum_{ l: \mu( T(p_{n, k}) \cap p_{n, l} > 0)} 1 \right) \\ 
		& \geq &  \sum_{k\in B} \frac{q(n)}{\mu(X)} \sum_{ l=1}^{q(n)}  \mu( T(p_{n, k}) \cap p_{n, l} ) \\
		& = &  \sum_{k\in B} 1 = |B|
	\end{eqnarray*}
	for any arbitrary $B \subset \mathcal{P}_{n}$. Hence, $G_n$ has a perfect matching.
\end{proof}

We remark that \eqref{eq:equalsized} plays an essential role in the proof of \cref{lemm:beautifull}, since it is easy to construct counter-examples of non-uniform partitions which fail to satisfy \eqref{eq:condition}: e.g. choose a compact set $A\in\mathcal{M}$ for which $T(A) \cap A = \emptyset$ and pick a partition wherein $A$ is divided into $r$ parts and $T(A)$ into $s$ parts, with $r\ne s$. The value of \cref{lemm:beautifull} is that it can be used as a means to bound the distance between the forward and inverse images of the partition elements in the Hausdorff metric.

\begin{lemma} \label{lemma:second}
	Let $T:X\mapsto X$ satisfy the hypothesis stated in \cref{thm:important} and let $\left\{\mathcal{P}_n\right\}^{\infty}_{n=1}$ be a sequence of measurable partitions that satisfy both the conditions \eqref{eq:equalsized} and \eqref{eq:partitioncond}. If $\left\{T_{n}: \mathcal{P}_{n} \mapsto \mathcal{P}_{n} \right\}^{\infty}_{n=1}$ is a sequence of bijective maps which satisfy the property \eqref{eq:condition} for each $n\in\mathbb{N}$, then  
	\begin{equation} 
	\lim_{n\rightarrow \infty} \displaystyle \sum^{k}_{l = -k} \displaystyle \max_{p \in \mathcal{P}_n}  d_H(T^l(p), T^l_n(p))  = 0  \label{eq:convergence}
	\end{equation}
	for every $k\in\mathbb{N}$.
\end{lemma}
\begin{proof}
	We will prove this claim using induction. Set $k=1$, if $T_n: \mathcal{P}_n \mapsto  \mathcal{P}_n$ is a bijective map satisfying  the property \eqref{eq:condition} then we must have that:  
	\begin{equation} 
		d_H( T(p_{n, j})  , T_n(p_{n, j}) )  \leq \mathrm{diam}( T(p_{n, j})) + \mathrm{diam} (T_n(p_{n, j})) \label{eq:something}
	\end{equation}
	and 
	$$ d_H( T^{-1}(p_{n, j})  , T^{-1}_n(p_{n, j}) )  \leq \mathrm{diam}( T^{-1}(p_{n, j})) + \mathrm{diam} (T^{-1}_n(p_{n, j})).$$
	Let $\epsilon>0$ and note that $T$  has a continuous inverse, since the map is a continuous bijection on a compact metric space $X$. By compactness, there exist a $\delta>0$ such that: 
	$$\mathrm{diam} (p_{n, j}) < \delta  \qquad \Rightarrow  \qquad  \mathrm{diam} (T(p_{n, j})) < \epsilon/4, \quad  \mathrm{diam} (T^{-1}(p_{n, j})) < \epsilon/4$$
	Pick $n\in \mathbb{N}$ so that $l(n)<\min\left\{\delta, \epsilon/4 \right\}$ to obtain:
	\begin{eqnarray*}
		d_H( T^{-1}(p_{n, j})  , T^{-1}_n(p_{n, j}) ) +  d_H( T(p_{n, j})  , T_n(p_{n, j}) ) & < & \frac{\epsilon}{4}  + \min\left\{\delta, \frac{\epsilon}{4} \right\} + \frac{\epsilon}{4}  + \min\left\{\delta, \frac{\epsilon}{4} \right\} \leq \epsilon
	\end{eqnarray*}	
	Since $\epsilon$ and $p_{n, j}\in \mathcal{P}_n$ are both arbitrary, we have proved \eqref{eq:convergence} for the case where $k=1$.
	
	Now to prove the result for some fixed $k>1$, we note from the triangle inequality that:
	$$ d_H(T^{k}(p_{n,j}), T^{k}_n(p_{n,j}) ) \leq  d_H(T(T^{k-1}(p_{n,j})), T(T^{k-1}_n(p_{n,j})) ) + d_H(T(T^{k-1}_n(p_{n,j})), T_n(T^{k-1}_n(p_{n,j})) ) $$
	Using the inductive hypothesis, assume that \eqref{eq:convergence} is true for $k-1$ so that the distance $d_H (T^{k-1}(p_{n,j}) , T^{k-1}_n(p_{n,j}))$ can be made arbitrarily small by choosing $n\in\mathbb{N}$ sufficiently large. By uniform continuity of $T$, there exist a $N_a\in \mathbb{N}$ sufficiently large, so that:
	$$ d_H(T(T^{k-1}(p_{n,j})), T(T^{k-1}_n(p_{n,j})) )\leq \frac{\epsilon}{4}\quad \mbox{and}\quad d_H(T(T^{k-1}_n(p_{n,j})), T_n(T^{k-1}_n(p_{n,j})) ) \leq  \frac{\epsilon}{4} $$
	for all $n>N_a$. Analogously, there exists a $N_b\in\mathbb{N}$ so that: 
	$$ d_H(T^{-1}(T^{-k+1}(p_{n,j})), T^{-1}(T^{-k+1}_n(p_{n,j})) )\leq \frac{\epsilon}{4}\quad \mbox{and}\quad d_H(T^{-1}(T^{-k+1}_n(p_{n,j})), T^{-1}(T^{-k+1}_n(p_{n,j})) ) \leq  \frac{\epsilon}{4} $$
	for all $n>N_b$. Setting $N_c=\max\left\{N_a, N_b\right\}$, we get:
	$$ \displaystyle  d_H(T^{-k}(p_{n,j}), T^{-k}_n(p_{n,j})) +  d_H(T^k(p_{n,j}), T^k_n(p_{n,j})) \leq {\epsilon} $$
	for all $n>N_c$.  Overall, we have:
	$$ \displaystyle \sum^{k}_{l = k} \displaystyle   d_H(T^l(p_{n,j}), T^l_n(p_{n,j}))  \leq k {\epsilon}, \qquad 
	\forall n>N_c.  $$
\end{proof}

Note that continuity of $T$ plays a critical role in the proof of \cref{lemma:second}. Furthermore, observe that $T$ is assumed to have an absolutely continuous invariant measure with $\mathrm{supp}(\mu) = X$, which ensures that both conditions \eqref{eq:equalsized} and \eqref{eq:partitioncond} are satisfied, simultaneously. With \cref{lemm:beautifull,lemma:second} we now have the necessary results to complete the proof of \cref{thm:important}.
\begin{proof}[Proof of \cref{thm:important}]
	Let $\left\{T_{n}: \mathcal{P}_{n} \mapsto \mathcal{P}_{n} \right\}^{\infty}_{n=1}$ be a sequence of bijective maps that satisfies the property of \cref{lemm:beautifull} for each $n\in\mathbb{N}$. Set $\epsilon>0$ and note that $A_n$ is compact, since it is a finite union of compact sets. By the triangle inequality, we have that:
	$$ \sum^{k}_{l = -k}  d_H(T^l(A), T^l_n(A_n))  \leq   \sum^{k}_{l = -k} d_H(T^{l}(A), T^{l}(A_n)) + \sum^{k}_{l = -k} d_H(T^{l}(A_n), T^{l}_n(A_n))$$ 
	We will show that each term above can be made arbitrarily small by selecting $n$ sufficiently large. From uniform continuity of $T$ and that $d_H(A,A_n)$ is monotonically decreasing to $0$ as $n \rightarrow \infty$, it follows that there exists a $N_1\in \mathbb{N}$ so that the first sum is bounded by $\epsilon/2$ for all $n>N_1$. In order to find also a bound for the second sum, notice that:
	\begin{eqnarray*}
		\sum^{k}_{l = -k} d_H(T^{l}(A_n), T^{l}_n(A_n))  & = &  \sum^{k}_{l = -k} d_H\left( \bigcup_{p\in P_n, p \cap A \ne \emptyset } T^l (p) ,  \bigcup_{p\in P_n, p \cap A \ne \emptyset } T^l_n (p) \right)  \\
		& \leq & \sum^{k}_{l = -k} \displaystyle \max_{p\in P_n, p \cap A \ne \emptyset} d_H(T^l(p), T^l_n(p)) \\
		& \leq & \sum^{k}_{l = -k} \displaystyle \max_{p\in P_n} d_H(T^l(p), T^l_n(p)) 
	\end{eqnarray*}
	where we have made repetitive use of the Haussdorf property: $d_H(A\cup B, C \cup D) \leq \max\left\{ d_H(A, C), d_H(B,D)  \right\}$ in the first inequality. From \cref{lemma:second}, it follows that we can choose $n>N_2$ so that the second sum is also bounded by $\epsilon/2$. We finally set $N = \max\left\{N_1, N_2 \right\}$ to complete the proof.  
\end{proof}

\section{Operator convergence} \label{sec:opconvergence}

Next we will prove how the sequence of discrete Koopman operators $\left\{ \K_n \right\}^\infty_{n=1}$ converges to the operator defined in \eqref{eq:KOOPMAN}, whenever $\left\{ T_n \right\}^\infty_{n=1}$ is a sequence of periodic approximations converging to the $T$ in the sense of  \cref{thm:important}. The following technical lemma is required.

\begin{lemma} \label{lemm:measureconv}
	Let $\left\{A_n\right\}^{\infty}_{n=1}$ be a sequence of compact sets converging monotonically to the compact set $A$ in the Haussdorf metric (i.e. $\lim_{n\rightarrow \infty} d_{H}(A, A_n) = 0$) and suppose that $\mu( A_n ) = \mu(A)$ for every $n\in\mathbb{N}$. Then,
	\begin{equation} 
	\lim_{n\rightarrow \infty} \mu(A \Delta A_n) = 0    \label{eq:measure}
	\end{equation}
\end{lemma}
\begin{proof}
	By definition: $\mu(A \Delta A_n) = \mu(A \setminus A_n) + \mu(A_n \setminus A)$, and hence, to prove \eqref{eq:measure} we must show that both of these terms go to zero in the limit. However, since $\mu( A_n ) = \mu(A)$, and:
	$$ \mu(A_n \setminus A)  =  \mu( A_n ) -  \mu( A \cap A_n), \qquad \mu(A \setminus A_n) =  \mu( A)  - \mu( A \cap A_n)$$
	it follows that $\mu(A \setminus A_n)  = \mu(A_n \setminus A)$, and therefore it sufficient to show that either $\mu(A \setminus A_n) $ or $\mu(A_n \setminus A)$ tends to zero.
	
	Consider $\mu(A_n \setminus A)$.  Let $A_{\epsilon(n)} := \bigcup_{x\in A} B_{\epsilon(n)}(x)$ be the $\epsilon(n)$-fattening of $A$ where $\epsilon(n) := d_{H}(A, A_n)$. Observe that:
	$$   (A_n \setminus A) \subset  (A_{\epsilon(n)}  \setminus A ) \qquad \Rightarrow \qquad  0\leq \mu(A_n \setminus A) \leq  \mu(A_{\epsilon(n)} \setminus A)  $$
	Since $\epsilon(n)$ is a monotonically decreasing sequence, from Theorem~1.19(e) in \cite{Rudin1987} we know that:
	$$ \bigcap_{n=1}^{\infty} A_{\epsilon(n)} = A\quad \mbox{and} \quad \mu(A_{\epsilon(n)}) \rightarrow  \mu(A) \mbox { as } n\rightarrow \infty.$$
	Therefore, $\mu(A_{\epsilon(n)} \setminus A) \rightarrow 0$ as $n\rightarrow \infty$, leading to the desired result.
\end{proof}

In addition to \cref{lemm:measureconv}, we will require knowing some properties on the averaging operation \eqref{eq:averageoperator}. The operator \eqref{eq:averageoperator} is an \emph{approximation of the identity} and therefore $\left\| g-g_n \right\| \rightarrow 0$ as $n\rightarrow \infty$. This property can be verified by first establishing this fact for continuous observables and then employ the fact that $C(X)$ is dense in $L^2(X,\mathcal{M},\mu)$. In addition to being an approximation to the identity,  \eqref{eq:averageoperator} is an orthogonal projector which maps observables $g\in L^2(X,\mathcal{M},\mu)$ to their best approximations $g_n\in L^2_n(X,\mathcal{M},\mu)$.
This follows from the fact that $\mathcal{W}_n$ is idempotent and that $\langle g- \mathcal{W}_n g, \mathcal{W}_n g  \rangle = 0$.
\begin{theorem}[Operator convergence] \label{thm:opconvergence}
	Let $T:X\mapsto X$ satisfy the hypothesis of \cref{thm:important} and suppose that   $\left\{T_{n}: \mathcal{P}_{n} \mapsto \mathcal{P}_{n} \right\}^{\infty}_{n=1}$ is a sequence of discrete maps that periodically approximates $T$ in the sense of \eqref{eq:setconvergence}. Then, given any $k\in \mathbb{N}$ and $g\in L^2(X,\mathcal{M},\mu)$, we have:
	\begin{equation} 	\lim_{n \rightarrow \infty}  \displaystyle \sum^{k}_{l = -k} \left\| \K^l g - \K^l_n g_n \right\|^2 = 0, \label{eq:opconvergence}
	\end{equation} 
	where $g_n\in L^2_n(X,\mathcal{M},\mu)$ refers to the operation \eqref{eq:averageoperator}.
\end{theorem}
\begin{proof}
	Consider observables of the kind: 
	\begin{equation}
	g = \sum_{j=1}^{q(m)} c_{j} \chi_{p_{m,j}} \in  L^2_m(X,\mathcal{M},\mu),\qquad m\in\mathbb{N}    \label{eq:simplefunctions}
	\end{equation}
	For notational clarity, write $A^{(j)} := p_{m,j}$, $g^{(j)} := \chi_{p_{m,j}}$ and $ g^{(j)}_n := \mathcal{W}_n g^{(j)}$. We have:
	\begin{eqnarray*}
		\sum^{k}_{l = -k} \left\| \K^l g - \K^l_n g_n \right\|^2 & = &  \sum^{k}_{l = -k} \left\| \sum_{j=1}^{q(m)} c_{j} ( \K^l g^{(j)} - \K^l_n g^{(j)}_n)   \right\|^2 \\
		& \leq &  \sum^{k}_{l = -k}  \left( \sum_{j=1}^{q(m)}  |c_{j}| \left\| \K^l g^{(j)} - \K^l_n g^{(j)}_n \right\| \right)^{2} \\
		& \leq & M  \sum_{j=1}^{q(m)} \left(  \sum^{k}_{l = -k}   \left\| \K^l g^{(j)} - \K^l_n g^{(j)}_n \right\|^2 \right),\qquad M = \sum_{j=1}^{q(m)}  |c_{j}|^2  \\
		& \leq & q(m) M   \max_{j=1,\ldots, q(m)} \left(  \sum^{k}_{l = -k}   \left\| \K^l g^{(j)} - \K^l_n g^{(j)}_n \right\|^2 \right)     
	\end{eqnarray*}
	Hence, it suffices to show that:
	\begin{equation} \lim_{n\rightarrow \infty}  \sum^{k}_{l = -k} \left\| \K^l g^{(j)} - \K^l_n g^{(j)}_n \right\|^2  = 0, \qquad j=1,\ldots,m. \label{eq:toconfirm}
	\end{equation}
	Since $\left\{\mathcal{P}_n\right\}^\infty_{n=1}$ are consecutive refinements, observe that $g^{(j)}_n = g^{(j)} $ for $n\geq m$, which implies: 
	\begin{eqnarray*}    \sum^{k}_{l = -k}  \left\| \K^l g^{j} - \K^l_n g^{(j)}_n \right\|^2  & = &   \sum^{k}_{l = -k}   \left\| (\K^l - \K^l_n) g^{(j)} \right\|^2, \qquad \mbox{if } n\geq m \\
		& = &   \sum^{k}_{l = -k} \mu( T^{-l}(A^{(j)}) \Delta T^{-l}_n(A^{(j)}) ).
	\end{eqnarray*}
	By  \cref{thm:important},  $\left\{ T^{-l}_n(A^{(j)}) )\right\}^{\infty}_{n=m}$ converges to $ T^{-l}(A^{(j)})$ in the Haussdorf metric. Also, since $\mu(T^{-l}(A^{(j)})) = \mu(T^{-l}_n(A^{(j)}) )$ for $n\geq m$, it follows from \cref{lemm:measureconv} that:
	$$   \sum^{k}_{l = -k} \mu( T^{-l}(A^{(j)}) \Delta T^{-l}_n(A^{(j)}) ) \rightarrow 0 \qquad\mbox{as } n\rightarrow \infty. $$
	This completes the proof for observables of the type \eqref{eq:simplefunctions}. For the general case, simply note that we can express $g_m = \mathcal{W}_m g$ and then perform the following manipulations:
	\begin{eqnarray*}
	 	\sum^{k}_{l = -k} \left\| \K^l g - \K^l_n g_n \right\|^2 & = &  \sum^{k}_{l = -k} \left\| \K^l (g-g_m + g_m) - \K^l_n (g-g_m + g_m)_n \right\|^2 \\
		& \leq &  \sum^{k}_{l = -k} 	 \left( \left\|g-g_m \right\| + \left\| \mathcal{W}_n (g-g_m) \right\| + \left\| \K^l g_m - \K^l_n (g_m)_n \right\| \right)^2 \\
			& \leq &   \sum^{k}_{l = -k} 	 \left( 2 \left\|g-g_m \right\| + \left\| \K^l g_m - \K^l_n (g_m)_n \right\| \right)^2 \\
		& \leq & \left(   \left(  \sum_{l=-k}^{k}  4 \left\|g-g_m \right\|^2 \right)^{\frac{1}{2}}  + \left(  \sum_{l=-k}^{k}  \left\| \K^l g_m - \K^l_n (g_m)_n \right\|^2 \right)^{\frac{1}{2}}   \right)^2\\
		& = &  \left(    2 \left\|g-g_m \right\|   + \left(  \sum_{l=-k}^{k}  \left\| \K^l g_m - \K^l_n (g_m)_n \right\|^2 \right)^{\frac{1}{2}}   \right)^2
	\end{eqnarray*}
	Given our previously established result, in the above inequality, we can make both terms arbitrarily small by choosing a sufficiently large $m,n\in\mathbb{N}$.
\end{proof}

\section{Spectral convergence} \label{sec:spectralconv}
In this section, we will analyze how  spectral projectors \eqref{eq:targetapprx} converge to \eqref{eq:target} in the limit. 
\subsection{An illuminating example}
Before we proceed to the general results, let us work out the details of a periodic approximation for a basic example in order to clarify certain subtelties on weak vs. strong convergence of the spectra.  In particular, consider the map:
$$ T(x) =  (x + \frac{1}{2})  \mod{1}, \quad x\in [0,1),$$
which is rotation on the circle by a half. 

For the partition  $\mathcal{P}_{n} = \left\{p_{n, 1}, p_{n, 2}, \ldots, p_{n, r^n} \right\}$ with  $p_{n, j} = \left(\frac{j-1}{r^n} ,\frac{j}{r^n} \right)$, $j=1,\ldots,r^n$ for some integer $r>1$, we may define the map\footnote{$\lfloor \cdot \rfloor$ denotes here the floor function.}:
$$	
T_{n}(p_{n, j}) =  p_{n, j^*}, \qquad j^* = \begin{cases} j + \lfloor \frac{r^n}{2} \rfloor  &  j + \lfloor \frac{r^n}{2} \rfloor \leq r^n \\ j + \lfloor \frac{r^n}{2} \rfloor - r^n   &  j + \lfloor \frac{r^n}{2} \rfloor > r^n    \end{cases},
$$
which clearly is a periodic approximation to the original transformation in the sense of \cref{thm:important}. The discrete Koopman operator associated with this map is isometric to the circulant matrix. That is, the permutation matrices:
$$  [\mathrm{U}_n]_{ij} = \langle \K_n \chi_{p_{n,i}},  \chi_{p_{n,j}}  \rangle, $$
are equal to the circulant matrices:
$$ \mathrm{U}_n = \begin{bmatrix} d_1 & d_{r^n} & \cdots & d_3 & d_2 \\ d_2 & d_1 & d_{r^n} &  & d_3 \\ \vdots & d_2 & d_1 & \ddots & \vdots \\ d_{r^n-1} & & \ddots & \ddots & d_{r^n} \\ d_{r^n} & d_{r^n-1} & \cdots & d_2 & d_1  \end{bmatrix}, \quad \mbox{where $d_{\lfloor \frac{r^n}{2}\rfloor+1} = 1$ and zero otherwise}. $$
The spectral decomposition of a circulant matrix can be obtained in closed-form using the Discrete Fourier Transform. In our specific case, we have\footnote{The scaling $1/\sqrt{r^n}$ here is required for normalization.}:
\begin{equation} 	v_{n, k}(x) = \frac{1}{\sqrt{r^n}} \sum_{j=1}^{r^n}  e^{2\pi i\frac{ (k-1)(j-1)}{r^n}}   \chi_{p_{n,j}}(x), \qquad  \theta_{n, k} = \begin{cases} \frac{(-1)^{k} +1}{2} \pi & r \mbox{ is even} \\ \frac{(-1)^{k} +1}{2} \pi - \frac{k-1}{r^n} \pi & r \mbox{ is odd} \end{cases}. \label{eq:discretevectors}
\end{equation}
Recall that the spectra of the true (i.e. infinite-dimensional) operator consists of only two eigenvalues located at $1$ and $-1$. Yet, from the equations above, we see that this property is maintained for the discrete analogue when $r$ is an even number. For an odd $r$, the eigenvalues of the discrete Koopman operators seem to densely fill up the unit circle as $n\rightarrow \infty$. 

At first sight, this fragility of the spectrum in the discretization appears as a serious problem. However, if we weaken our notion of what it means to converge spectrally, this issue can be avoided. A closer examination of the eigenvalue-eigenvector pairs $ \left\{\theta_{n, k}, v_{n, k} \right\}^{r^n}_{k=1}$ in \eqref{eq:discretevectors} provide clues on what approach should be taken.
After application of the smoothing operator \eqref{eq:averageoperator}, any square-integrable observable  $g\in L^2_n(\mathbb{T}, \mathcal{B}(\mathbb{T}), \mu)$ on the circle (with $\mu$ being the standard Lebesgue measure in this case), can be written in terms of the eigenvectors: 
$$g_n(x) =  \sum^{r^n}_{k=1} c_{n,k} v_{n, k} (x), $$
where $v_{n, k}\in L^2_n(\mathbb{T}, \mathcal{B}(\mathbb{T}), \mu)$ are nothing else but the discrete analogues of the Fourier harmonics (see \eqref{eq:discretevectors}). Let $d(\theta_{n,k},\{-\pi,0\}) $ denote the distance of $\theta_{n,k}$ to the points $-\pi$ and $0$ (i.e the locations of the true eigenvalues) on the circle and fix  $\delta>0$. Then for any $\epsilon>0$, there exists a $N\in\mathbb{N}$ such that:
$$  \sum_{d(\theta_{n,k},\{-\pi,0\})<\delta} |c_{n,k}|^2 <\epsilon, \qquad \forall n\geq N. $$
In other words, as $n$ approaches infinity, most of ``spectral energy'' will get concentrated around the eigenvalue points (see \cref{fig:illustration}). This observation is not particular to this simple example, but applies more generally.

\begin{figure}
\includegraphics[width=.3\textwidth,trim = {10cm 0cm 10cm 0cm}, clip]{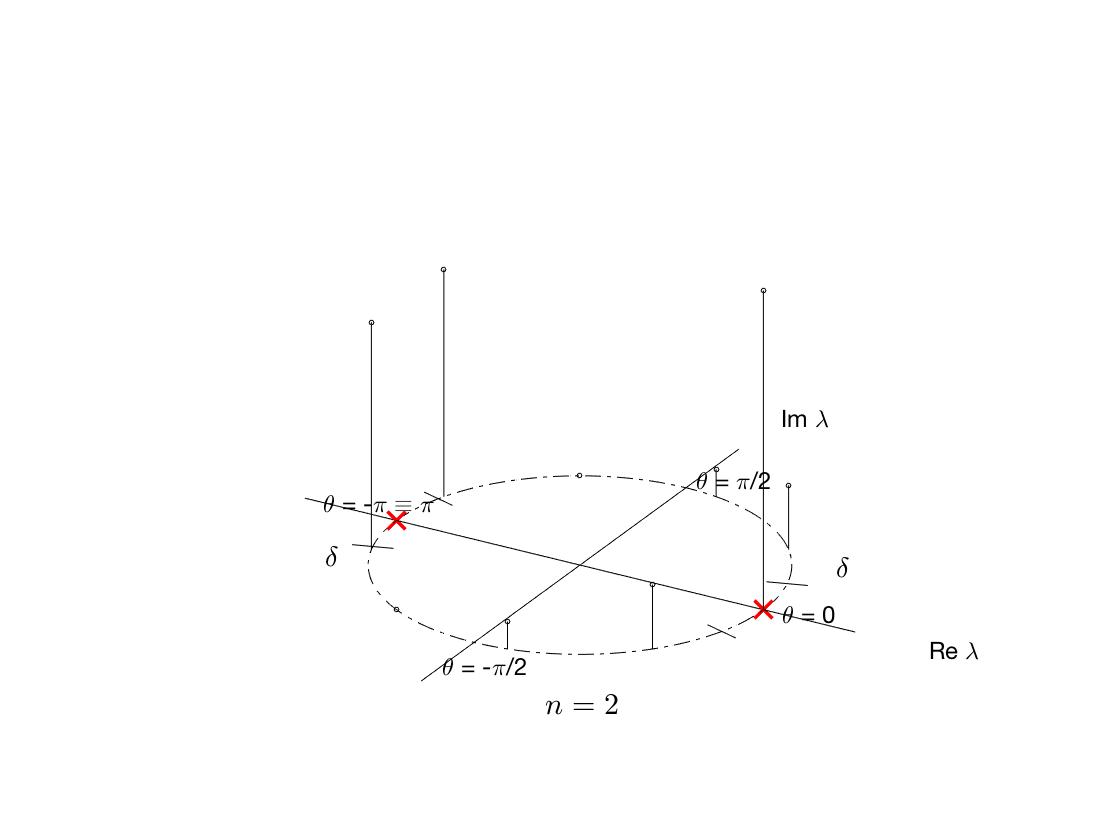} \includegraphics[width=.3\textwidth,trim = {10cm 0cm 10cm 0cm}, clip]{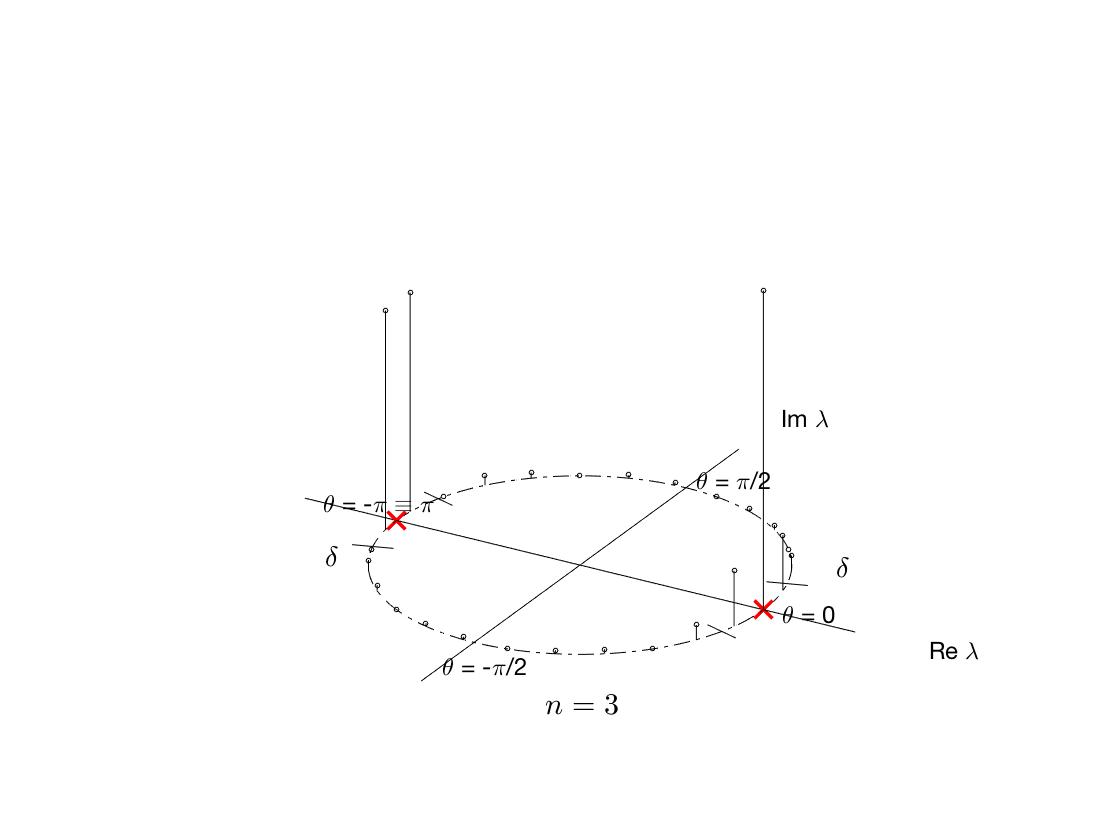} \includegraphics[width=.3\textwidth,trim = {10cm 0cm 10cm 0cm}, clip]{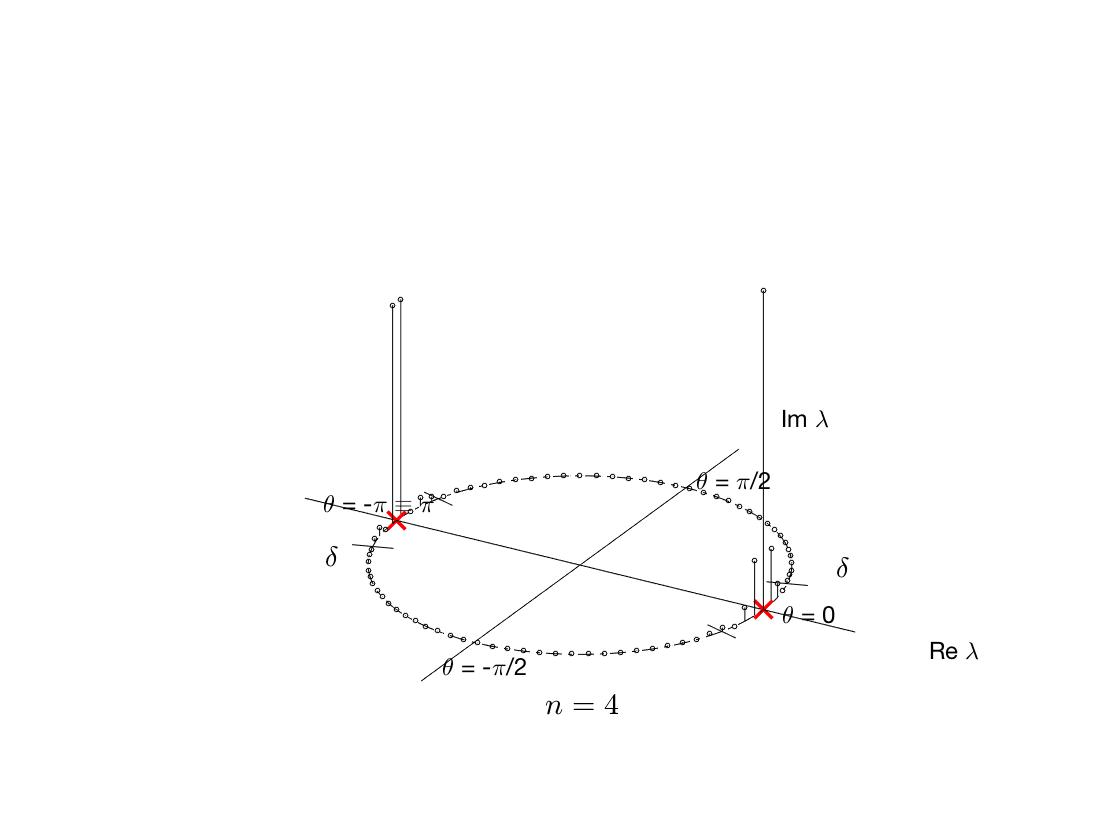}
\caption{Weak convergence of the spectra.} \label{fig:illustration}
\end{figure}

\subsection{Approximation of the spectral projectors} 
Consider any smooth test function $\varphi\in\mathcal{D}([-\pi,\pi))$ on the circle, and define:
$$ \Proj_{\varphi} g = \int_{-\pi}^{\pi} \varphi (\theta) \dd\Proj({\theta}) g,\qquad  \Proj_{n, \varphi} g_n =  \displaystyle \sum_{k=1}^{q(n)} \varphi (\theta_{n, k}) \Proj_{n} (\theta_{n, k}) g_n.$$
We will prove the following.
\begin{theorem} \label{thm:testfunction}
		Let $T:X\mapsto X$ satisfy the hypothesis of \cref{thm:important} and suppose that   $\left\{T_{n}: \mathcal{P}_{n} \mapsto \mathcal{P}_{n} \right\}^{\infty}_{n=1}$ is a sequence of discrete maps that periodically approximates $T$ in the sense of \eqref{eq:setconvergence}. For any smooth test function $\varphi\in\mathcal{D}([-pi,\pi))$ and observable $g \in L^2(X,\mathcal{M},\mu)$, we have:
		$$ \lim_{n\rightarrow \infty }\left\| \Proj_{\varphi} g - \Proj_{n,\varphi} g_n \right\|   = 0.$$
\end{theorem}
\begin{proof}
	Expand the smoothed indicator function $\varphi (\theta)$ by its Fourier series:
	$ \varphi (\theta) = \sum_{l=-\infty}^{\infty}  b_l e^{il\theta}, $
	and note that the series is uniformly convergent (a consequence of $\varphi (\theta)$ being smooth). We see that:
	$$ \Proj_{\varphi} g =\int_{-\pi}^{\pi} \left( \sum_{l=-\infty}^{\infty}  b_l e^{il\theta} \right) \dd\Proj ({\theta}) g \\
	= \sum_{l=-\infty}^{\infty} b_l \left( \int_{-\pi}^{\pi} e^{il\theta}  \dd\Proj({\theta}) g  \right) \\
	=  \sum_{l=-\infty}^{\infty} b_l  \K^l g, $$
	where we employed the spectral theorem of unitary operators \cite{Akhiezer1963} in the last equality. Similarly, it also holds that:
	$$  \Proj_{n, \varphi} g_n =   \sum_{l=-\infty}^{\infty} b_l  \K^l_n g_n $$
	Hence,
	$$ \Proj_{\varphi} g -  \Proj_{n, \varphi} g_n = \sum_{l=-\infty}^{\infty} b_l (\K^l g - \K^l_n g_n) $$
	Now let $\epsilon> 0$ and choose $k\in \mathbb{N}$ such that:
	$$ \sum_{|l|>k} |b_l| < \frac{\epsilon}{ 4 \left\| g \right\| } $$
	This is possible, because the Fourier coefficients of $\varphi (\theta)$ are absolutely summable (again, a consequence of $\varphi (\theta)$ being smooth). Write:
	\begin{eqnarray*}
		\left\| \Proj_{\varphi} g -  \Proj_{n, \varphi} g_n \right\| & \leq &  \sum_{l=-\infty}^{\infty} |b_l| \left\| \K^l g - \K^l_n g_n  \right\| \\ 
		& = & \sum_{l=-k}^{k}|b_l| \left\| \K^l g - \K^l_n g_n  \right\|  + \sum_{|l|>k}  |b_l| \left\| \K^l g - \K^l_n g_n  \right\| \\
		& \leq &  M \sum_{l=-k}^{k} \left\| \K^l g - \K^l_n g_n  \right\|  + 2  \left\| g \right\|  \sum_{|l|>k}  |b_l|  \\
		& \leq &  M \sum_{l=-k}^{k} \left\| \K^l g - \K^l_n g_n  \right\|  + \frac{\epsilon}{2}   \\
	\end{eqnarray*}
	For a fixed $k\in \mathbb{N}$, we can choose by \cref{thm:opconvergence} there exist an $N\in \mathbb{N}$ so that\footnote{Note that the squaring of the norms in \cref{thm:opconvergence} is immatarial for a finite sum.}
	$$\sum_{l=-k}^{k} \left\| \K^l g - \K^l_n g_n  \right\| \leq \frac{\epsilon}{2 M}, \qquad \forall n\geq N.$$
	This yields:
	$$ \left\| \Proj_{\varphi} g -  \Proj_{n, \varphi} g_n \right\| \leq \epsilon.$$
\end{proof}
\begin{remark}
	Note that in the proof of \cref{thm:testfunction}, we explicitly  made use of the fact that the discrete operators \eqref{eq:discreteoperator} are unitary, which in turn is a consequence of the periodic approximation. Therefore, our arguments would break down if the periodic approximation was replaced with a many-to-one map.
\end{remark}
\begin{remark}
	Observe that \cref{thm:testfunction} only concerns approximation of $\Proj_{\varphi}$ within the cyclic space generated by $g$ and not in its entirety. That is, we do \emph{not} have $\left\| \Proj_{\varphi}  - \Proj_{n,\varphi} \right\|\rightarrow 0$, but only $\left\| \Proj_{\varphi} g - \Proj_{n,\varphi} g_n \right\|\rightarrow 0$ for some specific $g$.
\end{remark}	
Recall the spectral projectors \eqref{eq:target} and \eqref{eq:targetapprx}, and notice that they may be re-expressed as:
$$ \Proj_D g = \int_{-\pi}^{\pi} \chi_D (\theta) \dd\Proj({\theta}) g,\qquad  \Proj_{n, D} g_n =  \displaystyle \sum_{k=1 }^{q(n)} \chi_D (\theta_{n, k}) \Proj_{n}( \theta_{n, k}) g_n,$$
where $\chi_D (\theta)$ is an indicator function on the circle for the interval $D$. Consider a smoothed version of the projectors using the summability kernels \cite{katznelson2004introduction}. That is, for some $0<\alpha<2\pi$, define $\varphi_{\alpha}: [-\pi,\pi) \times [-\pi,\pi) \mapsto \mathbb{R}_+$:
\begin{equation}\varphi_{\alpha}(x,y) = \begin{cases}   \frac{K}{\alpha} \exp\left({\frac{-1}{1- \left(\frac{d(x,y)}{\alpha} \right)^2 }}\right) &  \frac{d(x,y)}{\alpha} < 1    \\
0 & \mbox{otherwise}                          
\end{cases}, \label{eq:summkernel}
\end{equation}
where $d(x,y)$ is the Euclidian metric on $[-\pi,\pi)$ and $K=(\int^{1}_{-1} \exp(\tfrac{-1}{1-x^2})  \dd x)^{-1}$. Now replace the indicator function with: 
$$\chi_{D_{\alpha}} (\theta) = \int_{-\pi}^{\pi} \varphi_{\alpha}(\theta,\xi)  \chi_D(\xi) \dd \xi,$$ 
and define:
\begin{equation}
\Proj_{D_{\alpha}} = \int_{-\pi}^{\pi} \chi_{D_{\alpha}} (\theta) \dd\Proj({\theta}) g,\qquad ,\qquad  \Proj_{n, {D_{\alpha}}} g_n =  \displaystyle \sum_{k=1 }^{q(n)} \chi_{D_{\alpha}} (\theta_{n, k}) \Proj_{n} (\theta_{n, k}) g_n.
\end{equation}
We have the following corollary.
\begin{corollary}[Convergence of spectral projectors] \label{thm:projectorconvergence}
	Let $T:X\mapsto X$ satisfy the hypothesis of \cref{thm:important} and suppose that   $\left\{T_{n}: \mathcal{P}_{n} \mapsto \mathcal{P}_{n} \right\}^{\infty}_{n=1}$ is a sequence of discrete maps that periodically approximates $T$ in the sense of \eqref{eq:setconvergence}. Given any $0<\alpha<2\pi$ and interval $D \subset [-\pi,\pi)$, it follows that:
	$$  \lim_{n\rightarrow \infty }\left\| \Proj_{D_{\alpha}} g - \Proj_{n,D_{\alpha}} g_n \right\|   = 0, $$
	where $g \in L^2(X,\mathcal{M},\mu)$. 
\end{corollary}
\begin{remark}
	In general, the smoothing of the indicator functions cannot be avoided. If the operator happens to have discrete spectra on the boundaries of the inverval $D \subset [-pi,\pi)$, then the statement $\left\| \Proj_{D} g - \Proj_{n,D} g_n \right\|   \rightarrow 0$ as $n\rightarrow\infty$ is simply false. This is also the reason why smoothness conditions have to imposed onto $\varphi$ in \cref{thm:testfunction}
\end{remark}

\subsection{Approximation of the spectral density function}

Recall the definitions of the spectral density function $\rho(\theta;g)\in \mathcal{D}^{*}([-pi,\pi))$, along with its discrete analogue \eqref{eq:discrete}:
$$
	\rho_n(\theta;g_n) =  \sum_{k=1}^{q(n)} \left\| \Proj_{n} (\theta_{n, k}) g_n \right\|^2 \delta(\theta - \theta_{n,k})
$$
To assess the convergence of $\rho_n(\theta;g_n)$ to $\rho(\theta;g)$, we again make use of the summability kernels \eqref{eq:summkernel}.
The following result may be established. 
\begin{theorem}[Approximation of the spectral density function]
	Let:
	\begin{equation} \rho_{\alpha}(\theta; g) :=  \int_{-\pi}^{\pi} \varphi_{\alpha}(\theta,\xi)  \rho(\xi; g)  \dd \xi, \qquad \rho_{\alpha,n}(\theta; g_n) := \int_{-\pi}^{\pi} \varphi_{\alpha}(\theta, \xi)  \rho_{n}(\xi; g_n)  \dd \xi.  \label{eq:densityapprx}
	\end{equation}
	Then:
$$\lim_{n\rightarrow\infty} \rho_{\alpha,n}(\theta; g_n)  =  \rho_{\alpha}(\theta; g),\quad\mbox{uniformly.}$$
\end{theorem}
\begin{proof} To prove uniform convergence, we will establish that: (i) $\rho_{\alpha}(\theta; g) - \rho_{\alpha,n}(\theta;g_n)$ forms an equicontinuous family, and (ii) $\rho_{\alpha,n}(\theta;g_n)$ converges to $\rho_{\alpha}(\theta;g)$ in the $L^2$-norm.	 Uniform convergence of $\rho_{\alpha,n}(\theta;g_n)$ to $\rho_{\alpha}(\theta;g)$ is an immediate consequence of these facts. 
Indeed, if this was not the case, there would exist an $\epsilon>0$ and a subsequence $n_k$ such that $|\rho_{\alpha}(\theta_k;g) - \rho_{\alpha,n_k}(\theta_k;g_{n_k}) |\geq \epsilon$ for all $k\in \mathbb{N}$. But by equicontinuity, we may choose a $\delta>0$ such that: $$d\left( \rho_{\alpha}(\phi;g) - \rho_{\alpha,n}(\phi; g_{n}) , \rho_{\alpha}(\theta;g) - \rho_{\alpha,n}(\theta; g_{n}) \right) < \epsilon/2,\quad \mbox{whenever } d(\phi,\theta)<\delta.$$
This leads to a contradiction to (ii) as $|| \rho_{\alpha}(\cdot;g) - \rho_{\alpha,n_k}(\cdot; g_{n_k})  ||_2 \geq \epsilon/2 \sqrt{\delta}$. What follows next is a derivation of the claims (i) and (ii):
\begin{enumerate}[(i)]
	\item To show that $\rho_{\alpha}(\theta; g) - \rho_{\alpha,n}(\theta;g_n)$ is an equicontinuous family, we will confirm that its derivative $\rho^{'}_{\alpha}(\theta; g) - \rho^{'}_{\alpha,n}(\theta;g_n)$ is uniformly bounded. Consider the fourier expansion of $\rho_{\alpha}(\theta; g) - \rho_{\alpha,n}(\theta;g_n)$:
	$$ \rho_{\alpha}(\theta; g) - \rho_{\alpha,n}(\theta;g_n) = \frac{1}{2\pi}\sum_{l\in\mathbb{Z}} b_n(l;g) e^{il\theta} ,\qquad b_n(l;g) :=   \int_{-\pi}^{\pi} e^{-il\theta}  \left( \rho_{\alpha}(\theta;g)  -   \rho_{\alpha,n}(\theta;g_n) \right)  \dd \theta.  $$
	According to the spectral theorem of unitary operators \cite{Akhiezer1963}, we have by construction that:
	$$ a(l;g) := \int_{-\pi}^{\pi} e^{-il\theta} \rho(\theta;g) \dd \theta =  \langle  g , \K^{l} g \rangle,\qquad a_n(l;g_n) := \int_{-\pi}^{\pi} e^{-il\theta} \rho_n(\theta;g_n) \dd \theta =  \langle  g_n , \K^{l}_n g_n \rangle, \qquad l\in \mathbb{Z}.$$
	The functions $\rho_{\alpha}(\theta; g)$ and $\rho_{\alpha,n}(\theta;g_n)$ are defined as convolutions with a $C^{\infty}$-smooth function. Recognizing that convolutions implies pointwise multiplication in Fourier domain, we obtain:
	$$   b_n(l;g) =  d_{\alpha}(l)   ( a(l;g) - a_n(l;g_n) ), $$
	where:
	$$ d_{\alpha}(l) :=  \int_{-\pi}^{\pi} e^{-il\theta}  \varphi_{\alpha}(\theta,0) \dd \theta \quad \mbox{and} \quad |d_{\alpha}(l)|\leq \frac{C_{\alpha}}{1+|l|^N} \mbox{ for every }N\in\mathbb{N}. $$ 
	Now examining the  derivative $\rho^{'}_{\alpha}(\theta; g) - \rho^{'}_{\alpha,n}(\theta;g_n)$ more closely, we see that:
\begin{eqnarray*}
	\left| \rho^{'}_{\alpha}(\theta; g) - \rho^{'}_{\alpha,n}(\theta;g_n) \right| & = &  \left| \frac{1}{2\pi}\sum_{l\in\mathbb{Z}} il b_n(l;g) e^{il\theta} \right| \\
	& \leq &  \frac{1}{2\pi} \sum_{l\in\mathbb{Z}} |l| |b_n(l;g)|\\
	& = &\frac{1}{2\pi} \sum_{l\in\mathbb{Z}} |l| |d_{\alpha}(l)|\left|a(l;g) - a_n(l;g_n) \right| \\
	& = &  \frac{1}{2\pi} \sum_{l\in\mathbb{Z}} |l| |d_{\alpha}(l)|\left| \langle  g , \K^{l} g \rangle - \langle  g_n , \K^{l}_n g_n \rangle \right| \\
	&  \leq & \frac{ \left\| g \right\|^2}{\pi}  \sum_{l\in\mathbb{Z}}  \frac{C_{\alpha} |l| }{1+|l|^N},
	\end{eqnarray*}
   which is a convergent sum for $N\geq 3$. Notice that summability is possible because the constant $C_{\alpha}$ only depends on $\alpha$ and is independent of $n$.
	\item To show that $\rho_{\alpha,n}(\omega;g_n)$ converges to $\rho_{\alpha}(\omega;g)$ in the $L^2$-norm, we will use Parseval's identity to confirm that the sum: 	$ \sum_{l\in\mathbb{Z}} \left| b_n(l;g)\right|^2 $
	can be made arbitrarily small. At first, note that:
	$$ a(l;g) - a_n(l;g_n) =   \langle g ,   \K^{l} g -  \K^{l}_n g_n \rangle   -   \langle g - g_n   ,   \K^{l}_n g_n \rangle $$
	By the triangle inequality and Cauchy-Schwarz, we obtain:
	\begin{eqnarray*}  |a(l;g) - a_n(l;g_n)| & \leq &  \left\| g \right\|  \left\| \K^{l} g -  \K^{l}_n g_n \right\| +        \left\| g - g_n \right\|   \left\|  \K^{l}_n g_n  \right\|   \\
		& \leq &   \left\| g \right\| \left(  \left\| \K^{l} g -  \K^{l}_n g_n \right\|    +   \left\| g - g_n \right\|          \right).
	\end{eqnarray*}
	Let $\epsilon>0$, and choose $k\in \mathbb{N}$ such that:
\begin{equation} 
\sum_{|l|>k}   |d_{\alpha}(l)|^2 \leq \epsilon.   \label{eq:yeahbaby}
\end{equation}
	This is always possible, because $\varphi_{\alpha}(\theta,0)$ is a $C^{\infty}$-smooth function, and therefore also square-integrable. The following upper bound can be established:
	\begin{eqnarray*}
		\sum^{\infty}_{l=-\infty} \left| b_n(l;g)\right|^2 
		& = & \sum^{\infty}_{l=-\infty} |d_{\alpha}(l)|^2 |a(l;g) - a_n(l;g_n)|^2 \\
		& \leq &  \sum^{\infty}_{l=-\infty}   |d_{\alpha}(l)|^2  \left( \left\| g \right\| \left( \left\| \K^{l} g -  \K^{l}_n g_n \right\|    +   \left\| g - g_n \right\| \right)    \right)^2 \\
		& \leq &  \left\| g \right\|^2  \sum^{\infty}_{l=-\infty}   |d_{\alpha}(l)|^2  \left(  \left\| \K^{l} g -  \K^{l}_n g_n \right\|^2    + 2\left\| \K^{l} g -  \K^{l}_n g_n \right\|\left\| g - g_n \right\| +    \left\| g - g_n \right\|^2 \right)     \\
		& = & \left\| g \right\|^2  \sum^{k}_{l=-k}   |d_{\alpha}(l)|^2  \left(  \left\| \K^{l} g -  \K^{l}_n g_n \right\|^2    + 2\left\| \K^{l} g -  \K^{l}_n g_n \right\|\left\| g - g_n \right\| +    \left\| g - g_n \right\|^2 \right)   \\
		& &  + \left\| g \right\|^2 \sum_{|l|>k}   |d_{\alpha}(l)|^2  \left(  \left\| \K^{l} g -  \K^{l}_n g_n \right\|^2    + 2\left\| \K^{l} g -  \K^{l}_n g_n \right\|\left\| g - g_n \right\| +    \left\| g - g_n \right\|^2 \right)  \\
		& \leq & \left\| g \right\|^2  \max_{-k \leq l \leq k} |d_{\alpha}(l)|^2    \sum^{k}_{l=-k} \left\| \K^{l} g -  \K^{l}_n g_n \right\|^2 + 16\left\| g \right\|^2   \sum_{|l|>k}   |d_{\alpha}(l)|^2\\
		&   & + \left\| g \right\|^2  \left\| g - g_n \right\| \sum^{k}_{l=-k}   |d_{\alpha}(l)|^2  \left(  2\left\| \K^{l} g -  \K^{l}_n g_n \right\| +    \left\| g - g_n \right\| \right).
	\end{eqnarray*}
	Now apply \eqref{eq:yeahbaby} and \cref{thm:opconvergence} to complete the proof.
\end{enumerate}

\end{proof}

\part{Applications: A convergent numerical method for volume-preserving maps on the $m$-torus}

The second part of the paper is devoted to how  spectral projections \eqref{eq:target} and density functions \eqref{eq:spectraldensity}  are computed in practice. We will limit ourselves to volume-preserving maps on the $m$-torus. That is, $X=\mathbb{T}^m$ and is equipped with the geodesic metric:
\begin{equation}
d(x,y):= \max_{k\in \{1,\ldots,m\}} \min \{|x_i-y_i|, |1 - x_i -y_i| \}.   \label{eq:torusmetric}
\end{equation}
From here onwards, $T: \mathbb{T}^m \mapsto \mathbb{T}^m$ refers to an invertible, continuous Lebesgue measure-preserving transformation on the unit $m$-torus. In other words,  $\mu$ the Lebesgue measure is lebesgue measure and $\mu(B) = \mu(T(B)) = \mu(T^{-1}(B))$ for every $B\in\mathcal{B}(\mathbb{T}^m)$, where $\mathcal{B}(\mathbb{T}^m)$ denotes the Borel sigma algebra on $\mathbb{T}^m$. 

\section{Details of the numerical method} \label{sec:numericalmethod}

The numerical method can be broken-down into three steps. The first step is the actual construction of the periodic approximation $T_n:\mathcal{P}_{n} \mapsto \mathcal{P}_{n}$. The second step is obtaining a discrete representation of the observable $g_n$. The third step is the evaluation of the spectral projections and density functions. What follows next is a detailed exposition for each of these three steps.

\subsection{Step 1: Construction of the periodic approximation}

A periodic approximation can be obtained either through analytic means or by explicit construction. 
We will first discuss the general algorithm for constructing periodic approximations.

\subsubsection{Discretization of volume-preserving maps}
Following the principles in part I, call:
\begin{equation} 
p_{n,\jhatbold} := \left[\frac{\jhat_1-1}{\tilde{n}} , \frac{\jhat_1}{\tilde{n}}  \right] \times \left[\frac{\jhat_2-1}{\tilde{n}} , \frac{\jhat_2}{\tilde{n}}  \right]  \times \ldots \times \left[\frac{\jhat_m-1}{\tilde{n}} , \frac{\jhat_m}{\tilde{n}}  \right] \quad    \label{eq:pnjdef}
\end{equation}
where:
\begin{equation}
\tilde{n} = C a^n, \qquad C, a\in\mathbb{N}, \label{eq:discrlevel}
\end{equation}
and $\jhatbold=(\jhat_1, \jhat_2,\ldots, \jhat_m)$ is a multi-index with $\jhat_i\in \{1, \ldots, \tilde{n} \}$. To simplify the notation in some cases, it is useful to alternate between the multi-index $\jhatbold$ and single-index $j \in \{1, 2,\ldots, \tilde{n}^m \}$  through  a lexicographic ordering whenever it is convenient. Consider a partitioning of the unit $m$-torus (with $\mu(\mathbb{T}^m) = 1$) into the $m$-cubes: 
$$\mathcal{P}_n = \left\{ p_{n,j} \right\}^{q(n)}_{j=1}, \qquad q(n)= \tilde{n}^m$$
and note that this set is isomorphic to $\mathbb{N}_{\tilde{n}}^m :=\{ \jhatbold \in \mathbb{N}^m: 0 < \jhat_i  \leq \tilde{n}, i=1,\ldots,m \}$. 

The sequence of partitions $\left\{\mathcal{P}_n\right\}^{\infty}_{n=1}$ of the $m$-torus form a collection of refinements which satisfy the properties:
$$ \mu(p_{n,j}) = \frac{1}{\tilde{n}^m} \quad \mbox{and} \quad \mathrm{diam}(p_{n,j})=  \frac{\sqrt{m}}{\tilde{n}}, \quad \qquad j\in\left\{1,2,\ldots,\tilde{n}^m\right\},$$
where $\tilde{n}$ is defined as in \eqref{eq:discrlevel}.
It is clear from these properties that $\left\{\mathcal{P}_n\right\}^{\infty}_{n=1}$ is a family of equal-measure partitions which meet the conditions: $\mu(p_{n,j})\rightarrow 0$ and $\mathrm{diam}(p_{n,j}) \rightarrow 0$  as $n\rightarrow \infty$.

\subsubsection{General algorithm for constructing periodic approximations} The construction algorithm which we propose here is a slight variant to the method proposed in \cite{kloeden1997constructing}. Although, here we will show that the construction is fast if the underlying map is Lipschitz continuous.

\paragraph{\emph{The bipartite matching problem}}
A bipartite graph $G=(X,Y, E)$ is a graph where every edge $e\in E$ has one vertex in $X$ and one in $Y$. A matching $M\subset E$ is a subset of edges where no two edges share a common vertex (in $X$ or $Y$). The goal of the bipartite matching problem is to find a maximum {cardinality} matching, i.e. one with largest number of edges possible. A matching is called {perfect} if all vertices are matched. A perfect matching is possible if, and only if, the bipartite graph satisfies the so-called Hall's marriage condition:
$$ \mbox{for every }B\subset X \mbox{ implies }|B | \leq | N_G(B)|$$
where $N_G(B)\subset Y$ is the set of all vertices adjacent to some vertex in $A$  (see e.g. \cite{Cameron1994}).

Bipartite matching problems belong to the class of combinatorial problems for which well-established polynomial time algorithms exist, e.g. the  Ford-Fulkerson algorithm  may be used to find a maximum cardinality matching in $\mathcal{O}(|V| |E|)$ operations, wheareas the Hopcroft-Karp algorithm does it in $\mathcal{O}(\sqrt{|V|} |E|)$  (see e.g. \cite{Ahuja1993,Rivest2009}). 

\paragraph{\emph{The algorithm}}
In \cref{sec:periodicapprox} we showed that the solution to the bipartite matching problem of the graph\footnote{Note that the partition elements $p_{n,s} \in \mathcal{P}_{n}$ are  interpreted here as vertices in a bipartite graph.}:
\begin{equation} 
G_n =(\mathcal{P}_{n}, \mathcal{P}_{n} , E_n ), \qquad (p_{n,s}, p_{n,l}) \in E_n\quad\mbox{if}\quad\mu( T(p_{n, s}) \cap p_{n, l}) > 0 \label{eq:doublystochastic} 
\end{equation}
has a perfect matching. This fact played a key role in proving \cref{thm:important}. In practice, it is not advisable to construct this graph explicity since it requires the computation of set images. 

Fortunately, a periodic approximation may be obtained from solving a related bipartite matching problem for which set computation are not necessary. Consider the partition $\mathcal{P}_n$ and associate to each partition element $p_{n,\jhatbold}\in \mathcal{P}_n$ a representative point:
\begin{equation}x_{n,\jhatbold}  = \psi_n( p_{n,\jhatbold}) := \left( \frac{\jhat_{1}-\frac{1}{2}}{\tilde{n}} , \ldots, \frac{\jhat_{m}-\frac{1}{2}}{\tilde{n}}  \right)\in p_{n,\jhatbold}. \label{eq:surrogate}
\end{equation}
Define:
\begin{equation}
\varphi_n ( x )  =  \left(\tilde{n} x_1 + \frac{1}{2}, \ldots , \tilde{n} x_m + \frac{1}{2}\right).
\end{equation}
to be the function that maps the representative points onto the  lattice $\mathbb{N}_{\tilde{n}}^m :=\{ \jhatbold \in \mathbb{N}^m: 0 <\jhat_i  \leq \tilde{n}, i=1,\ldots,m \}$,   i.e.  $\jhatbold = \varphi_n ( x_{n,\jhatbold} )$. If $\lfloor\cdot\rfloor$  ($\lceil\cdot\rceil$) denote the floor (resp. ceil) function, we may define the following set-valued map:
\begin{equation} F^{(t)}(x) = \left\{ \jhatbold \in \mathbb{N}_{\tilde{n}}^m :        \lfloor \varphi_n(x_i) \rfloor-t+1 \leq \jhat_i \leq \lceil  \varphi_n(x_i) \rceil+t-1, \quad i=1,\ldots,m \right\} 
\end{equation}
to introduce the family of neighborhood graphs:
\begin{equation}
\hat{G}^{(t)}_n :=( \mathcal{P}_{n}, \mathcal{P}_{n} , \hat{E}^{(t)}_n), \qquad (p_{n,\shatbold}, p_{n,\lhatbold}) \in \hat{E}^{(t)}_n \quad\mbox{if}\quad \lhatbold\in F^{(t)}(T(x_{n,\shatbold}))     \label{eq:neigbor}
\end{equation}
The proposal is to find a maximum cardinality matching for the bipartite graph $\hat{G}^{(t)}_n$. Notice that $\hat{E}^{(t)}_n \subseteq \hat{E}^{(s)}_n$ whenever $s\geq t$. Hence, if no perfect matching is obtained for some value of $t$, repeated dilations of the graph eventually will. The algorithm is summarized below.
\begin{AlgDef} \label{alg1} To obtain a periodic approximation, do the following:
	\begin{enumerate}[1.]
		\item Initialize $t=1$.
		\item Find the maximum cardinality maching of $\hat{G}^{(t)}_n$. \label{step2}
		\item If the matching is perfect, assign matching to be $T_n$. Otherwise set $t\leftarrow t+1$ and repeat step \ref{step2}.
	\end{enumerate}
\end{AlgDef}

\begin{theorem}[Algorithmic correctness] Algorithm~\ref{alg1} terminates in a finite number of steps and yields a map $T_n: \mathcal{P}_{n} \mapsto \mathcal{P}_{n}$ with the desired asymptotic property described in \eqref{eq:setconvergence}.
\end{theorem}
\begin{proof}
	For large enough $t$, the graph $\hat{G}^{(t)}_n$ eventually becomes the complete graph. Hence, the algorithm must terminate in a finite number of steps. However, we must show that the algorithm terminates way before that, yielding a map with the asymptotic property \eqref{eq:setconvergence}. 
	
	In \cite{diamond1993numerical}, it was proven that for any Lebesgue measure-preserving map $T:\mathbb{T}^d \mapsto \mathbb{T}^d$, there exists a periodic map for which:
	\begin{equation}   \sup_{p_{n,\jhatbold} \in \mathcal{P}_n} \inf_{x\in \mathbb{T}^d}  \max\{ d(x ,\psi_{n}( p_{n,\jhatbold} ))    ,  d( T(x) , \psi_{n}(T_n( p_{n,\jhatbold} )) ) \}   \leq \frac{1}{2\tilde{n}},  \label{eq:diamond}
	\end{equation}
	where $d(\cdot,\cdot)$ refers to the metric defined in \eqref{eq:torusmetric}. By virtue of this fact, if one is able to construct a bipartite graph $\tilde{G}_n= ( \mathcal{P}_{n}, \mathcal{P}_{n} , \tilde{E}_n)$ for which:
	$$ (p_{n,\shatbold}, p_{n,\lhatbold}) \in \tilde{E}_n\quad \mbox{if}\quad \inf_{x\in p_{n,\shatbold} } d ( T(x), \psi_{n}( p_{n,\lhatbold} ) ) \leq \frac{1}{2\tilde{n}},$$
	then this bipartite graph, or any super-graph, should admit a perfect matching. Indeed, if we would assume the contrary, then this would imply that one is required to construct a periodic approximation $T_n$ for which  $d( T(x) , \psi_{n}(T_n( p_{n,\jhatbold} )))\geq \frac{1}{2\tilde{n}}$, an immediate contradiction of \eqref{eq:diamond}. 
	
	This fact has implications on the graph $\hat{G}^{(t)}_n$. As soon as $t$ hits a value for which $\hat{E}^{(t)}_n \supseteq \tilde{E}_n$, the graph $\hat{G}^{(t)}_n$ is gauranteed to have a perfect matching (although it may occur even before that). Notice that by compactness of the $m$-torus, $T$ admits a modulus of continuity $\omega: [0,\infty) \mapsto [0,\infty)$ with $ d ( T(x), T(y)) \leq \omega( d(x,y))$. An upper bound on the value of $t$ for which $\hat{G}^{(t)}_n$ becomes a supergraph of $\tilde{G}_n$ is given by:
	$$t^*=\lceil \tilde{n}\omega(1/\tilde{n}) + \frac{1}{2} \rceil\leq (\tilde{n}+1) \left( \omega(1/{\tilde{n}})  + \frac{1}{2\tilde{n}}\right) .$$
	That is, the algorithm should terminate for $t\leq t^*$.
	
	By assuming that the algorithm terminates at $t^*$, a worst-case error bound can be established for any map generated from the bipartite graphs.  We see that:
	$$ \sup_{p_{n,\jhatbold} \in \mathcal{P}_n} d(\psi_n (T_n (p_{n,\jhatbold})),   T(\psi_n(p_{n,\jhatbold}))  ) \leq \frac{t^*}{\tilde{n}} \leq \frac{\tilde{n}+1}{\tilde{n}} \left( \omega(1/{\tilde{n}})  + \frac{1}{2\tilde{n}}\right).$$
	Next, using the inequality:
	$$ d_H(T(p_{n,\jhatbold}),T_n(p_{n,\jhatbold})) \leq \mathrm{diam}(T_n(p_{n,j})) +  d(\psi_n (T_n (p_{n,\jhatbold})),   T(\psi_n(p_{n,\jhatbold}))  )    +  \mathrm{diam}(T(p_{n,j})),  $$
	it follows that:
	$$ \lim_{n\rightarrow \infty} d_H(T(p_{n,\jhatbold}),T_n(p_{n,\jhatbold})) \leq  \lim_{n\rightarrow \infty}   \frac{1}{\tilde{n}} +  \frac{\tilde{n}+1}{\tilde{n}} \left( \omega(1/{\tilde{n}})  + \frac{1}{2\tilde{n}}\right) + \omega(1/{\tilde{n}}) =0.$$
	From here onwards, one can proceed in the same fashion as the proofs of he proofs of \cref{lemma:second} and \cref{thm:important} to show that $\{T_n \}^{\infty}_{n=1}$ indeed satisfies the convergence property \eqref{eq:setconvergence}.
\end{proof}

\paragraph{\emph{Complexity of Algorithm~\ref{alg1}}}
Since the number of edges in $\hat{G}^{(t)}_n$ is proportional to the number of nodes for small $t$-values, the Hopcroft-Karp algorithm will solve a matching problem in $\mathcal{O}(\tilde{n}^{\frac{3m}{2}})$ operations. If additionally, the map $T$ is assumed to be Lipschitz, i.e. $\omega(d( x, y) ) = K d(x,y) $, the algorithm terminates before $t^* =  \lceil K+ \tfrac{1}{2} \rceil $, i.e. independently  of the discretization level $\tilde{n}$.  Hence, $\mathcal{O}(\tilde{n}^{\frac{3m}{2}})$ is also the overall time-complexity of algorithm. We note that explicit storage of the periodic map is $\mathcal{O}(\tilde{n}^{m})$ in complexity. Ideally, it is desirable to obtain a formulaic expression for $T_n$.

\subsubsection{Analytic constructions of periodic approximations} With the boxed partition of the $m$-torus \eqref{eq:pnjdef}, a sub-class of maps may be periodically aproximated analytically through simple algebraic manipulations (see e.g. \cite{earn1992exact}). In addition to some purely mathematical transformations, this also includes maps which arise in physical problems. In fact, any perturbed Hamiltonian twist map like Chirikov's Standard Map \cite{Chirikov1979}, or certain kinds of volume preserving maps such as Feingold's ABC map \cite{Feingold1988} belong to this category. 

A key observation is that these Lebesgue measure-preserving transformations consists of compositions of the following basic operations: 
\begin{enumerate}[1.]
	\item A signed permutation:
	\begin{equation}
	T(x) = \begin{bmatrix} 
	\xi(1) x_{\sigma(1)} \\
	\vdots \\
	\xi(m) x_{\sigma(m)}
	\end{bmatrix}, \quad \mbox{with }\mathrm{\sigma(\cdot)}\mbox{ a permutation and }\xi(\cdot)\in\{-1,1\}.
	\end{equation}
	which is periodically approximated by:
	$$ T_n (p_{n,\jhat}) = \psi^{-1}_n  \circ \varphi^{-1}_n \circ   \left(   \begin{bmatrix} 
	\xi(1) \jhat_{\sigma(1)} \\
	\vdots \\
	\xi(m) \jhat_{\sigma(m)}
	\end{bmatrix}  \mod \tilde{n}  \right) \circ \varphi_n  \circ \psi_n   (p_{n,\jhat}).  $$
	\item A translation:
	\begin{equation}
	T(x) = \begin{bmatrix} x_1 + \omega_1\\ \vdots \\ x_m + \omega_m\end{bmatrix} \mod 1. \label{eq:translation}
	\end{equation}
	which is periodically approximated by:
	$$ T_n (p_{n,\jhat}) = \psi^{-1}_n  \circ \varphi^{-1}_n \circ   \left(   \begin{bmatrix} \jhat_1 +  \lfloor \varphi_n(\omega_1) \rceil \\ \vdots \\ \jhat_m  +  \lfloor \varphi_n(\omega_m) \rceil \end{bmatrix}  \mod \tilde{n}  \right) \circ \varphi_n  \circ \psi_n   (p_{n,\jhat}). $$
	where $\lfloor\cdot \rceil$ denotes the nearest-integer function.
	\item A shear:
	\begin{equation}
	T(x) = \begin{bmatrix} x_1 \\  \vdots \\ x_i + f(x_1, \ldots , x_{i-1} , x_{i+1},\ldots,x_m) \\ \vdots \\ x_m  \end{bmatrix}
	\end{equation}
	which is periodicaly approximated by:
	$$ T_n (p_{n,\jhat}) = \psi^{-1}_n  \circ \varphi^{-1}_n \circ   \left(   \begin{bmatrix} \jhat_1 \\ \vdots \\ \jhat_i + \lfloor \varphi_n ( f(\varphi^{-1}_n(\jhat_1), \ldots , \varphi^{-1}_n(\jhat_{i-1}), \varphi^{-1}_n(\jhat_{i+1}),\ldots, \varphi^{-1}_n(\jhat_{m}) ) )  \rceil \\ \vdots \\ \jhat_m \end{bmatrix}  \mod \tilde{n}  \right) \circ \varphi_n  \circ \psi_n   (p_{n,\jhat}).  $$
\end{enumerate} 
Any map which is expressable as a finite composition of these operations can be periodically aproximated through approximation of each its components, i.e. the map
$ T(x) = T_s \circ \cdots \circ T_1 (x)$ is approximated by  $T_n(x) = T_{s,n} \circ \cdots \circ T_{1,n}  (x)$.
It is not hard to show that this would lead to a sequence of approximations that satisfy \eqref{eq:setconvergence}. Examples of maps which may be approximated in this manner are:
\begin{itemize}
	\item Arnold's Cat map \cite{Arnold1989}: 
	\begin{equation} 
	T(x) = \begin{bmatrix} 2 & 1 \\ 1 & 1 \end{bmatrix} \begin{bmatrix} x_1 \\ x_2 \end{bmatrix} \mod{1}, \label{eq:catmap}
	\end{equation}
	which is the composition of a permutation and shear map:
	$$ T(x) =  T_2 \circ T_1 \circ T_2 \circ T_1 (x), \quad \mbox{where:} \quad T_1(x) = \begin{bmatrix} 0 & 1 \\ 1 & 0 \end{bmatrix}\begin{bmatrix} x_1 \\ x_2 \end{bmatrix}, \quad T_2(x) = \begin{bmatrix} 1 & 0 \\ 1 & 1 \end{bmatrix}\begin{bmatrix} x_1 \\ x_2 \end{bmatrix}.$$
	\item Anzai's example \cite{anzai1951ergodic} of a skew product transformation:
	\begin{equation}
	T(x) = \begin{bmatrix} x_1+\gamma \\  x_1 + x_2 \end{bmatrix} \mod 1,      \label{eq:skewproduct}
	\end{equation}
	which is the composition of a shear map and a translation:
	$$ T(x) =  T_2 \circ T_1 (x), \quad \mbox{where:} \quad  T_1(x) = \begin{bmatrix} 1 & 0 \\ 1 & 1 \end{bmatrix}\begin{bmatrix} x_1 \\ x_2 \end{bmatrix}, \quad T_2(x) = \begin{bmatrix} x_1 + \gamma \\ x_2\end{bmatrix}.$$
	\item Chirikov Standard map \cite{Chirikov1979}:
	\begin{equation} 
	T(x) = \begin{bmatrix} x_1 + x_2 + K \sin(2\pi x_1)  \\  x_2 + K \sin(2\pi x_1) \end{bmatrix}\mod{1}, \label{eq:chirikovfamily}
	\end{equation}	
	which is the composition of two distinct shear maps in orthogonal directions:
	$$ T(x) =   T_2 \circ T_1 (x),\quad\mbox{where:} \quad T_1(x)  =  \begin{bmatrix} x_1   \\  x_2 + K \sin(2\pi x_1) \end{bmatrix}\mod{1},\quad T_2(x) = \begin{bmatrix} 1 & 1 \\ 1 & 0 \end{bmatrix}\begin{bmatrix} x_1 \\ x_2 \end{bmatrix}\mod{1}.$$
	\item The ABC-map of Feingold \cite{Feingold1988} which is the composition of 3 shear maps:
	$$ T(x) = T_3 \circ T_2 \circ T_1 (x),  $$
	where:
	$$ T_1(x) = \begin{bmatrix} x_1 + A \sin(2 \pi x_1) +  C \cos(2 \pi x_3)  \\ x_2 \\ x_3  \end{bmatrix}, \qquad T_2(x) = \begin{bmatrix} x_1  \\ x_2 + B \sin(2 \pi x_1) + A \cos(2\pi x_3) \\ x_3  \end{bmatrix},$$
	$$ T_3(x) = \begin{bmatrix} x_1  \\ x_2 \\ x_3 + C\sin(2 \pi x_2) + B\cos(2\pi x_1)   \end{bmatrix}.  $$
\end{itemize}

\subsection{Step 2: obtain a discrete representation of the observable}

The averaging operator \eqref{eq:averageoperator} is used to obtain a discrete representation to the observable. The operator \eqref{eq:averageoperator}  is an orthogonal projector which maps observables in $L^2(\mathbb{T}^m, \mathcal{B}(\mathbb{T}^m), \mu)$ onto their best approximations in $L^2_n(\mathbb{T}^m, \mathcal{B}(\mathbb{T}^m), \mu)$ . In practice, it makes no sense to explicitly evaluate this projection since we are mostly in the spectra of well-behaved functions. A discrete representation of the observable may also be obtain by simply sampling the function at the representative points \eqref{eq:surrogate}.  The averaging operator \eqref{eq:averageoperator} is replaced with:
$$ 	(\tilde{\mathcal{W}}_n g)(x) = \tilde{g}_{n}(x) := \sum_{j=1}^{q(n)} g( x_{n,j})\chi_{p_{n,j}}(x).   $$
Indeed, for observables that are continuous, it is not hard to show that $\left\|g_n - \tilde{g}_n  \right\| \rightarrow 0$ as $n\rightarrow \infty$.

\subsection{Step 3: Computing spectral projection and density function}

The discrete Koopman operator \eqref{eq:discreteoperator} is isomorphic to the permutation matrix $\mathrm{U}_n \in \mathbb{R}^{q(n) \times q(n)}$ by the similarity transformation:
\begin{equation}
\left[ \mathrm{U}_n \right]_{ij} := \left< \chi_{p_{n, i}}, \K_n \chi_{p_{n, j}} \right> = \begin{cases} 1  &\mbox{if }T(p_{n,j}) = p_{n,i}\\ 0 & \mbox{otherwise}\end{cases}.\label{eq:permutation}
\end{equation}
This property is exploited to evaluate the spectral projections and density functions very efficiently.
\subsubsection{The eigen-decomposition of a permutation matrix}
Permutation matrices admit an analytical expression for their eigen-decomposition. A critical part of this expression is the so-called cycle decomposition which decouples the permutation matrix into its cycles. More specifically, associated to every permutation $\mathrm{U}_n\in \mathbb{R}^{q(n) \times q(n)}$ one can find a different permutation matrix $\mathrm{P}_n\in \mathbb{R}^{q(n) \times q(n)}$ such that by a similarity transformation, we have:
\begin{equation} 
\mathrm{P}^T_n  \mathrm{U}_n {\mathrm{P}_n} =  \mathrm{C}_n  \label{eq:cyclicdecomp}
\end{equation}
where $\mathrm{C}_n$ is the circulant matrix:
$$ \mathrm{C}_n = \begin{bmatrix} \mathrm{C}^{(1)}_{n} & \\ & \ddots \\ & & \mathrm{C}^{(s)}_{n} \end{bmatrix},\qquad \mathrm{C}^{(k)}_{n} := \begin{bmatrix} 0 & 0  & \cdots      &   0   & 1 \\
1 &  0 &  \cdots     &   0   & 0   \\
0  & 1 &   \cdots    &   0  & 0 \\
\vdots  & \vdots  & \ddots     &   \vdots  & \vdots \\
0  &  0  &   \cdots   &  1 & 0    \end{bmatrix}\in \mathbb{R}^{n^{(k)} \times n^{(k)}},$$
$$ n^{(1)} + n^{(2)} + \ldots +  n^{(s)} = q(n).$$
The decomposition of \eqref{eq:permutation} into its cyclic subspaces is obtained as follows. Initialize $\mathrm{U}_{n,1} := \mathrm{U}_{n}$ and call:
$$ \mathrm{P}^{(1)}_n:= \begin{bmatrix}\mathrm{U}_{n,1}e_1 & \cdots & \mathrm{U}^{n^{(1)}}_{n,1}e_1 & e_{n^{(s)}+1} & \cdots & e_{q(n)}   \end{bmatrix}, \qquad\mbox{with }\mathrm{U}^{n^{(1)}}_{n,1}e_1 = e_1. $$
A similarity transformation yields:
$$ U_{n,2} := {\mathrm{P}^{(1)}_n}^T \mathrm{U}_{n,1} \mathrm{P}^{(1)}_n = \begin{bmatrix} \mathrm{C}^{(1)}_{n} \\ & \tilde{\mathrm{U}}_{n,1}  \end{bmatrix}. $$
Repeating this process for all $s$ cycles should result in $U_{n,s} = \mathrm{C}_n$ and $\mathrm{P}_n = \mathrm{P}^{(1)}_n \cdots \mathrm{P}^{(s)}_n$.

It is well-known that the Discrete Fourier Transform (DFT) matrix diagonalizes any circulant matrix. In particular, for $\mathrm{C}^{(k)}_{n}\in \mathbb{R}^{n^{(k)} \times n^{(k)}}$ the decomposition reads:            
\begin{equation}
\mathrm{C}^{(k)}_{n} =  \left(\mathrm{DFT}\right)^*_{n^{(k)}} \mathrm{\Lambda}_{n^{(k)}}   \left(\mathrm{DFT}\right)_{n^{(k)}}, \label{eq:DFTdecomp}
\end{equation}                              
where:
$$ \left(\mathrm{DFT}\right)_{n^{(k)}} = \frac{1}{\sqrt{n^{(k)}}}\begin{bmatrix} 1 & 1 & 1 & \cdots & 1\\
1 &\omega& \omega^2  & \cdots & \omega^{n^{(k)}-1}\\
1 & \omega^2 & \omega^4 & \cdots & \omega^{2(n^{(k)}-1)} \\
\vdots &    \vdots  & \vdots &   \ddots & \vdots \\
1 &  \omega^{1(n^{(k)}-1)} & \omega^{2(n^{(k)}-1)} & \cdots & \omega^{(n^{(k)}-1)(n^{(k)}-1)}
\end{bmatrix}, \mathrm{\Lambda}_{n^{(k)}} = \begin{bmatrix} \omega^0 & \\  & \ddots \\ & & \omega^{n^{(k)} -1}  \end{bmatrix},$$
and $\omega = \exp\left({\frac{2\pi}{n^{(k)}}}i\right)$.
By combining \eqref{eq:cyclicdecomp} with \eqref{eq:DFTdecomp}, the following explicit expression for the eigen-decomposition of \eqref{eq:permutation} is obtained:
\begin{equation}
\mathrm{U}_n = \mathrm{V}_n \mathrm{\Lambda}_n \mathrm{V}^*_n, \label{eq:eigperm}
\end{equation}
where:
$$ \mathrm{V}_n = {\mathrm{P}_n}  \begin{bmatrix}   \left(\mathrm{DFT}\right)^*_{n^{(1)}} & \\ & \ddots \\ & &   \left(\mathrm{DFT}\right)^*_{n^{(s)}} \end{bmatrix},\qquad \mathrm{\Lambda}_{n} =  \begin{bmatrix}  \mathrm{\Lambda}_{n^{(1)}} & \\ & \ddots \\ & &  \mathrm{\Lambda}_{n^{(s)}} \end{bmatrix}. $$

\subsubsection{Step 3a: Computing the spectral projection}
With the help of  \eqref{eq:eigperm} one may also obtain a closed-form expression for the vector representation of \eqref{eq:targetapprx}. Indeed, if 
$\mathrm{g}_n \in \mathbb{C}^{q(n)}$ and $\mathrm{S}_{n,D}\in \mathbb{C}^{q(n) \times q(n)}$ respectively denote:
$$ \left[ g_n \right]_{j} = \left< \chi_{p_{n, i}}, g_n \right>, \qquad \left[ \mathrm{S}_{n,D} \right]_{ij} := \left< \chi_{p_{n, i}}, \Proj_{n,D} \chi_{p_{n, j}} \right>,  $$
then it is straightforward to show that \eqref{eq:targetapprx} is reduced to the formula: 
\begin{equation} 
\mathrm{S}_{n,D} \mathrm{g}_n = \mathrm{V}_n  \mathbb{1}_{D}( \mathrm{\Lambda}_n ) \mathrm{V}^*_n  \mathrm{g}_n, \qquad \mbox{where: } \mathbb{1}_{D}(e^{i\theta})= \begin{cases} 1 & \mbox{if  }\theta \mod 2\pi \in D \\ 0 & \mbox{otherwise}\end{cases}. \label{eq:thecode}
\end{equation}
The evaluation of \eqref{eq:thecode} will involve finding the cycle decomposition \eqref{eq:cyclicdecomp}, which in turn requires traversing the cycles of the permutation matrix $\mathrm{U}_n \in \mathbb{R}^{q(n) \times q(n)}$. This requirement is equivalent to traversing the trajectories of the discrete map $T_n:\mathcal{P}_n\mapsto \mathcal{P}_n$. We arrive to the following algorithm.
\begin{AlgDef} \label{alg2} To compute the spectral projection, do the following:
	\begin{enumerate}[1.]
		\item Initialize a vector $f\in\mathbb{R}^{q(n)}$ of all zeros and set  $j=1$ and $k=1$.
		\item Given some $j\in\{1,\ldots, q(n) \}$, if $[f]_j = 1$ move on to step 4 directly. Otherwise, traverse the trajectories of the discrete map until it has completed a cycle, i.e.
		$$ p_{n,j(l)} = T^l_n(p_{n,j(1)}), \qquad  l=1,2,\ldots, n^{(k)}, \qquad \mbox{where: } j(1) = j \quad \mbox{and} \quad  p_{n,j(1)} = T^{n^{(k)}}_n(p_{n,j(1)}).$$
		To demarcate visited partition elements,  set $[f]_{j(l)} = 1$ for $ l=1,2,\ldots, n^{(k)}$.
		\item Denote by $\omega = \exp\left({\tfrac{2\pi}{n^{(k)}}}i\right)$. Apply the FFT and inverse-FFT algorithms to evaluate:
		$$ \left[\mathrm{S}_{n,D} \mathrm{g}_n\right]_{j(l)}  = \frac{1}{n} \sum_{\tfrac{2\pi (t-1)}{n^{(k)}}\in D }  \omega^{(l-1) (t-1)}   \left( \sum^{n^{(k)}}_{r=1}  \omega^{ (t-1)(r-1)}   g(\psi(p_{n,j(r)})) \right)$$
		for $l=1,2,\ldots,n^{(k)}$. Increment $k\leftarrow k+1$.
		\item Move on to the next partition element $p_{n,j} \in \mathcal{P}_{n}$, i.e. $j\leftarrow j+1$, and repeat step 2 until all partition elements have been visited.
	\end{enumerate}
\end{AlgDef}
\paragraph{\emph{Complexity of Algorithm~\ref{alg2}}}
The most dominant part of the computation is the application of the FFT and inverse-FFT algorithms on the cycles. Hence, the time-complexity of the algorithm is $ \mathcal{O}(\sum^s_{l=1} n^{(l)} \log n^{(l)})$. In the worst-case scenario, it may occur that the permutation \eqref{eq:permutation} consist of one single large cycle. In that case, we obtain the conservative estimate $\mathcal{O}(d \tilde{n}^d \log \tilde{n})$. We note that the evaluation of the spectral projections may be further accelerated using sparse FFT techniques. This is notably true for projections wherein the interval $D\subset [-\pi,\pi)$ is small, leading to a highly rank-deficient $\mathrm{S}_{n,D}\in \mathbb{C}^{q(n) \times q(n)}$.

\subsubsection{Step 3b: Computing the spectral density function} By modifying step 3 of the previous algorithm, one may proceed to compute an approximation of the spectral density function \eqref{eq:densityapprx}:
$$ \rho_{\alpha,n}(\theta; g_n) := \int_{-\pi}^{\pi} \varphi_{\alpha}(\xi,\theta)  \rho_{n}(\xi; g_n)  \dd \xi,      
$$
where:
$$
\varphi_{\alpha}(x,y) = \begin{cases}   \frac{K}{\alpha} \exp\left({\frac{-1}{1- \left(\frac{d(x,y)}{\alpha} \right)^2 }}\right) &  \frac{d(x,y)}{\alpha} < 1    \\
0 & \mbox{otherwise}                          
\end{cases}, \qquad\rho_n(\theta;g_n) =  \sum_{k=1}^{q(n)} \left\| \Proj_{n, \theta_{n, k}} g_n \right\|^2 \delta(\theta - \theta_{n,k}). 
$$
If \eqref{eq:densityapprx} needs to evaluated at some collection of points $\{ \theta_r \}^{R}_{r=1} \subset [-\pi, \pi)$, the algorithm would look as follows:
\begin{AlgDef}\label{alg3} To evaluate the spectral density function on $\{ \theta_r \}^{R}_{r=1} \subset  [-\pi, \pi)$, do the following:
	\begin{enumerate}[1.]
		\item Initialize a vector $f\in\mathbb{R}^{q(n)}$ of all zeros and also initialize $\rho_{\alpha,n}(\theta_r; g_n)=0$ for $r=1,\ldots,R$. Furthermore,  set  $j=1$ and $k=1$.
		\item Given some $j\in\{1,\ldots, q(n) \}$, if $[f]_j = 1$ move on to step 4 directly. Otherwise, traverse the trajectories of the discrete map until it has completed a cycle, i.e.
		$$ p_{n,j(l)} = T^l_n(p_{n,j(1)}), \qquad  l=1,2,\ldots, n^{(k)}, \qquad \mbox{where: } j(1) = j \quad \mbox{and} \quad  p_{n,j(1)} = T^{n^{(k)}}_n(p_{n,j(1)}).$$
		To demarcate visited partition elements,  set $[f]_{j(l)} = 1$ for $ l=1,2,\ldots, n^{(k)}$.
		\item Denote by $\omega = \exp\left({\tfrac{2\pi}{n^{(k)}}}i\right)$. Apply the FFT algorithms to update:
		$$ \rho_{\alpha,n}(\theta_r; g_n) \leftarrow \rho_{\alpha,n}(\theta_r; g_n) + \frac{K}{n\alpha} \sum_{ d\left(\tfrac{2\pi (t-1)}{n^{(k)}},\theta_r \right) < \alpha}  \exp\left({\frac{-1}{1- \left(\frac{d\left(\tfrac{2\pi (t-1)}{n^{(k)}},\theta_r \right)}{\alpha} \right)^2 }}\right)  \left| \sum^{n^{(k)}}_{l=1}  \omega^{ (t-1)(l-1)}   g(\psi(p_{n,j(l)})) \right|^2 $$
		for $r=1,2,\ldots,R$. Increment $k\leftarrow k+1$.
		\item Move on to the next partition element $p_{n,j} \in \mathcal{P}_{n}$, i.e. $j\leftarrow j+1$, and repeat step 2 until all partition elements have been visited.
	\end{enumerate}
\end{AlgDef}
\paragraph{\emph{Complexity of Algorithm~\ref{alg3}}}
Since $\alpha>0$ is typically small, the most dominant part of the computation wil again be the application of the FFT algorithm in step 3. Assuming that $R<<n^{(k)}$,  the time-complexity of the algorithm will be $ \mathcal{O}(\sum^s_{l=1} n^{(l)} \log n^{(l)})$, reducing to $\mathcal{O}(d \tilde{n}^d \log \tilde{n})$ in the worst-case scenario of a cyclic permutation.

\section{Canonical examples} \label{sec:examples}

In this section we analyze the behavior of our numerical method on three canonical examples of Lebesgue measure-preserving transformations on the $m$-torus: the translation map \eqref{eq:translation}, Arnold's cat map \eqref{eq:catmap}, and Anzai's skew-product transformation \eqref{eq:skewproduct}.  The considered maps permit an easy closed form solution for both the spectral projectors and density functions, hence making them a good testing bed to analyze the performance of the numerical method. Furthermore, the examples cover the full range of possible ``spectral scenario's'' that one may encounter in practice: the first map has a fully discrete spectrum, the second one has a fully continuous spectrum\footnote{Expect for the trivial eigenvalue at 1}, and the third map has a mixed spectrum.

\subsection{Translation on the $m$-torus}
Consider a translation on the torus \eqref{eq:translation}:
$$T(x) = \begin{bmatrix} x_1 + \omega_1\\ \vdots \\ x_m + \omega_m\end{bmatrix} \mod 1. 
$$
This is the type of map one would typically obtain from a Poincare section of an integrable Hamiltonian expressed in action-angle coordinates. Denote by ${k} = (k_1,\ldots,k_m) \in \mathbb{Z}^m$ and ${\omega} = (\omega_1,\ldots,\omega_m)\in\mathbb{R}^m$.  The map \eqref{eq:translation} is ergodic if and only if $ {k} \cdot {\omega}  \notin \mathbb{Z}$ for all ${k}\in \mathbb{Z}^m$ and ${k}\ne 0$.

\subsubsection{Singular and regular subspaces}
The Koopman operator associated with \eqref{eq:translation} has a fully discrete spectrum. Moreover, the Fourier basis elements turn out to be eigenfunctions for the operator, i.e. 
$$ \K e^{2\pi i ({k}\cdot x) } =  e^{2\pi i ({k}\cdot {\omega}) }  e^{2\pi i ({k}\cdot x) }. $$
The singular and regular subspaces are respectively:
$$ H_s =   \overline{\spn \{ e^{2\pi i ({k}\cdot x) }, \quad {k}\in \mathbb{Z}^m    \}},\qquad H_r  =  \emptyset.   $$
\subsubsection{Spectral density function}
The spectral density function for observables in the form:
$$ g = \sum_{{k} \in \mathbb{Z}^m } a_{{k}} e^{2\pi i ({k}\cdot x) }  \in L^m(\mathbb{T}^m, \mathcal{B}(\mathbb{T}^m), \mu), $$
is given by the expression:
\begin{equation} \rho(\theta;g) =  \sum_{{k}\in \mathbb{Z}^{m}} |a_{{k}}|^2 \delta(\theta - 2\pi ({k} \cdot {\omega}) \mod 2\pi).      \label{eq:specdensitytranslation}
\end{equation}
Noteworthy to mention is that the spectra of the operator for the map \eqref{eq:translation} can switch from isolated to dense with arbitrarily small perturbations to $\omega\in\mathbb{R}^m$. 
\subsubsection{Spectral projectors}
Given an interval $D=[a,b)\subset [-\pi,\pi)$, we have the following straightforward expression for the spectral projectors:
$$ \Proj_D g = \sum_{2\pi ({k} \cdot {\omega}) \mod 2\pi \in D} a_{{k}} e^{2\pi i ({k}\cdot x) }. $$
\subsubsection{Numerical results}
In \cref{fig:densityplotcircle}, we plot the spectral density function \eqref{eq:spectraldensity} for the case when $d=1$, $\omega_1 = 1/3$, and:
\begin{equation} 
g(x) =  \frac{\sin( 4 \pi x)}{ 1 + \cos^2( 2\pi x) + \sin(7\pi x)}. \label{eq:1Dobservable}
\end{equation}
As the partitions are refined, it is evident from \cref{fig:densityplotcircle} that the energy  becomes progressively more concentrated around the eigenvalues of the true spectra at $e^{i\frac{2\pi (k-1)}{3}}$, with $k=1,2,3$. 

\begin{figure}[h!]
	\begin{center}
		\includegraphics[width=.25\textwidth]{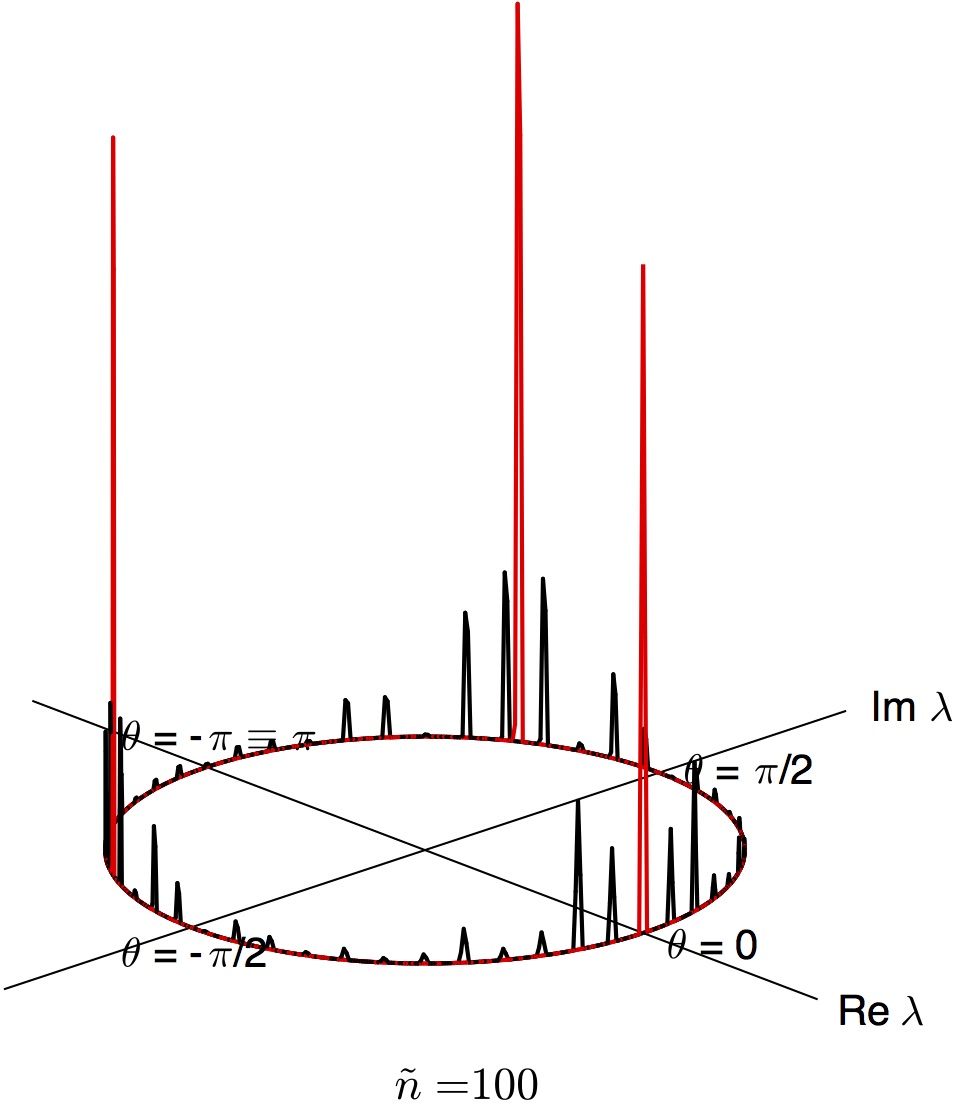}\includegraphics[width=.25\textwidth]{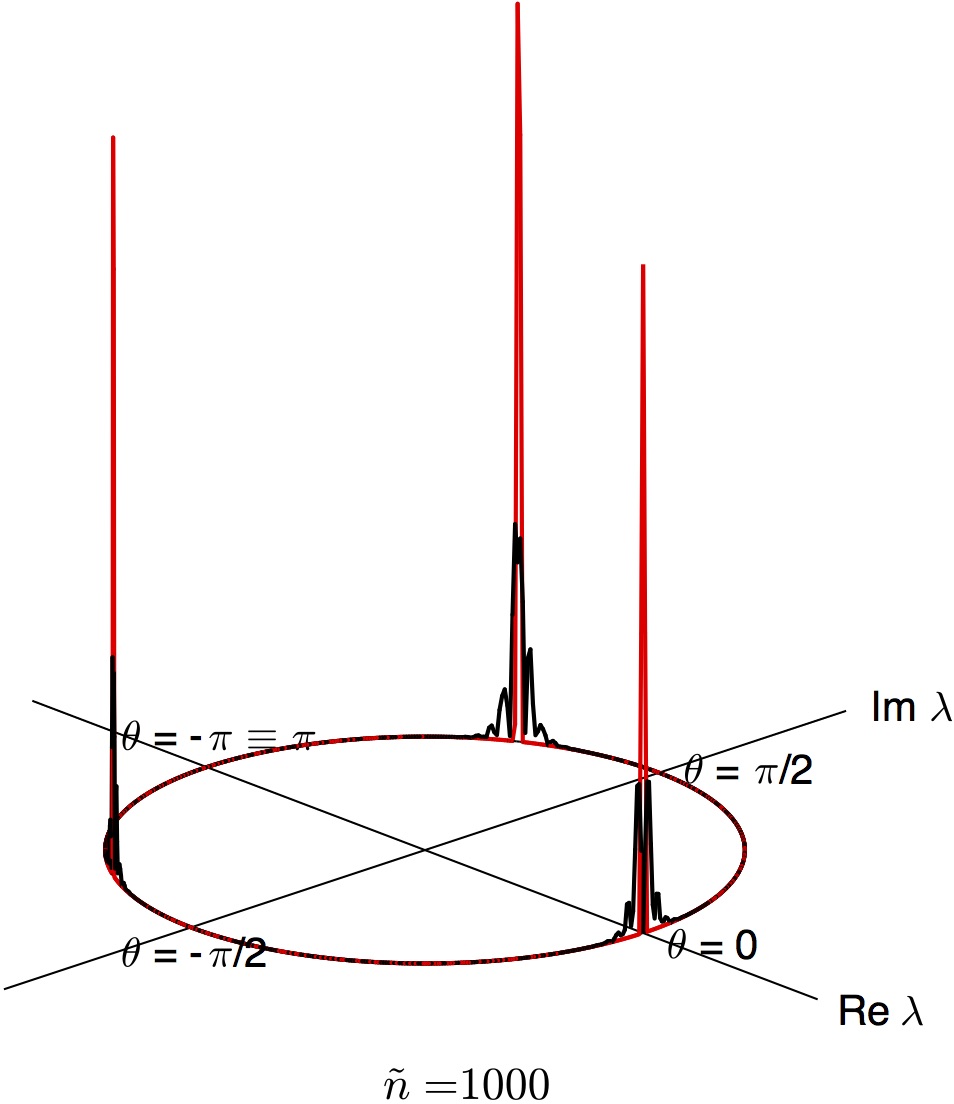} \includegraphics[width=.25\textwidth]{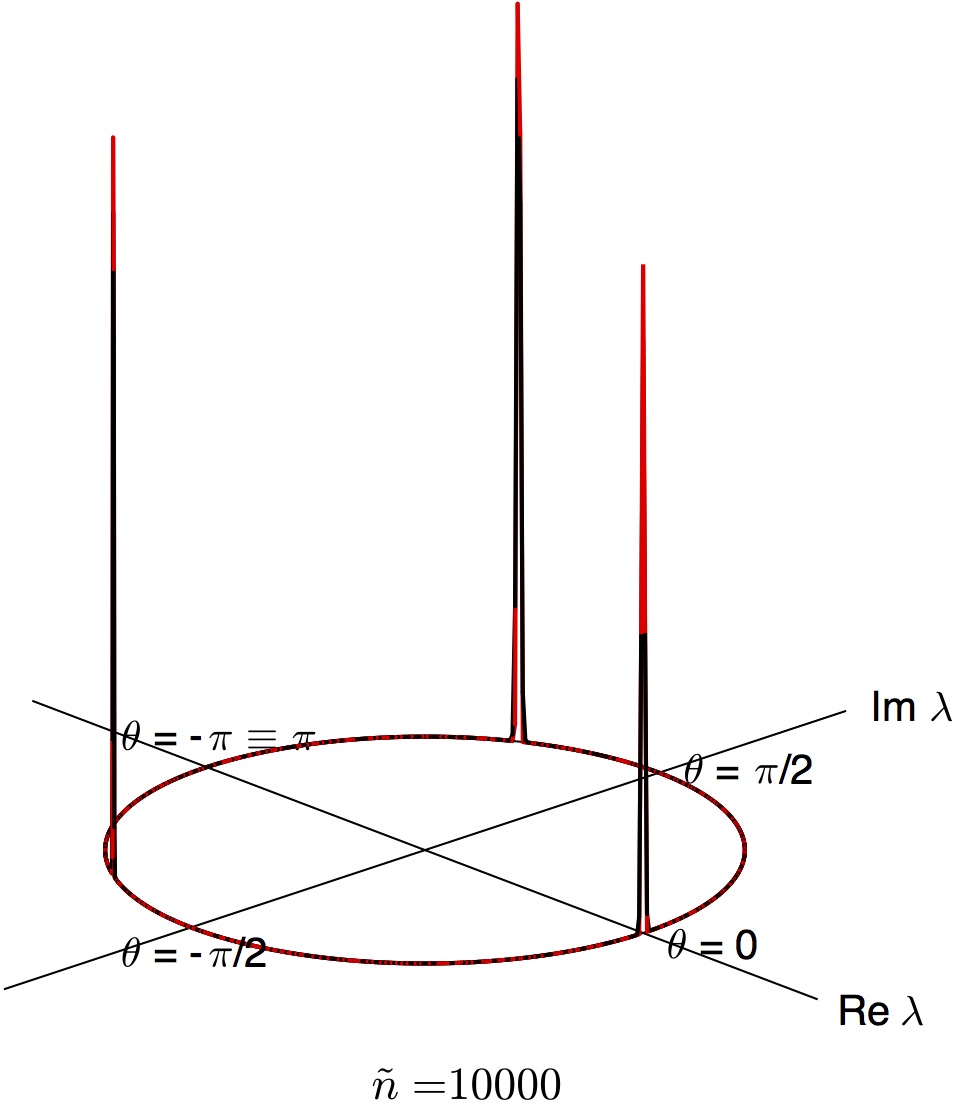}
	\end{center}
	\caption{Approximations of the spectral density function \eqref{eq:spectraldensity} for the translation map \eqref{eq:translation} with $m=1$ and $\omega_1 = 1/3$. Spectra is approximated for the observable \eqref{eq:1Dobservable} with $\alpha=  2\pi / 500$. The black curve is approximation and the red curve is the true density.} \label{fig:densityplotcircle}
	\begin{center}
		\subfigure[Spectral projections of the observable at narrow interval around the eigenvalues $\theta_k \tfrac{2\pi(k-1)}{6} -0.02$. ]{ \begin{tabular}{cc} \includegraphics[width=.46\textwidth]{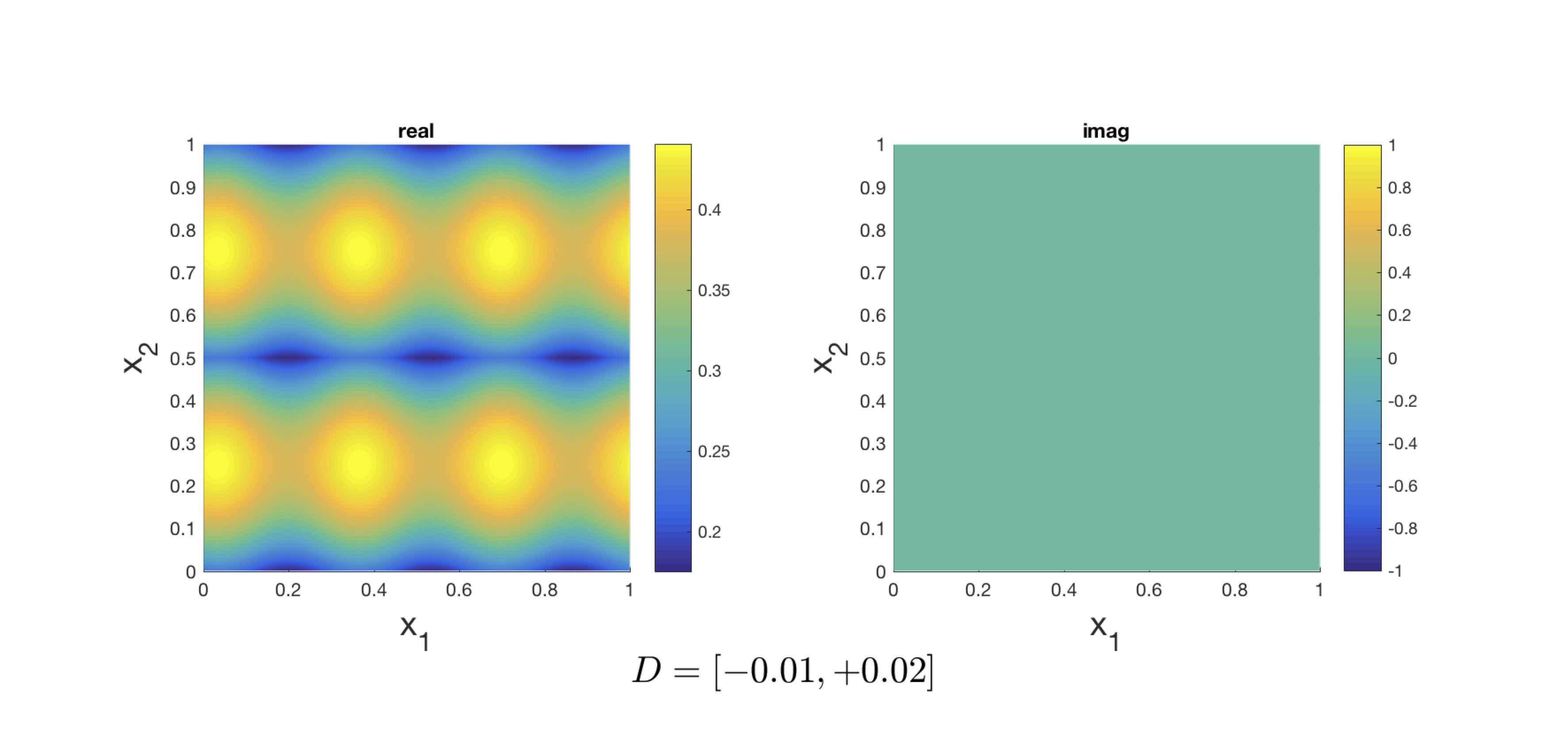} & \includegraphics[width=.46\textwidth]{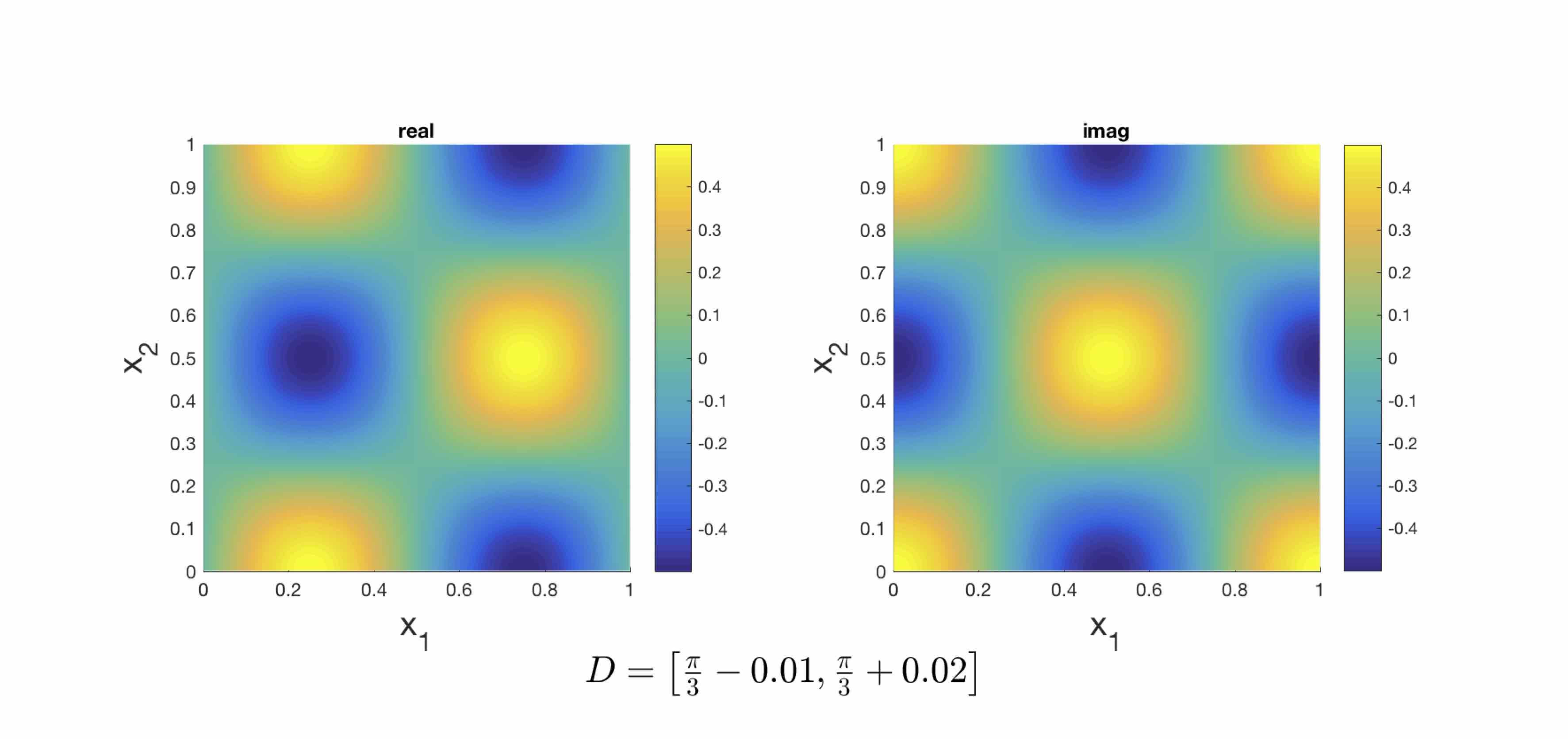}  \\
				\includegraphics[width=.46\textwidth]{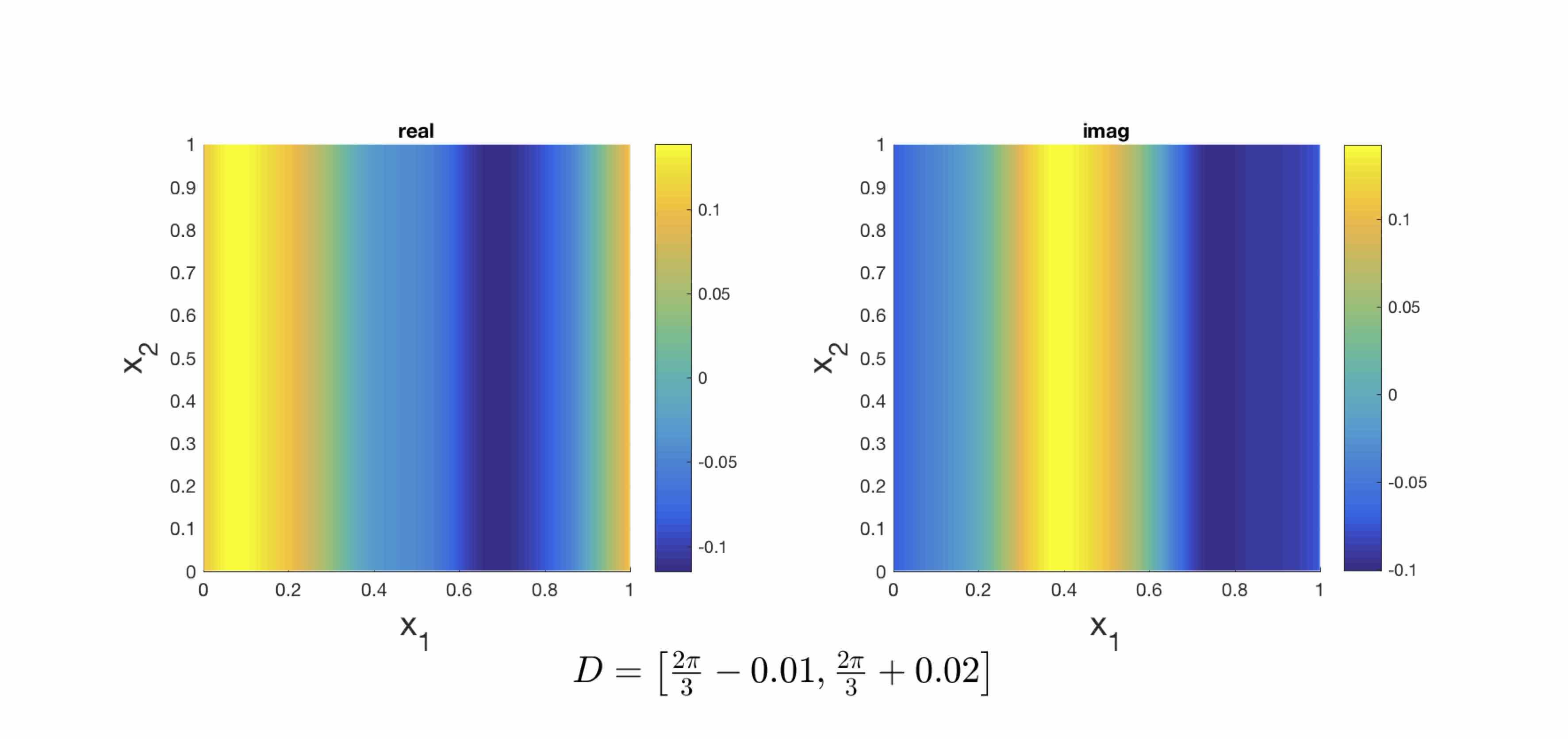} & \includegraphics[width=.46\textwidth]{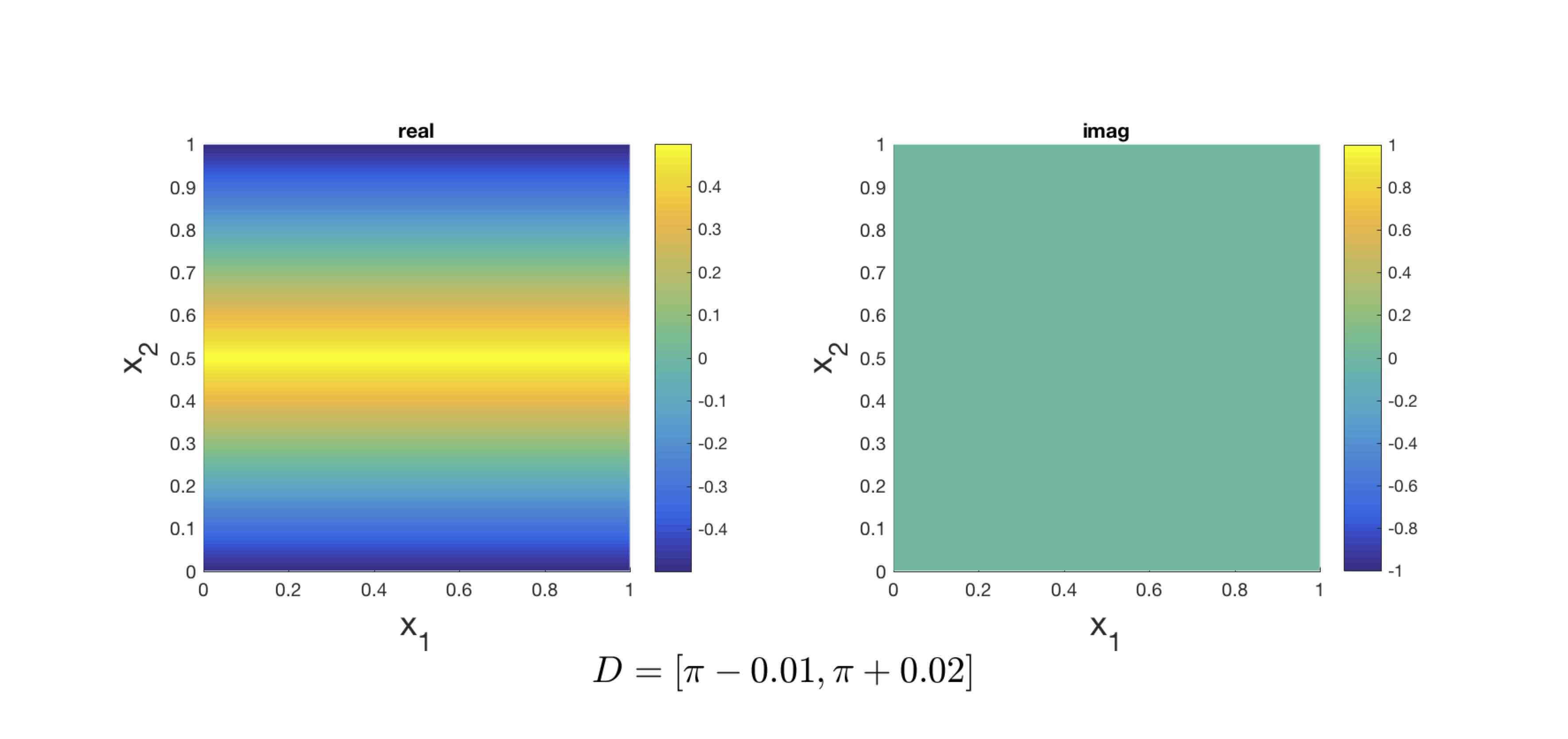}  \\
				\includegraphics[width=.46\textwidth]{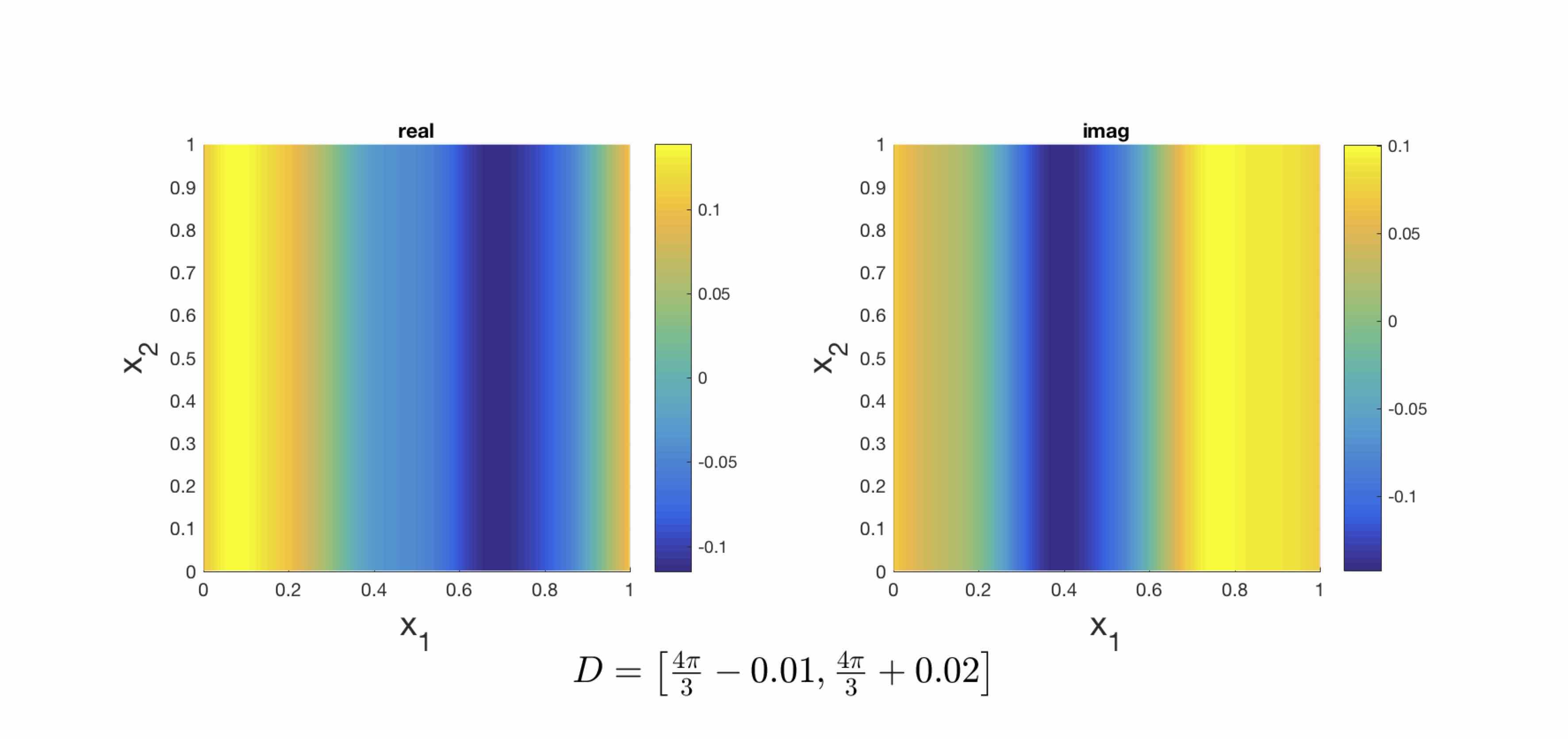} & \includegraphics[width=.46\textwidth]{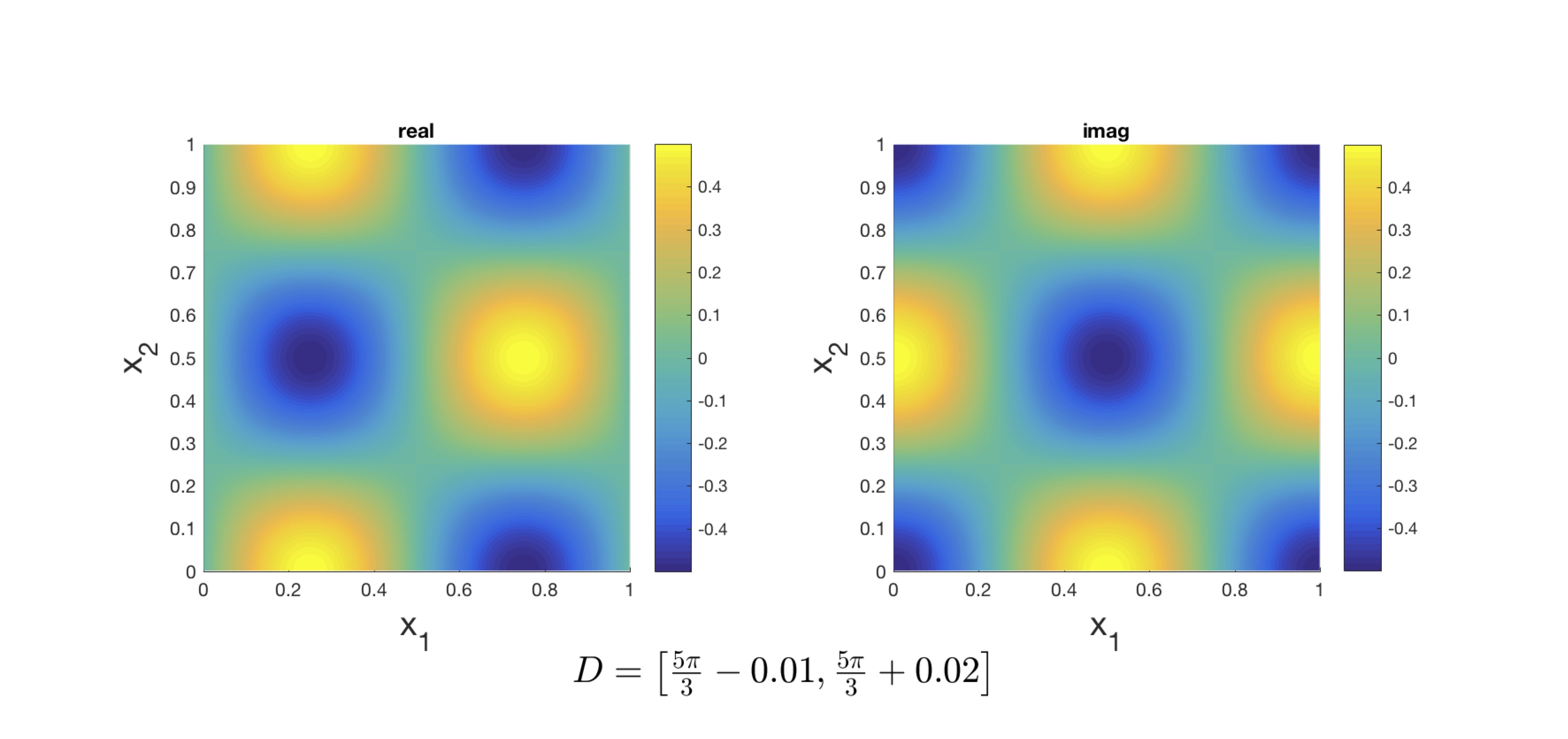}   \end{tabular}}
	\end{center}
	\caption{Results obtained for the translation map \eqref{eq:translation} with $m=2$ and $(\omega_1, \omega_2) = (\tfrac{1}{2}, \tfrac{1}{3})$. Computations were performed for the observable \eqref{eq:observablewhatever} at a spatial discretization of $\tilde{n} = 1000$.   }  \label{fig:6periodic}
\end{figure}
\iffalse
the root-squared-mean error:
\begin{equation}
\mathrm{rsm} := \left( \frac{1}{q(n)}\sum_{p_{n\jhatbold}\in \mathcal{P}_n} |(\Proj_D g)(\psi(p_{n,\jhatbold})) - (\Proj_{n,D} g_n)(\psi(p_{n,\jhatbold})) |^2  \right)^{\frac{1}{2}} \approx \left\| \Proj_{D} g - \Proj_{n,D} g_n \right\| 
\end{equation}
\fi
In \cref{fig:6periodic}, we examine the case when $d=2$ and $(\omega_1, \omega_2) = (\tfrac{1}{2}, \tfrac{1}{3})$. The trajectories of this map are all 6-periodic, and eigenvalues are also located 6th roots of unity  (i.e. $e^{i\frac{2\pi (k-1)}{6}}$, with $k=1,2,\ldots,6$). Shown are spectral projections \eqref{eq:targetapprx} of the observable: 
\begin{equation}
g(x) = \sin(2 \pi x_1 ) \cos(2 \pi x_2) +  \sin(\pi x_2) +  \frac{1}{\sin^2(\pi x_1)+1} - 1 \label{eq:observablewhatever}
\end{equation}
around narrow intervals of these eigenvalues.

\subsection{Arnold's cat map}
Consider the area-preserving map $T:\mathbb{T}^2\mapsto \mathbb{T}^2$, defined in \eqref{eq:catmap}:
$$
T(x) = \begin{bmatrix} 2 & 1 \\ 1 & 1 \end{bmatrix} \begin{bmatrix} x_1 \\ x_2 \end{bmatrix} \mod{1}.
$$
Arnold's cat map \eqref{eq:catmap} is a well-known example of a Anosov diffeomorphism, in which the dynamics are locally characterizable by expansive and contractive directions. The map \eqref{eq:catmap} has positive Kolmogorov entropy, is strongly mixing, and therefore also ergodic. 
\subsubsection{Singular and regular subspaces}
The Koopman operator associated with Arnold's cat map is known to have a  ``Lebesgue spectrum'' (see e.g. \cite{Arnold1989}). A Lebesgue spectrum implies the existence  of an orthonormal basis for $L^2(\mathbb{T}^2, \mathcal{B}(\mathbb{T}^2), \mu)$ consisting of the constant function and $\{\varphi_{s,t}\}_{s\in I, t\in \mathbb{Z} }$, $I\subseteq \mathbb{N}$, such that:
$$  \K^l \varphi_{s,t} = \varphi_{s,t+l}, \quad l\in \mathbb{Z}. $$
It follows from such a basis that the eigenspace at $\lambda=1$ is simple, i.e. consisting only the constant function, with the remaining part of the spectrum absolutely continuous. In the case of the Cat map, the basis $\{\varphi_{s,t}\}_{s\in I, t\in \mathbb{Z} }$ is just a specific re-ordering of the Fourier elements $ \{  e^{2\pi i  (k_1 x_1 + k_2 x_2)} \}_{(k_1,k_2)\in\mathbb{Z}^2}$. Indeed, by applying the Koopman operator \eqref{eq:KOOPMAN} onto a Fourier element, we observe the transformation:
\begin{equation} 
\K  e^{2\pi i  (k_1 x_1 + k_2 x_2)} = e^{i 2\pi ({k'}_1 x_1 + {k'}_2 x_2)} , \qquad\begin{array}{lcl}  {k'}_1 & = & 2 k_1 + k_2 \\ {k'}_2 & = & k_1 + k_2 \end{array}, \qquad (k_1, k_2)\in \mathbb{Z}^2.   \label{eq:orbits}
\end{equation}
Partitioning $\mathbb{Z}^2$ into the orbits of \eqref{eq:orbits}, i.e. $ \mathbb{Z}^2 = \xi_{0} \cup  \xi_{1}  \cup \xi_{2} \cup \xi_{3} \ldots$, 
we notice that all orbits, except for $\xi_0 = \{(0,0)\}$, consists of countable number of elements. The singular and regular subspaces are respectively:
$$ H_s =   {\spn \{ 1   \}},\qquad H_r  =  \overline{\spn \{ e^{i 2\pi (k_1 x_1 + k_2 x_2)}:\quad (k_1,k_2)\ne 0  \}}.   $$

\subsubsection{Spectral density function}
The spectral density function \eqref{eq:densityapprx} of a generic observable is found by solving the trigonometric moment problem \cite{Akhiezer1963}:
$$ d_l := \langle \K^l g, g \rangle = \int_{\mathbb{T}^2} \left( \K^l g \right)^* (x) g(x)  \dd\mu =  \int_{[-pi,\pi)} e^{il\theta} \rho(\theta;g) \dd\theta, \quad l\in \mathbb{Z}.$$
If the observable consists of a single Fourier element, we note: (i) $(k_1,k_2) = (0,0)$ implies $d_l=1$, $\forall l \in \mathbb{Z}$ which leads to $\rho(\theta; 1) = \delta(\theta)$ (i.e. Dirac delta distribution),  (ii)  $(k_1,k_2) \ne (0,0)$ implies $d_l = 0$  $\forall l\ne 0$ and $d_0 = 1$, leading to $\rho(\theta; e^{i 2\pi (k_1 x_1 + k_2 x_2)}) = \tfrac{1}{2\pi}$ (i.e. a uniform density). By virtue of these properties, the spectral density function \eqref{eq:densityapprx} of a generic observable in the form:
\begin{equation} g = a_0 + \sum_{s\in \mathbb{N}} \sum_{t\in \mathbb{Z}} a_{s,t} \varphi_{s,t} \in L^2(\mathbb{T}^2, \mathcal{B}(\mathbb{T}^2), \mu),\qquad {|a_0|^2} + \sum_{s\in \mathbb{N}} \sum_{t\in \mathbb{Z}} | a_{s,t}|^2 < \infty,   \label{eq:glebesgue}
\end{equation}
can be expressed as:
\begin{equation} \rho(\theta;g) =  |a_0|^2 \delta(\theta) +   \frac{1}{2\pi}     \sum_{l\in \mathbb{Z}}\left(\sum_{s\in \mathbb{N}} \sum_{t\in \mathbb{Z}} a^*_{s,t-l} a_{s,t}  \right) e^{i l\theta} .      \label{eq:specdensitycatmap}
\end{equation}
\subsubsection{Spectral projectors} For an interval $D = [c-\delta,c+\delta) \subset [-pi,\pi)$, one can use the Fourier series expansions of the indicator function: 
\begin{equation} \chi_D  = \sum_{l\in \mathbb{Z}} b_l(D) e^{il\theta}, \qquad b_l(D) := \int_{[-pi,\pi)} e^{-i\theta} \chi_D(\theta) \dd \theta =  \frac{1}{2 \pi} \begin{cases} \frac{i}{ l}e^{-i c l}( e^{-i \delta l} - e^{i \delta l} ) & l \ne 0 \\ {2 \delta}  & l = 0  \end{cases}.  \label{eq:indD}
\end{equation}
to obtain an approximation of the spectral projector with the help of functional calculus. Using the Lebesgue basis in \eqref{eq:glebesgue}, the following expression for the spectral projection can be obtained: 
$$ \Proj_D g  = \mathbb{\nu}(D) + \sum_{l\in \mathbb{Z}} b_l(D) \K^l g =  a_0 \mathbb{\nu}(D) +  \sum_{l\in \mathbb{Z}} b_l(D) 
\left( \sum_{s\in \mathbb{N}} \sum_{t\in \mathbb{Z}} a_{s,t} \varphi_{s,t+l}\right)  $$
where the singular measure $\mathbb{\nu}(D) = 1$ whenever $0\in D$. The frequency content of the projection increases drastically as the width of the interval is shrinked.  Assuming $a_0 = 0$, and $\left\| g\right\| = 1$, we see that:
$$ \frac{\Proj_D g}{ \left\| \Proj_D g \right\|}  =  \sum_{l\in \mathbb{Z}} \frac{b_l(D)}{2\delta} 
\left( \sum_{s\in \mathbb{N}} \sum_{t\in \mathbb{Z}} a_{s,t} \varphi_{s,t+l}\right), \qquad \left|\frac{b_l (D)}{2 \delta}\right|^2 = \frac{2-2\cos(l\delta)}{(2 \pi l \delta)^2}, 
\quad l\ne 0.$$
The Lebesgue basis functions $\varphi_{s,t}$ become increasingly oscillatory for larger values of $t$. Reducing the width of the interval places more weight on the higher frequency components. As $\delta\rightarrow 0$, the frequency content in fact becomes unbounded.

\subsubsection{Numerical results}
For observables that consist of only a finite number of Fourier components, a closed-form expression of the spectral density function can be obtained. For example:
\begin{equation}
\begin{aligned}[l]
g_1(x) &=  e^{2\pi i(2 x_1+ x_2)}\\
g_2(x) &=  e^{2\pi i(2 x_1+ x_2)} + \frac{1}{2}e^{(2\pi i(5 x_1+ 3 x_2))}\\
g_3(x) &= e^{2\pi i(2 x_1+ x_2)} + \frac{1}{2}e^{(2\pi i(5 x_1+ 3 x_2))}+ \frac{1}{4} e^{2\pi i (13 x_1 + 8 x_2) } 
\end{aligned}
\quad\Longrightarrow\quad
\begin{aligned}[l]
\rho(\theta;g_1) &= \frac{1}{2\pi}\\
\rho(\theta;g_2) &= \frac{1}{2\pi}\left( \frac{5}{4}+\cos \theta \right) \\
\rho(\theta;g_3) &= \frac{1}{2\pi} \left(\frac{21}{16}+\frac{10}{8} \cos \theta +\frac{1}{2} \cos 2\theta \right)
\end{aligned}   \label{eq:spectdensityg1g2g3}
\end{equation}
In \cref{fig:densityplotcatmap1}, the spectral densities of these observables are approximated with \eqref{eq:densityapprx} using the proposed method.  Clearly, the result indicate that better approximations are obtained by increasing the discretization level.   In \cref{fig:catmapprojections}, the spectral projections for the first observable are plotted. As expected, shrinkage of the interval leads to more noisy figures.  
\begin{figure}[h!]
	\begin{center}
		\subfigure[Approximation of $\rho(\theta;g_1)$. ]{\includegraphics[ width=.24\textwidth ]{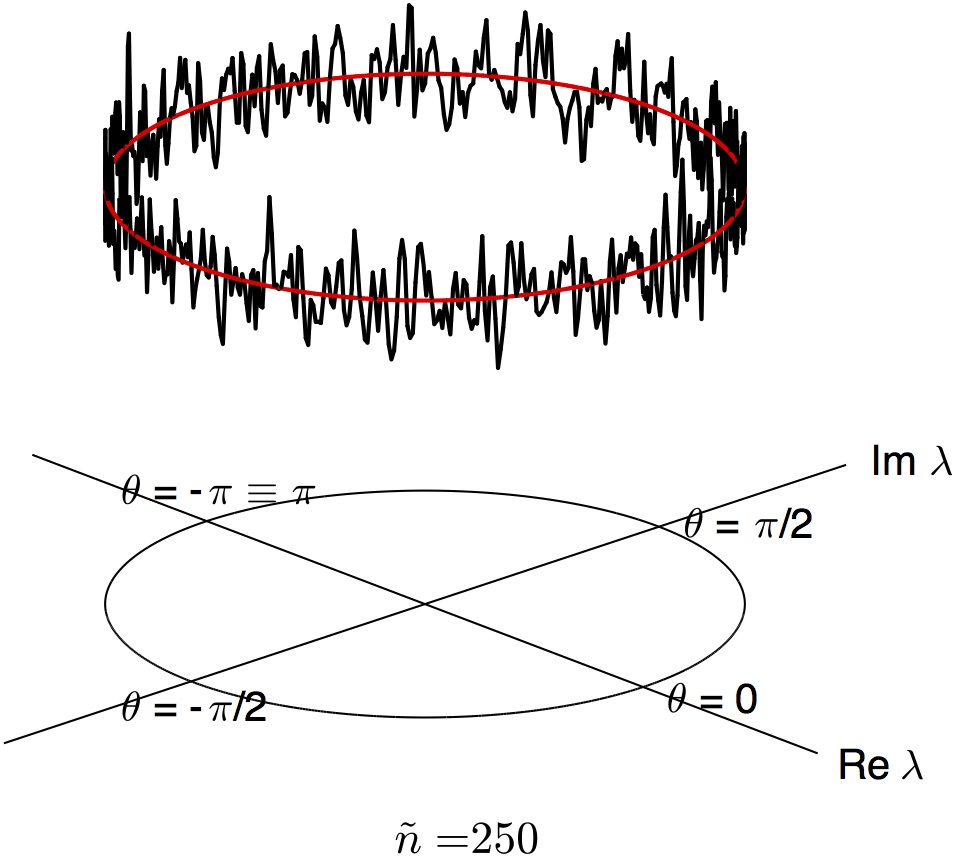}\includegraphics[width=.24\textwidth]{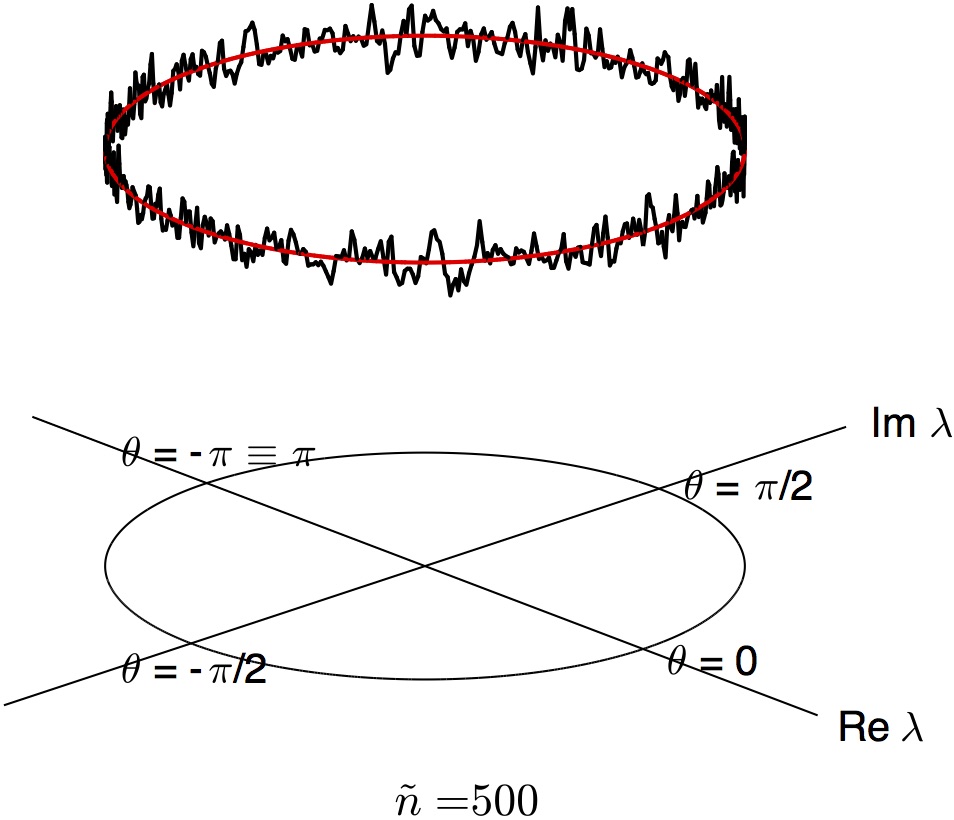}\includegraphics[width=.24\textwidth]{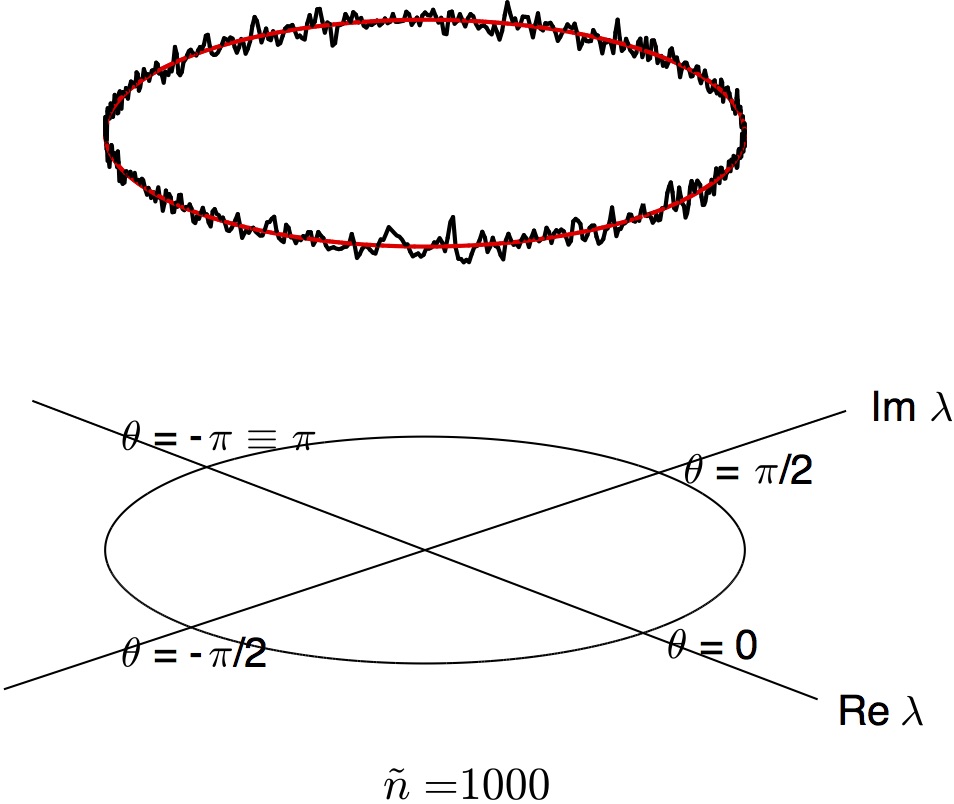}\includegraphics[width=.24\textwidth]{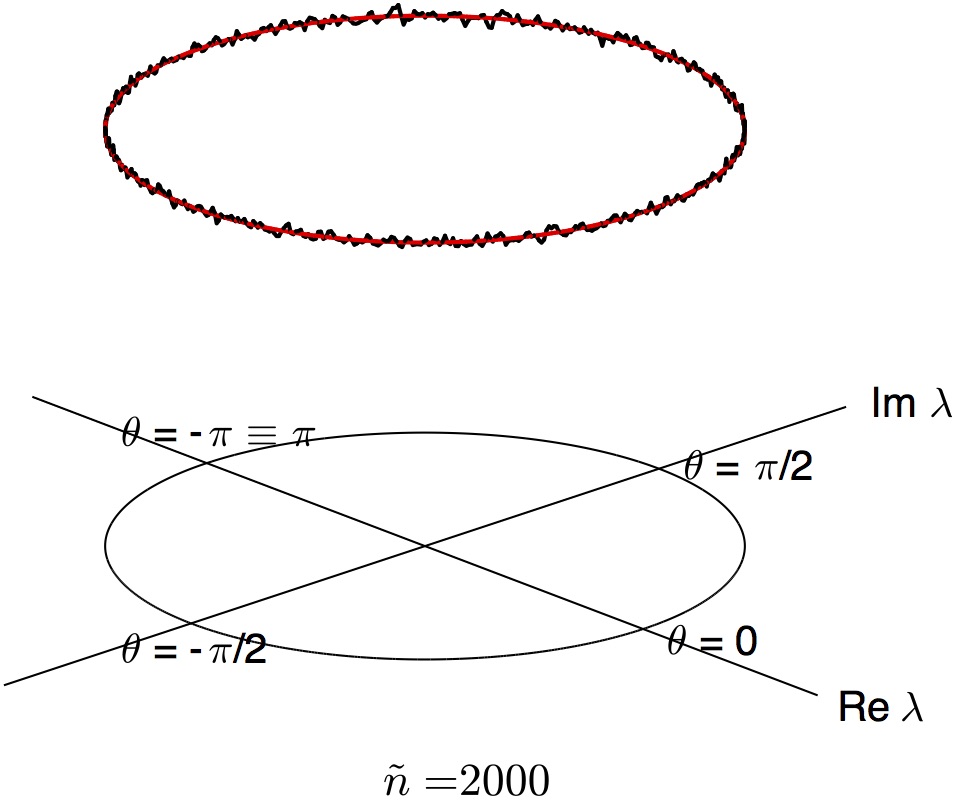}  } \\
		\subfigure[Approximation of $\rho(\theta;g_2)$. ]{\includegraphics[ width=.24\textwidth ]{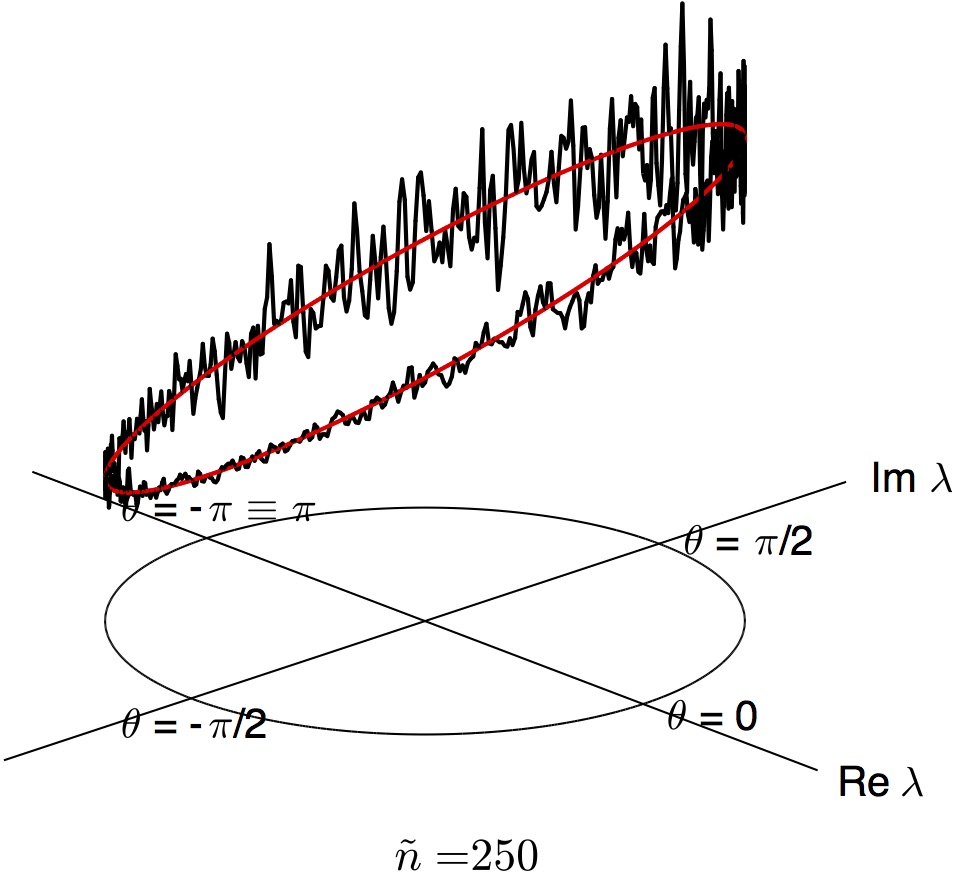}\includegraphics[width=.24\textwidth]{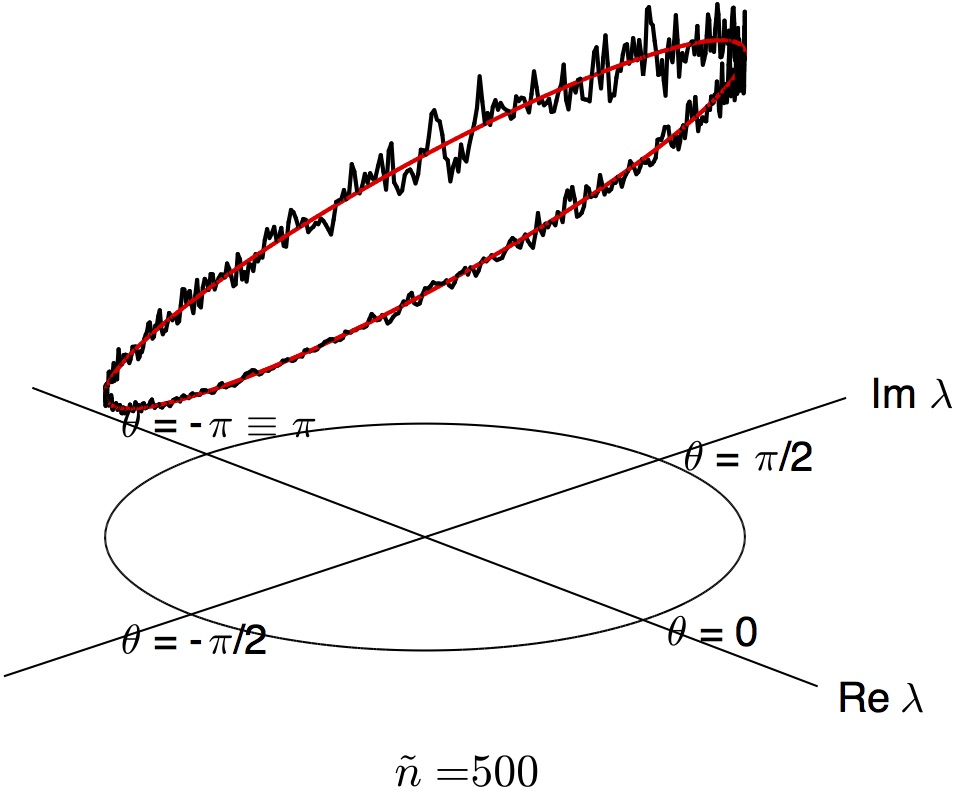}\includegraphics[width=.24\textwidth]{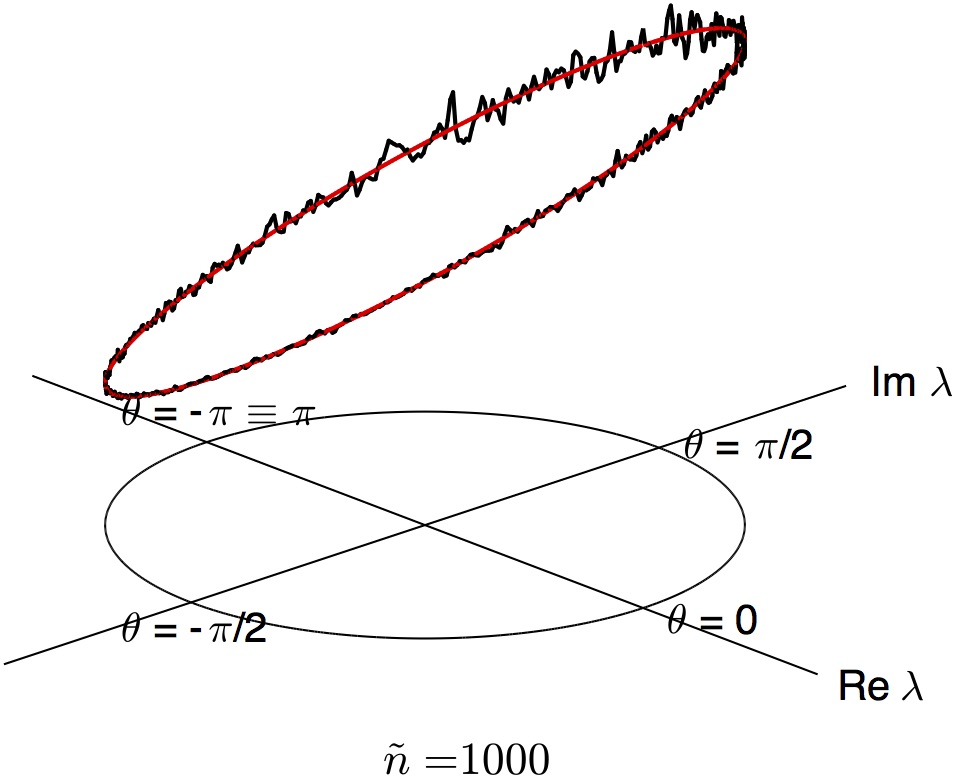}\includegraphics[width=.24\textwidth]{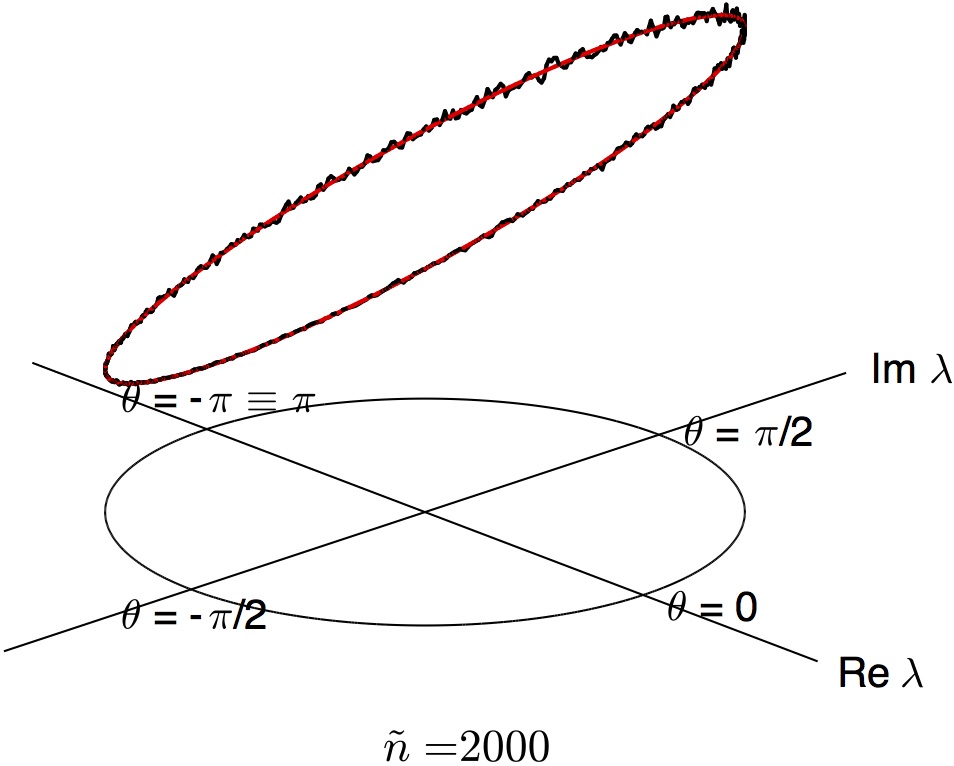}  } \\
		\subfigure[Approximation of $\rho(\theta;g_3)$.]{\includegraphics[ width=.24\textwidth ]{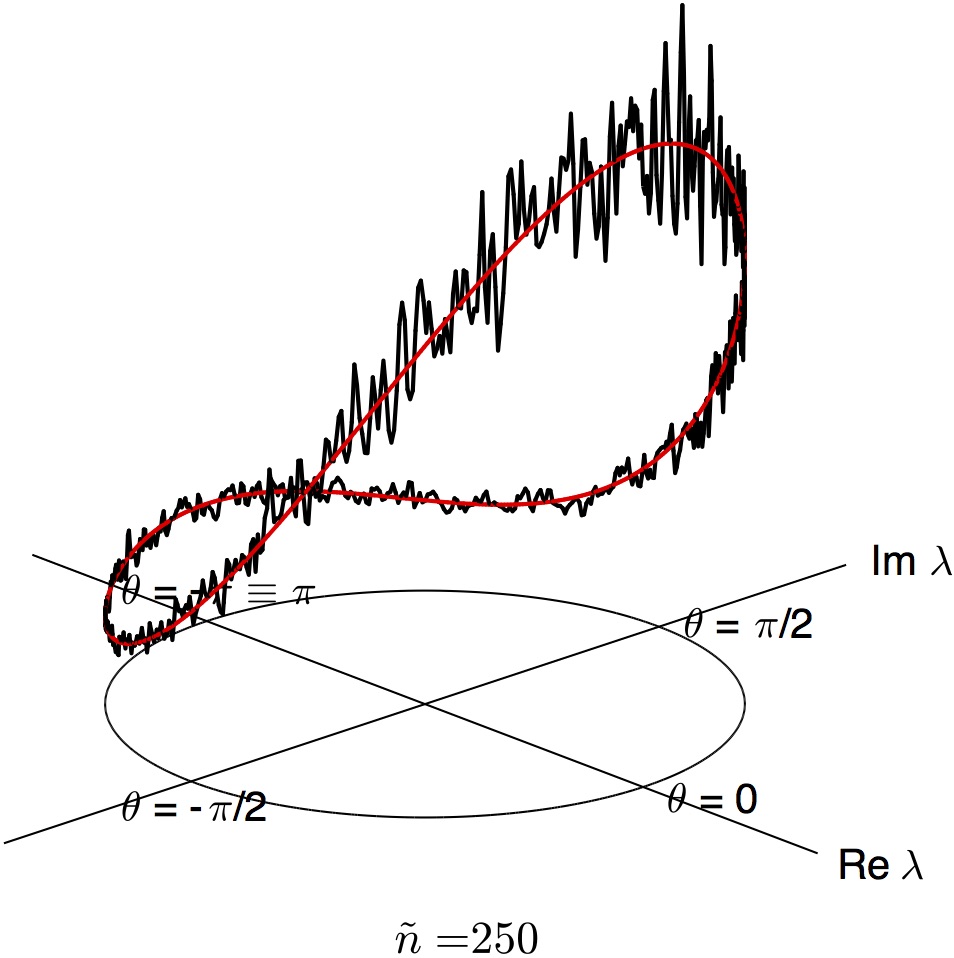}\includegraphics[width=.24\textwidth]{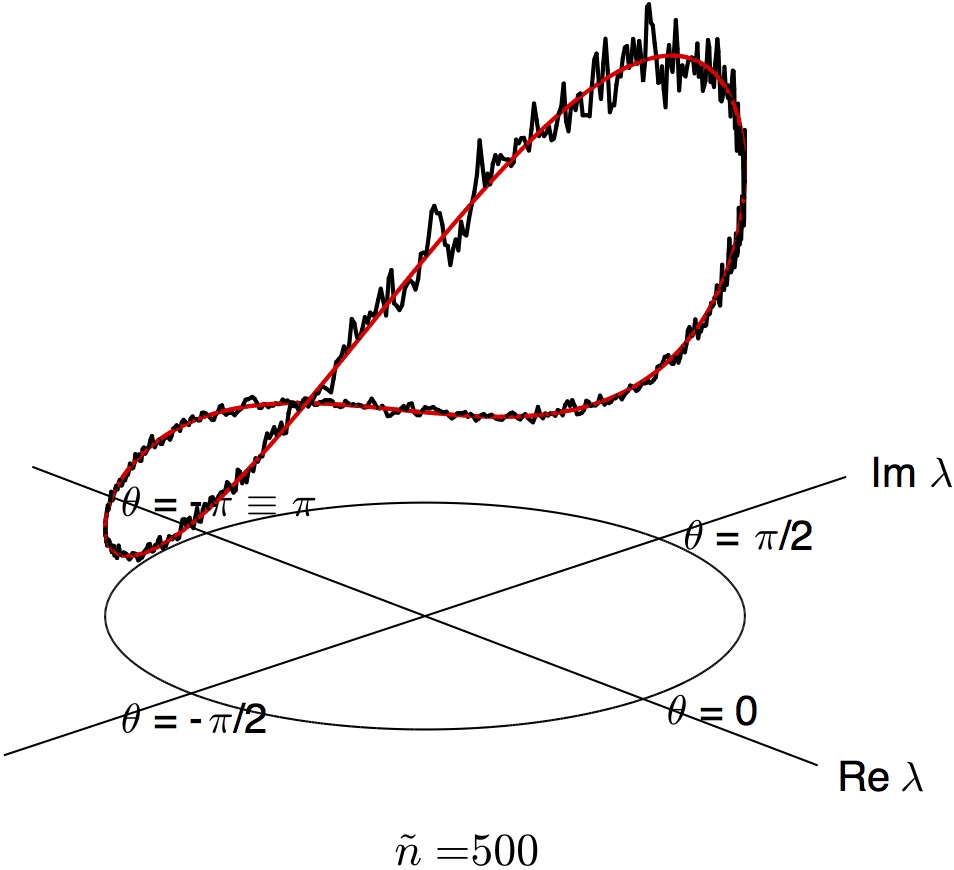}\includegraphics[width=.24\textwidth]{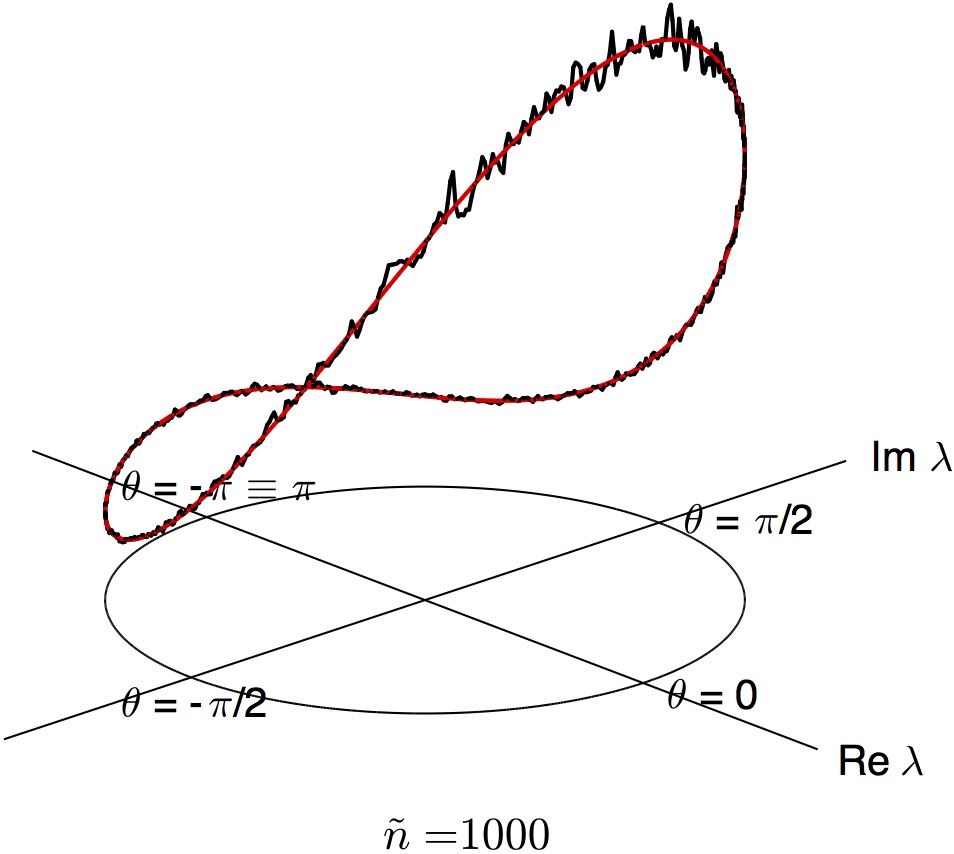}\includegraphics[width=.24\textwidth]{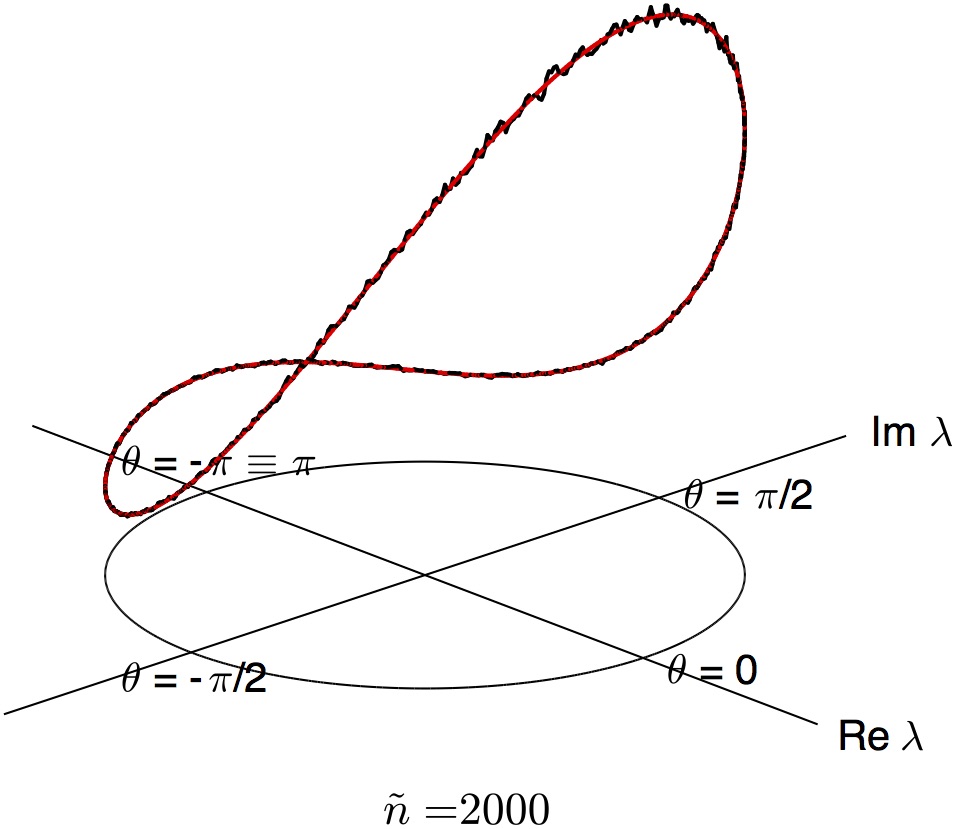}  }
	\end{center}
	\caption{Approximations of the spectral density function of the cat map \eqref{eq:catmap} for the observables in \eqref{eq:spectdensityg1g2g3}. The spectral resolution is set to $\alpha = 2\pi / 500$. The black curves denote the approximations and the red curve is the true density.} \label{fig:densityplotcatmap1}
\end{figure}
\clearpage
\begin{figure}[h!]
	\begin{center}
		\begin{tabular}{cc}
			\includegraphics[width=.47\textwidth]{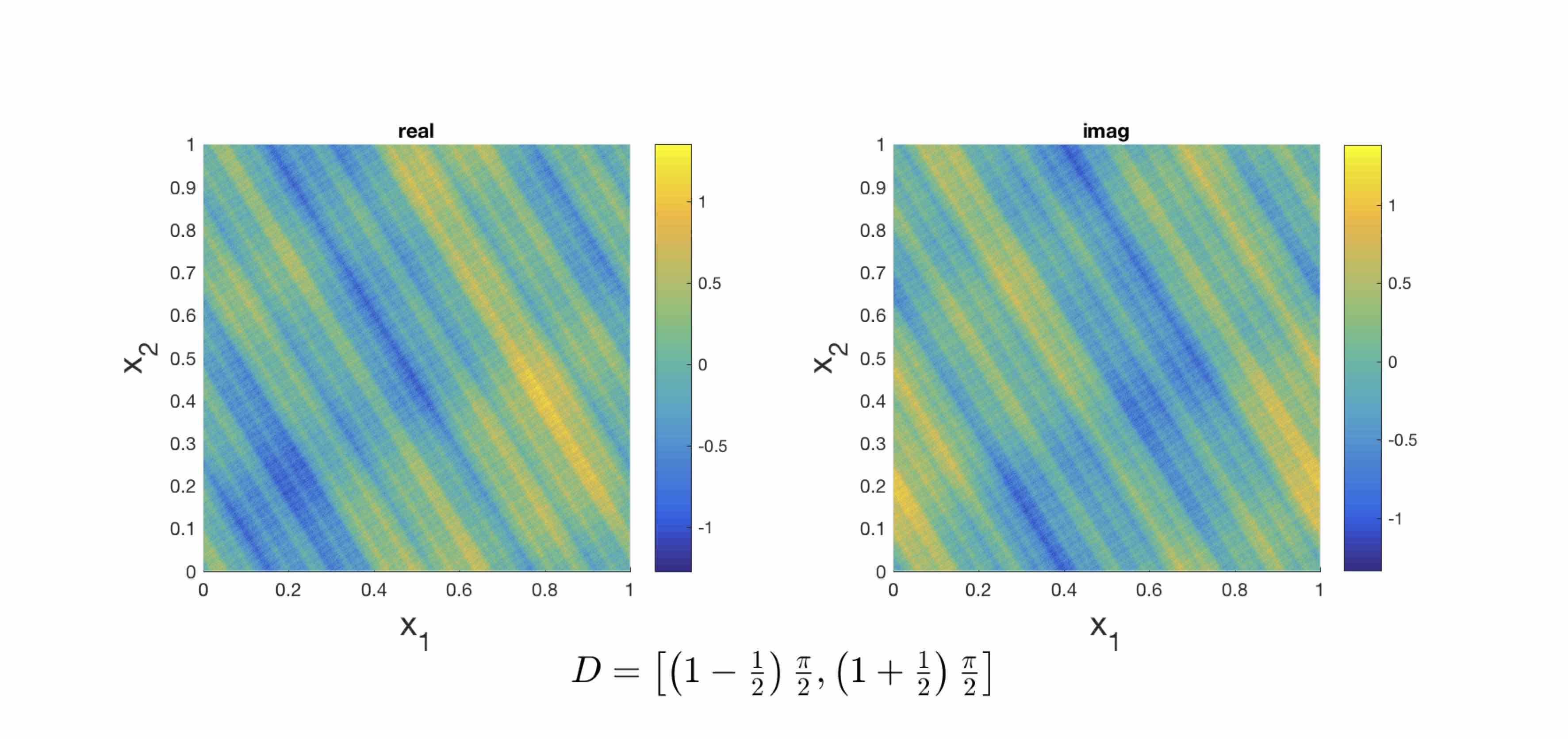} &  \includegraphics[width=.47\textwidth]{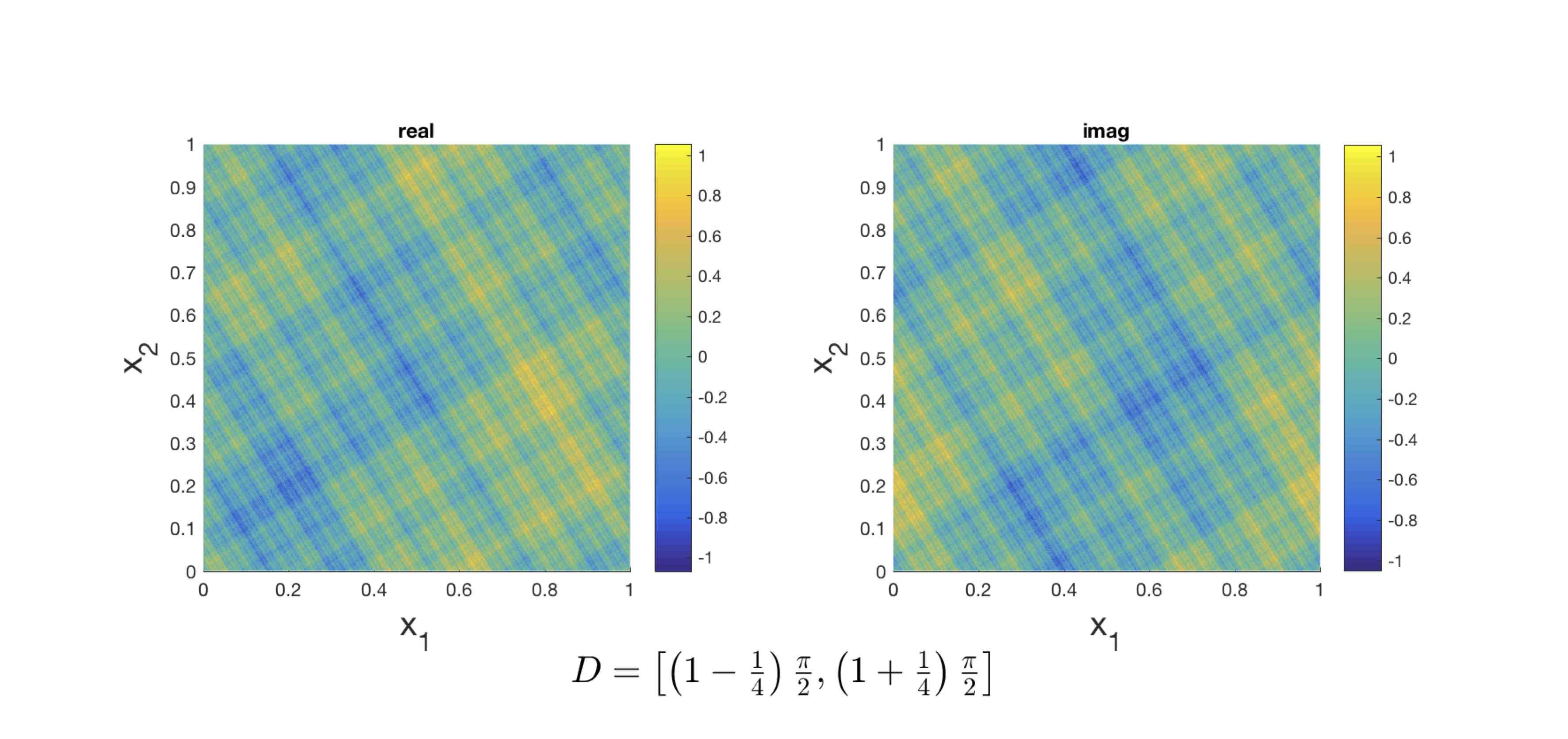} \\
			\includegraphics[width=.47\textwidth]{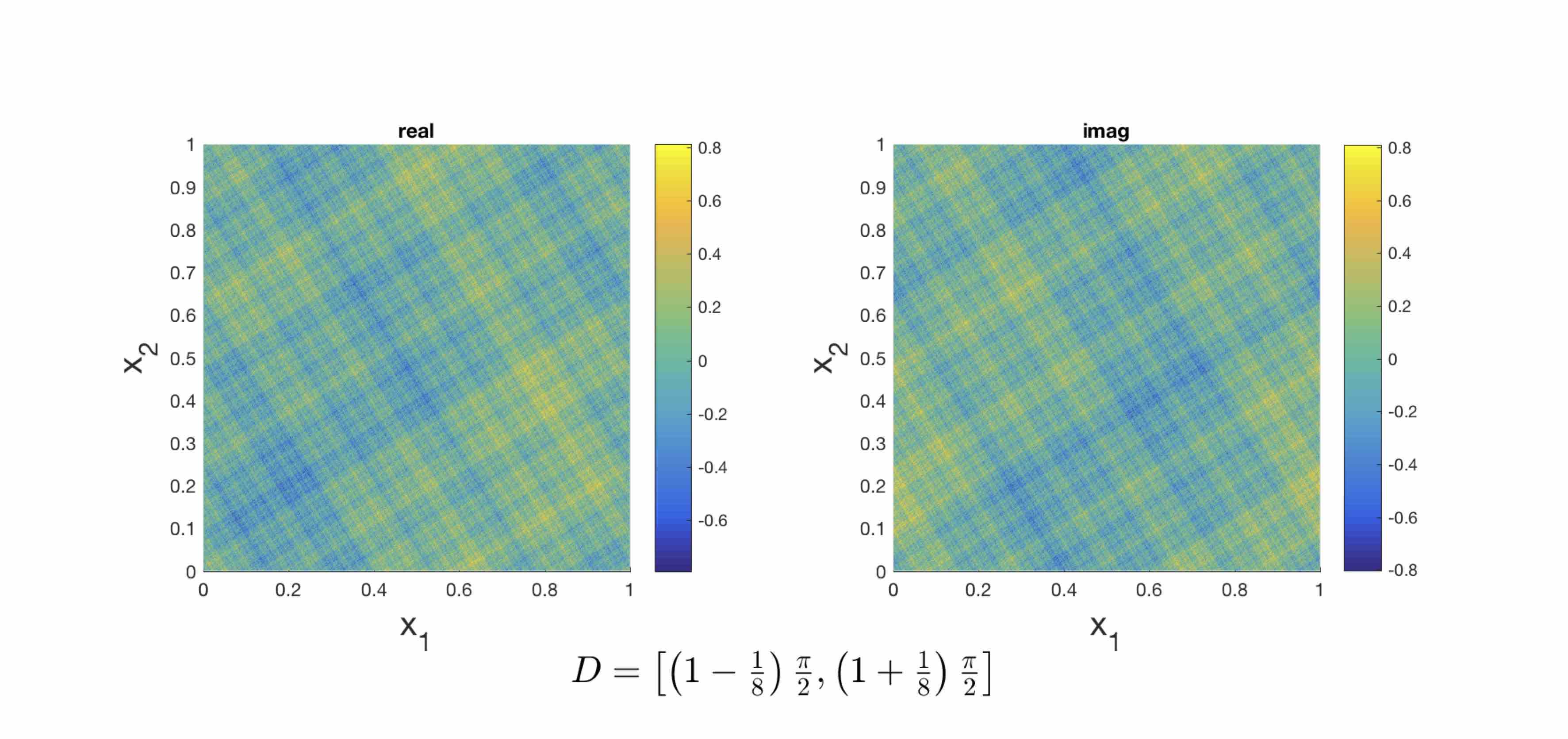} & \includegraphics[width=.47\textwidth]{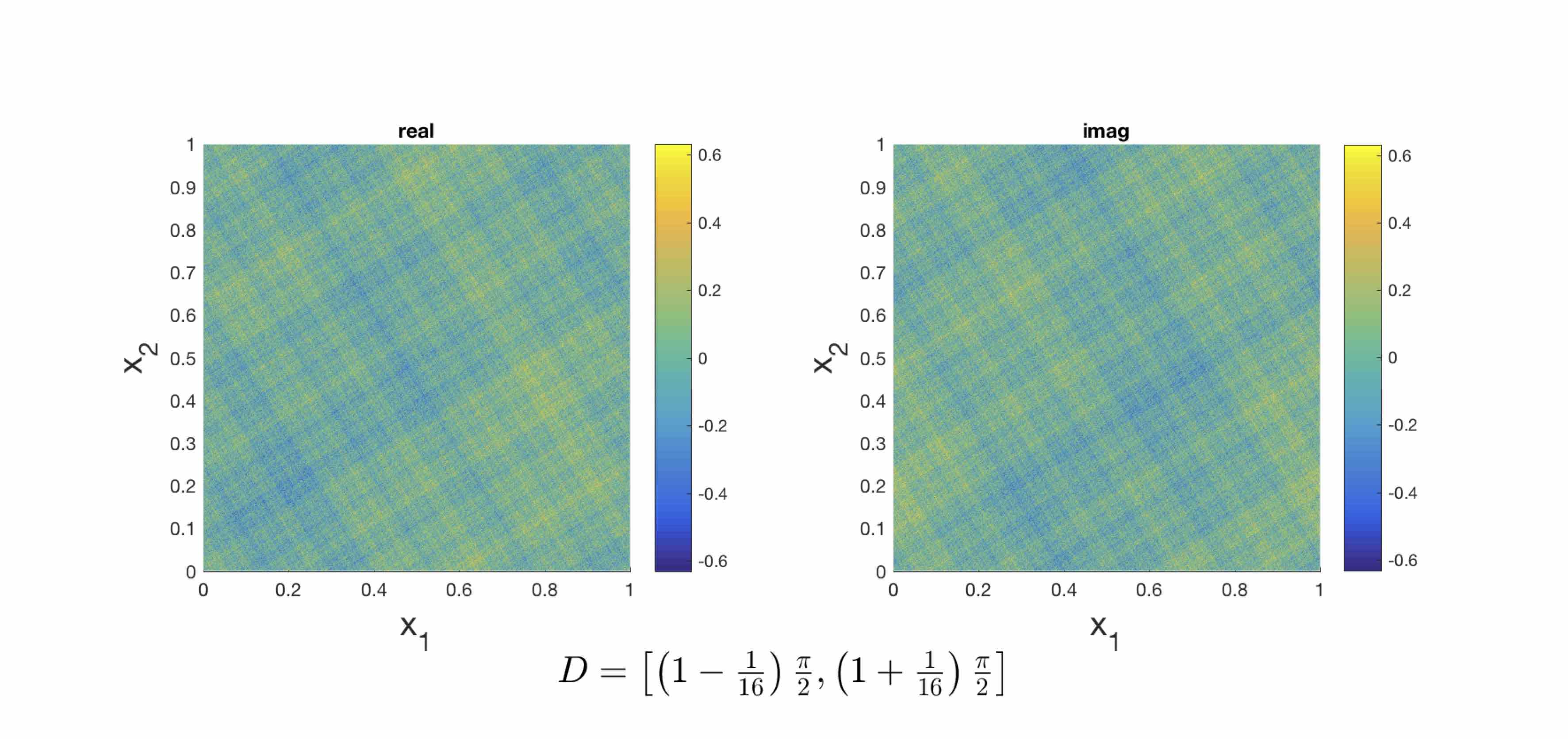} 
		\end{tabular}
	\end{center}
	\caption{The spectral projections of cat map \eqref{eq:catmap} for the  observable in \eqref{eq:spectdensityg1g2g3}. Projections were computed at the discretization level $\tilde{n}=2000$. } \label{fig:catmapprojections}
\end{figure}

\subsection{Anzai's skew-product transformation}
One can use skew-products of dynamical systems to construct examples of maps that have mixed spectra \cite{anzai1951ergodic}. An example of such a transformation is the map \eqref{eq:skewproduct}:
\begin{equation}
T(x_1, x_2) = (x_1+\gamma, x_1 + x_2) \mod 1,
\end{equation}
for which $\gamma\in [0,1)$.  The map \eqref{eq:skewproduct} is the composition of a translation and a shear. The transformation is ergodic whenever $\gamma\in\mathbb{R}$ is irrational.
\subsubsection{Singular and regular subspaces} The mixed spectrum of the operator is recognized by examining the cyclic subspaces generated by the Fourier basis elements. Essentially, Fourier elements that solely depend on the coordinate $x_1$ behave as if the dynamics of the map are that of a pure translation, whereas the remaining Fourier elements do observe a shear and belong to a cyclic subspace of infinite length. The singular and regular subspaces of the operator are respectively given by:
$$ H_s =   \overline{\spn \{ \sigma_{t} := e^{  2\pi i (t x_1)}, \quad t\in \mathbb{Z}    \} }   $$
and 
$$ H_r  =  \overline{\spn \{\varphi_{s,t,w} :=  e^{  2\pi i ((t+sw)x_1 + s x_2)},\quad s,w\in \mathbb{Z}, t\in\mathbb{N}  \mbox{ with }s\ne 0, t < |s|  \}}. $$ 
\subsubsection{Spectral density function}
Application of the Koopman operator yields:
$$ \K^l \sigma_{t} = e^{ 2\pi i lt \gamma} \sigma_{t} \qquad \mbox{and} \qquad   \K^l   \varphi_{s,t,w} =  e^{ 2\pi il \left( t + s \left(  w + \tfrac{l-1}{2}  \right) \right) \gamma }  \varphi_{s,t,w+l}. $$
Solving the trigonometric moment problem for observables expressed in the form:
$$ g = \sum_{t\in \mathbb{Z}} a_t \sigma_t  + \sum_{s\in \mathbb{Z}, s\ne 0} \sum_{t\in \mathbb{N}, t < |s|} \sum_{w\in \mathbb{Z}} c_{s,t,w} \varphi_{s,t,w}  \in L^2(\mathbb{T}^2, \mathcal{B}(\mathbb{T}^2), \mu), $$
yield the density function:
\begin{equation} \rho(\theta;g) =  \sum_{t\in \mathbb{Z}} |a_t|^2 \delta(\theta - 2\pi\gamma t \mod 2\pi) +   \frac{1}{2\pi} \sum_{l\in \mathbb{Z}} \left( \sum_{s\in \mathbb{Z}, s\ne 0} \sum_{t\in \mathbb{N}, t < |s|}  \left( \sum_{w\in \mathbb{Z}} e^{ -2\pi il \left( t + s \left(  w + \tfrac{l-1}{2}  \right) \right) \gamma } {c^*_{s,t,w-l}} c_{s,t,w}  \right)  \right) e^{i l\theta}.      \label{eq:specdensityanzai}
\end{equation}
\subsubsection{Spectral projectors} For the interval $D = [c-\delta,c+\delta) \subset [-pi,\pi)$, an expression for the spectral projection is given by:
$$ \Proj_D g = \sum_{2\pi t\gamma \mod 2\pi \in D} a_{{t}} \sigma_t +  \sum_{l\in \mathbb{Z}} b_l(D)\left( \sum_{s\in \mathbb{Z}, s\ne 0} \sum_{t\in \mathbb{N}, t < |s|} \sum_{w\in \mathbb{Z}} e^{2\pi il \left( t + s \left(  w + \tfrac{l-1}{2}  \right) \right) \gamma } c_{s,t,w}  \varphi_{s,t,w+l}  \right), $$
where the coefficients $b_l(D)$ are defined as in \eqref{eq:indD}.

\subsubsection{Numerical results}
Now set $\gamma=1/3$ and consider the observable:
\begin{equation}  g(x_1,x_2 )= \frac{1}{20}e^{2\pi i x_1} +\frac{1}{20}e^{4\pi i x_1} + \frac{1}{5}e^{6\pi i x_1} + e^{2\pi i x_2} + \frac{1}{2}e^{(2\pi i(x_1+  x_2))}
\label{eq:obsanzai} 
\end{equation}
The spectral decomposition is given by:
$$  \rho(\theta;g) =  \frac{1}{400} \left(\delta(\theta ) + \delta(\theta - \frac{2\pi}{3})\right)  + \frac{1}{25}\delta(\theta - \frac{4\pi}{3}) + \frac{1}{2\pi}\left( \frac{5}{4}+\cos \theta \right) $$
In \cref{fig:anzaidensity}, the spectral density function is plotted for different discretization levels. Convergence of the spectra is observed by refinement of the grid. In \cref{fig:anzaieig}, spectral projections of the observable are shown for small intervals centered around the eigenvalues $e^{i\frac{2\pi (k-1)}{3}}$, with $k=1,2,3$. Note that the exact eigenfunctions are not recovered due to the presence of continuous spectra which is interleaved in the projection. In \cref{fig:anzaicont}, projections are shown in a region where only continuous spectrum is present.

\begin{figure}[h!]
	\begin{center}
		\subfigure[Approximation of the spectral density function at different discretization levels. The spectral resolution is set to $\alpha = 2\pi / 500$. The black curves denote the approximations and the red curves are the true density.]{\includegraphics[width=.3\textwidth ]{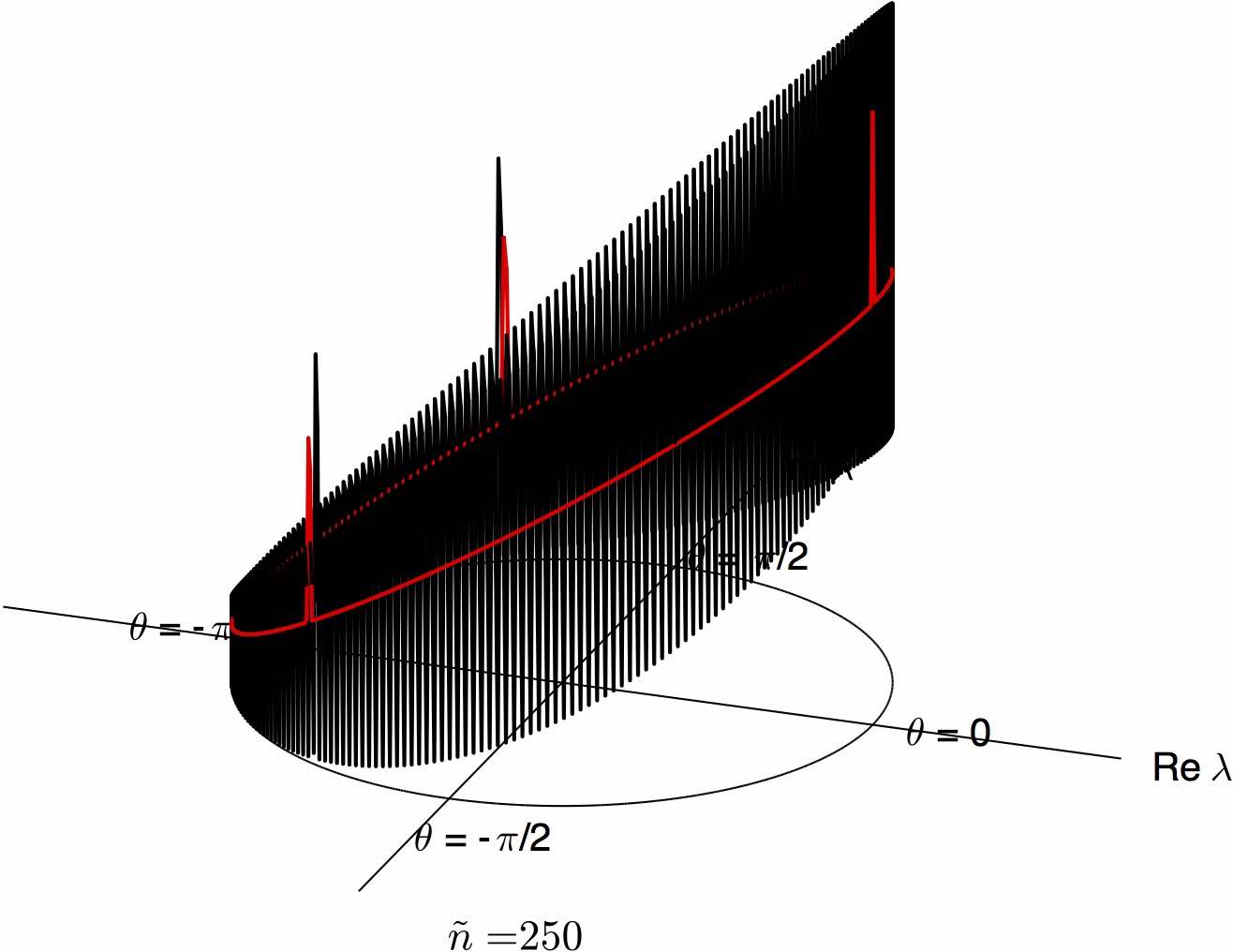} \includegraphics[width=.3\textwidth]{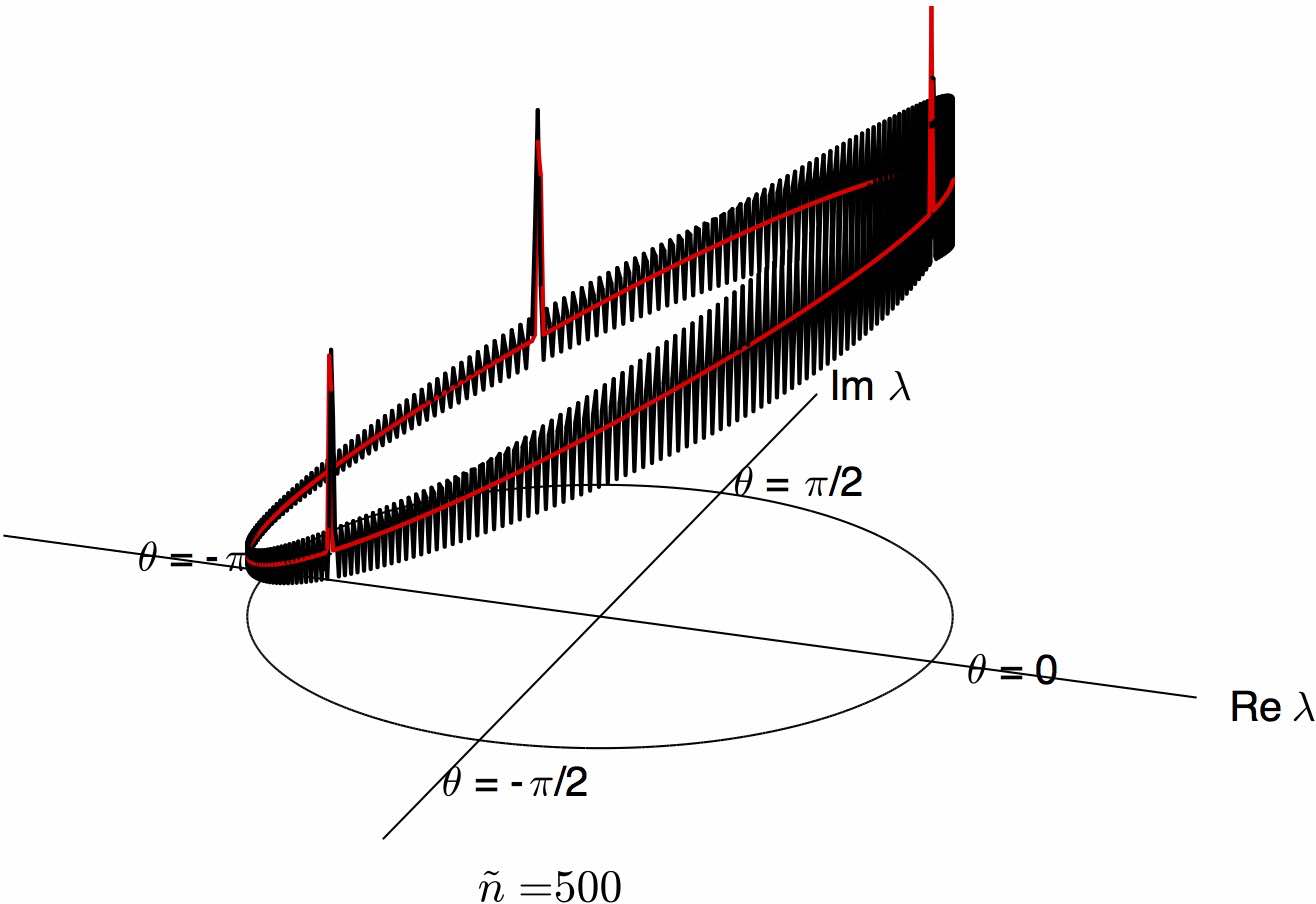}\includegraphics[width=.3\textwidth]{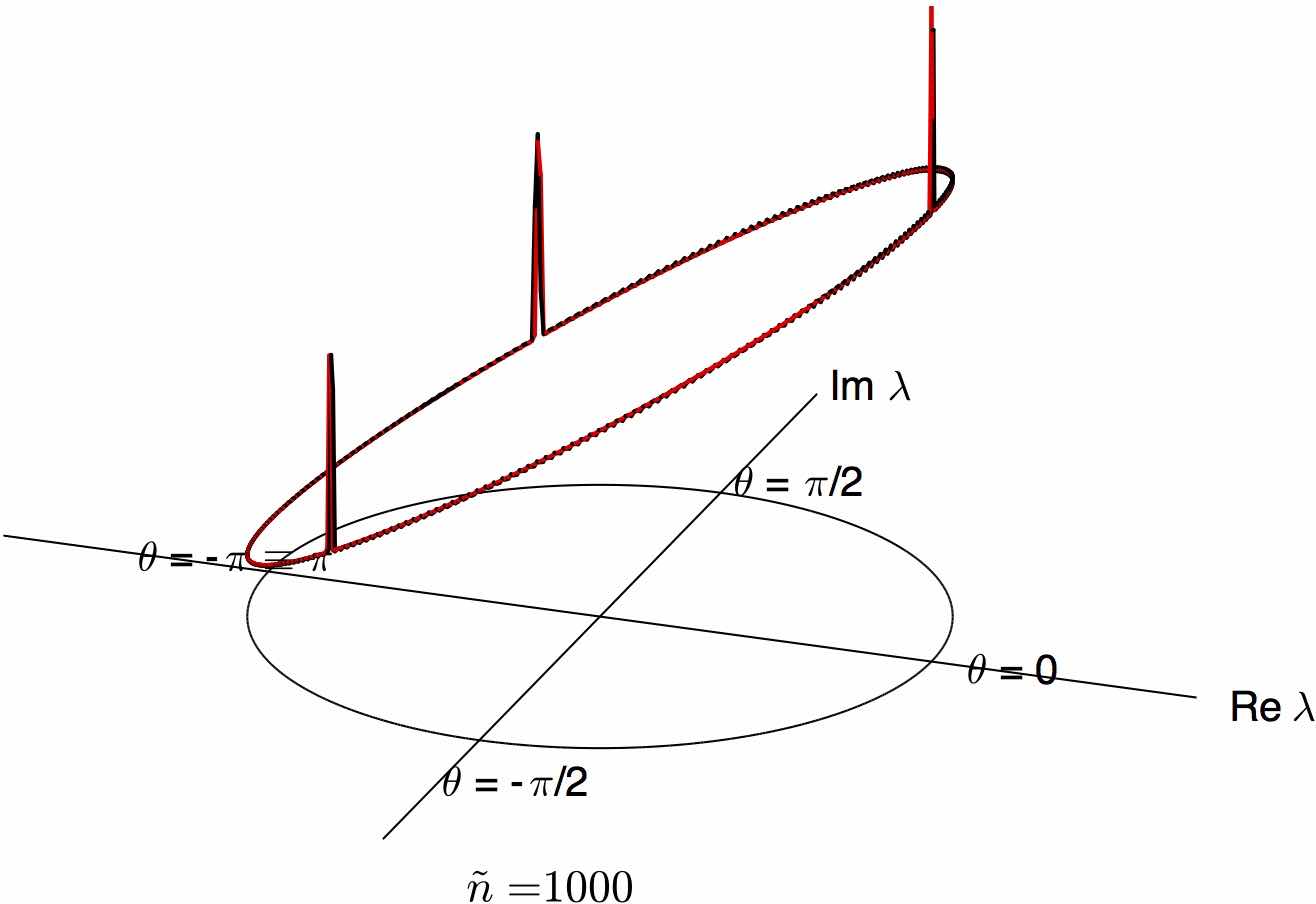} \label{fig:anzaidensity}  } \\
		\subfigure[Spectral projections computed at narrow intervals around the eigenfrequencies $\theta = 0,2\pi/3$.]{\includegraphics[width=.45\textwidth ]{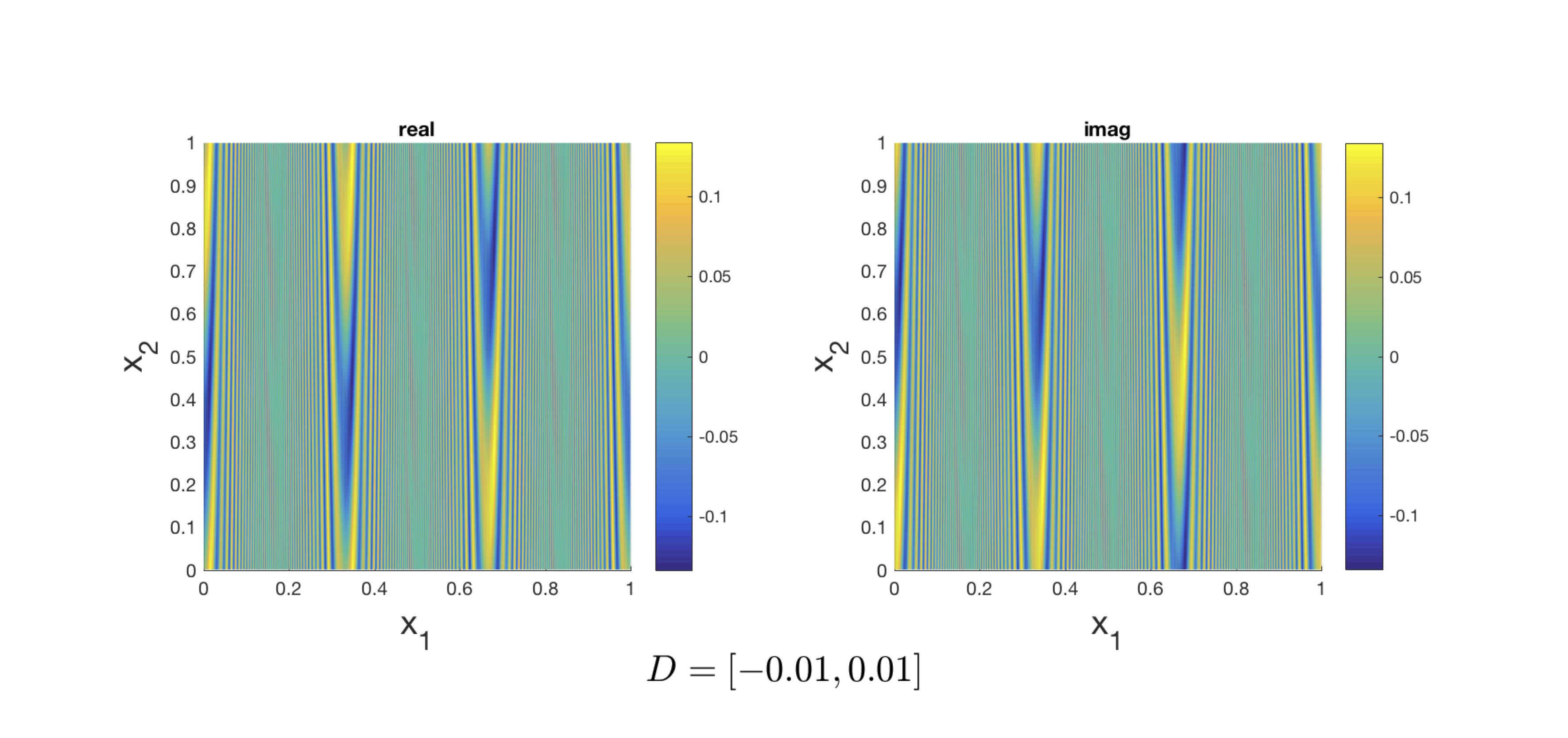}  \includegraphics[width=.45\textwidth]{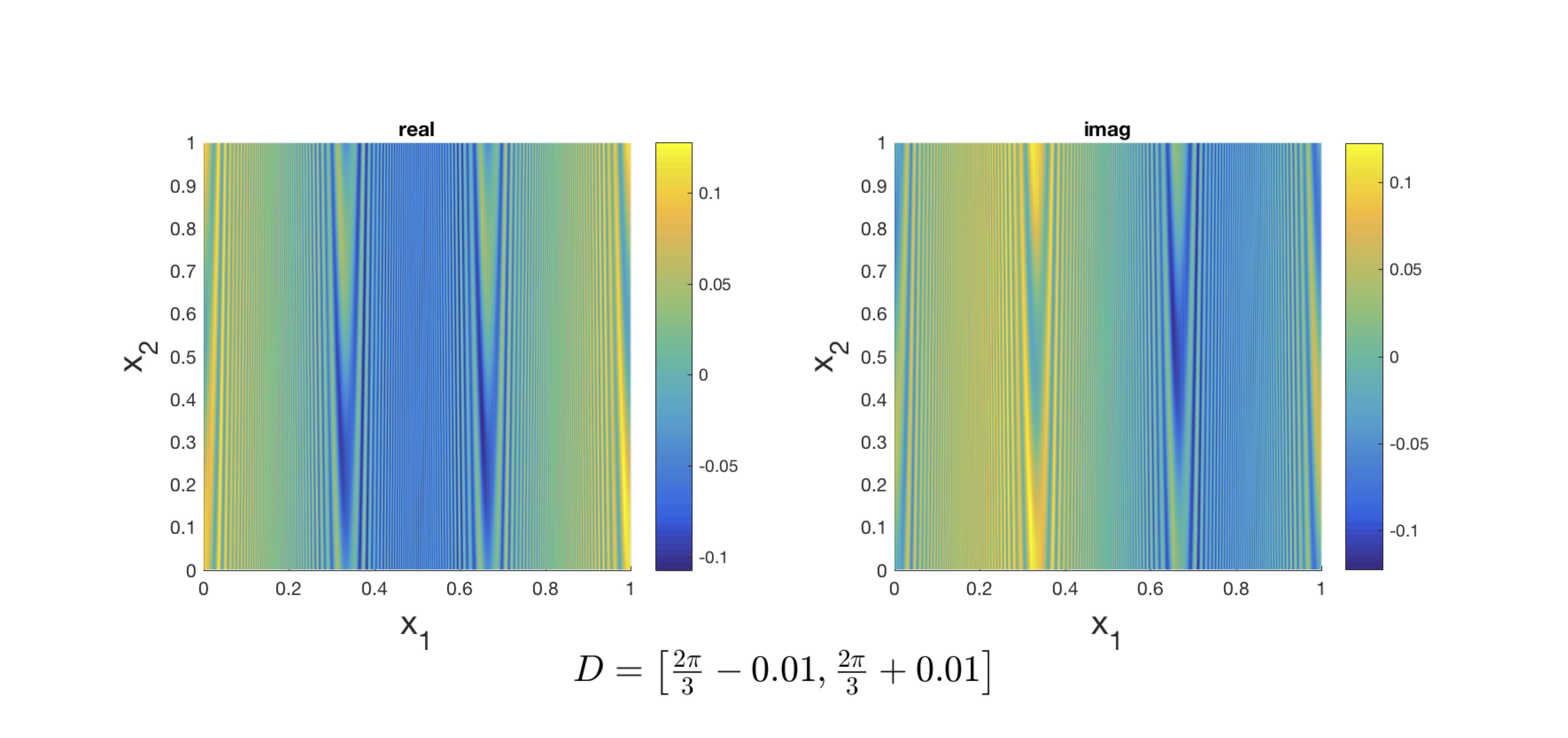} \label{fig:anzaieig} }  \\
		\subfigure[Spectral projections computed at intervals that contain only continuous spectra.]{\includegraphics[width=.45\textwidth ]{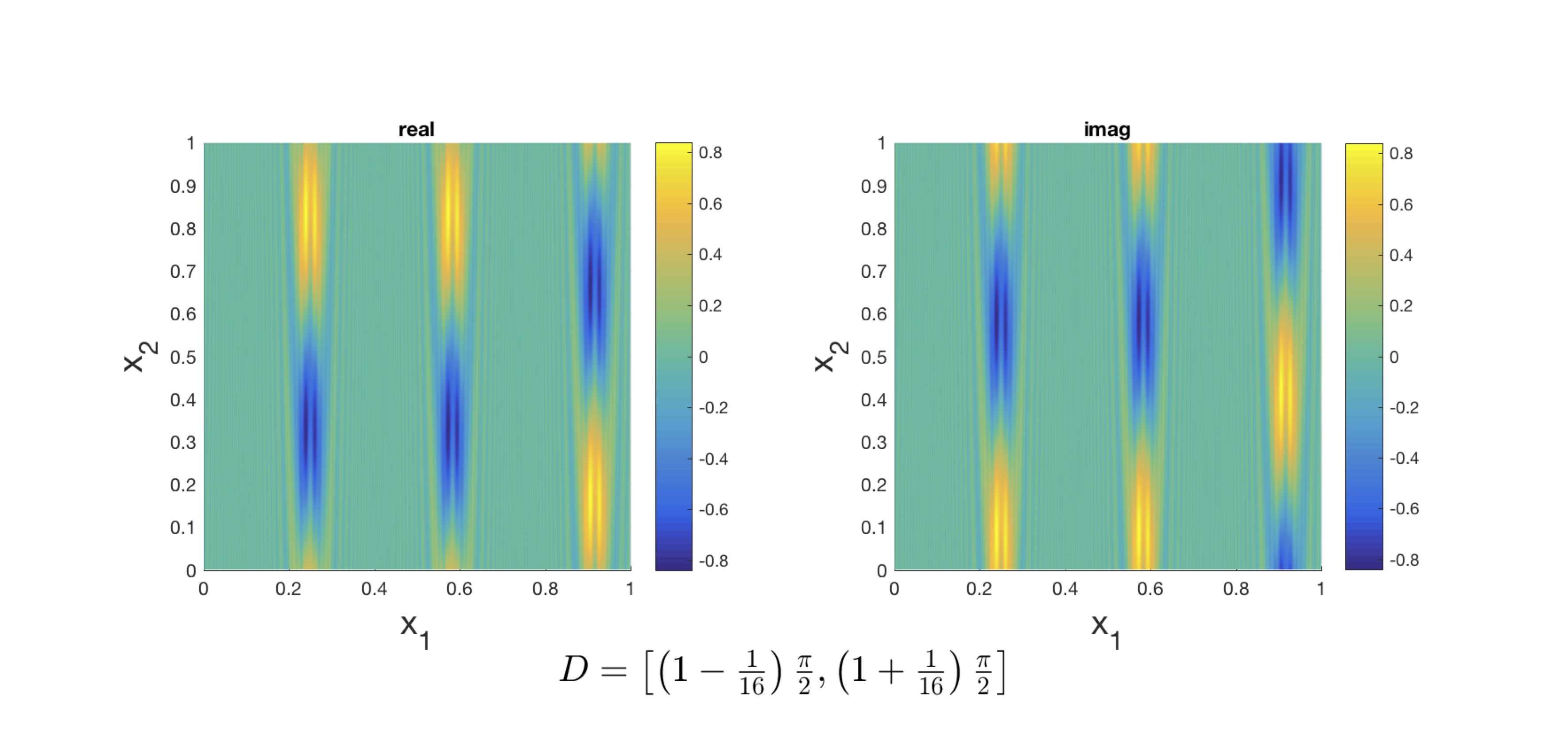} \includegraphics[width=.45\textwidth]{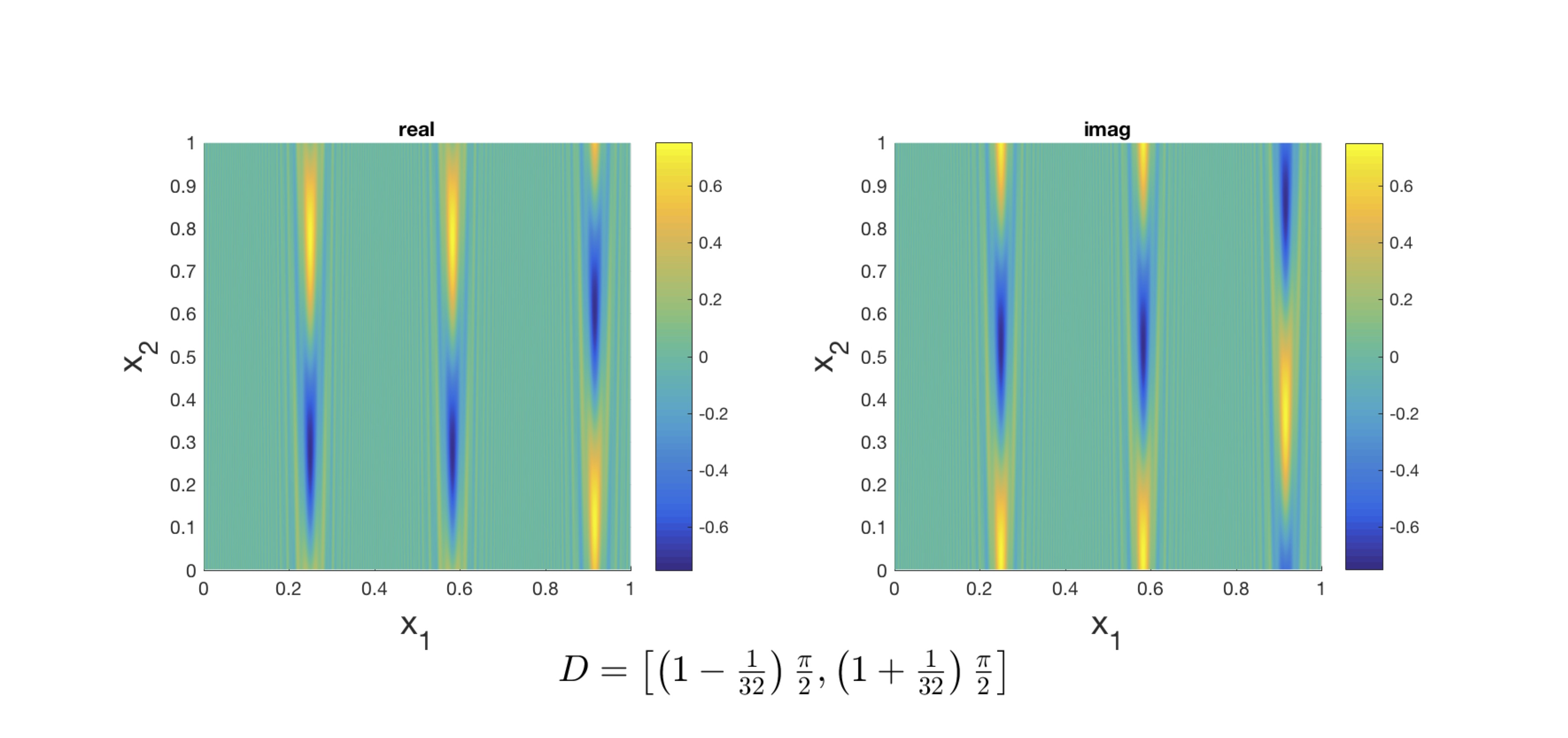} \label{fig:anzaicont} } 
	\end{center}
	\caption{Results obtained for Anzai's skew product transformation \eqref{eq:skewproduct} with $\gamma = 1/3$ and observable \eqref{eq:obsanzai}. Spectral projections were computed at a discretization level $\tilde{n}=2000$. }
\end{figure}

\section{Case study: the Chirikov standard map} \label{sec:chirikov}
In this section we apply our method to  the family of area-preserving maps introduced by \emph{Chirikov} \eqref{eq:chirikovfamily}:
$$ 
T(x) = \begin{bmatrix} x_1 + x_2 + K \sin(2\pi x_1)  \\  x_2 + K \sin(2\pi x_1) \end{bmatrix}\mod{1}
$$
Unlike the examples of the previous section, finding an explicit expression of the spectra is highly non-trivial (except for the case when $K=0$). In \cref{fig:chirikovdensity,fig:1period,fig:2period,fig:3period,fig:4period} the spectral properties of the standard map are examined for $K$-values ranging between $0$ and $0.35$.  In \cref{fig:chirikovdensity},  approximations of the spectral density functions \eqref{eq:densityapprx} are plotted for the observable:
\begin{equation}
g(x) = e^{i 4\pi x_1} + e^{i 3\pi x_1} + 0.01 e^{i 2\pi x_2}   \label{eq:chirikovobs}
\end{equation}
It can be seen that sharp peaks form at locations other than eigen-frequency $\theta = 0$. These peaks illustrate that the purely continuous spectra disintegrate, with the rise of discrete spectra for the operator. Eigenfunctions of \eqref{eq:KOOPMAN} may be recovered from the spectral projections by means of centering the projection on narrow intervals around the respective eigenfrequency. In \cref{fig:1period,fig:2period,fig:3period,fig:4period}, this is done for respectively $\theta=0,\pi, 2\pi/3, \pi/2$. The eigenfunction at $\theta = 0$ yields an invariant partition of the state-space, the eigenfunctions of the other frequencies $\theta=\pi, 2\pi/3, \pi/2$ provide periodic partitions of period 2,3 and 4, respectively (see also \cite{levnajic2010ergodic,levnajic2015ergodic}). 

\begin{figure}[h!]
	\begin{center}
		\begin{tabular}{ccc}
			\includegraphics[width=.25\textwidth]{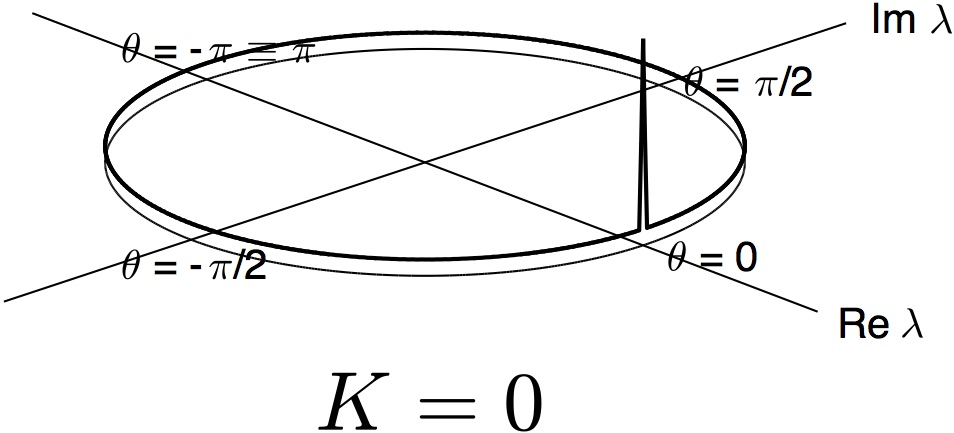} & \includegraphics[width=.25\textwidth]{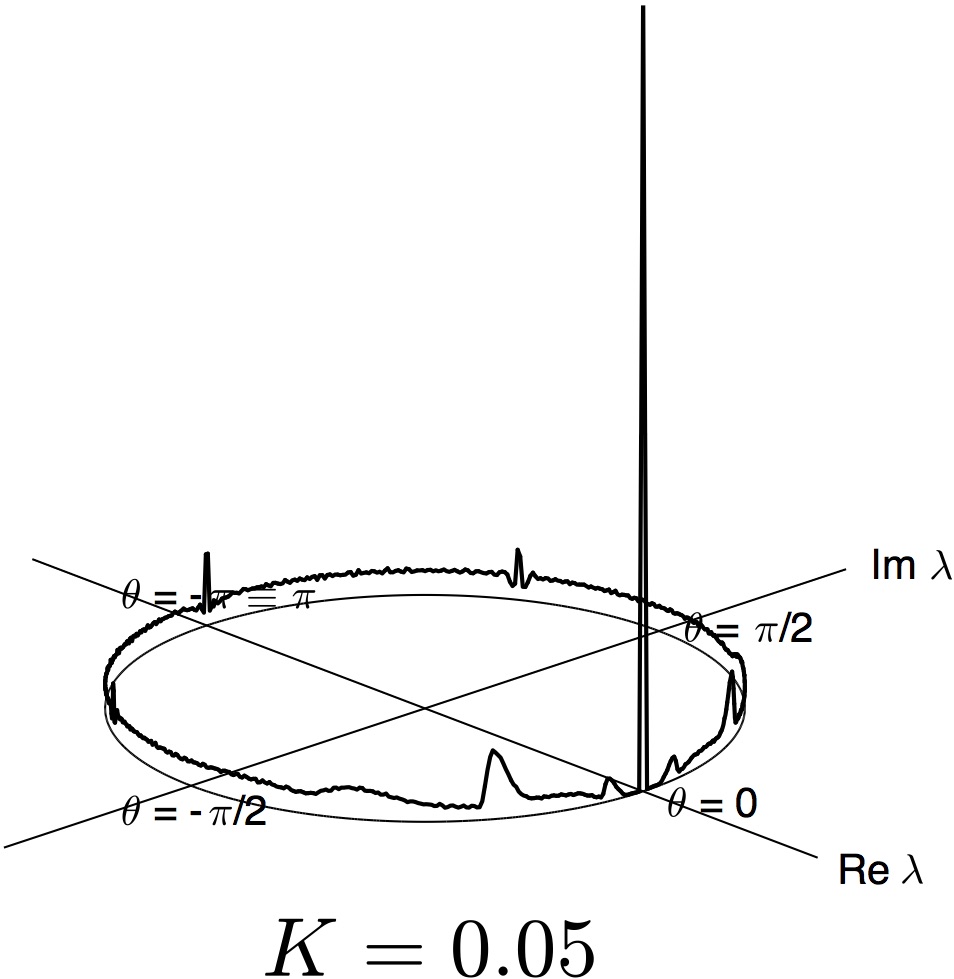} & \includegraphics[width=.25\textwidth]{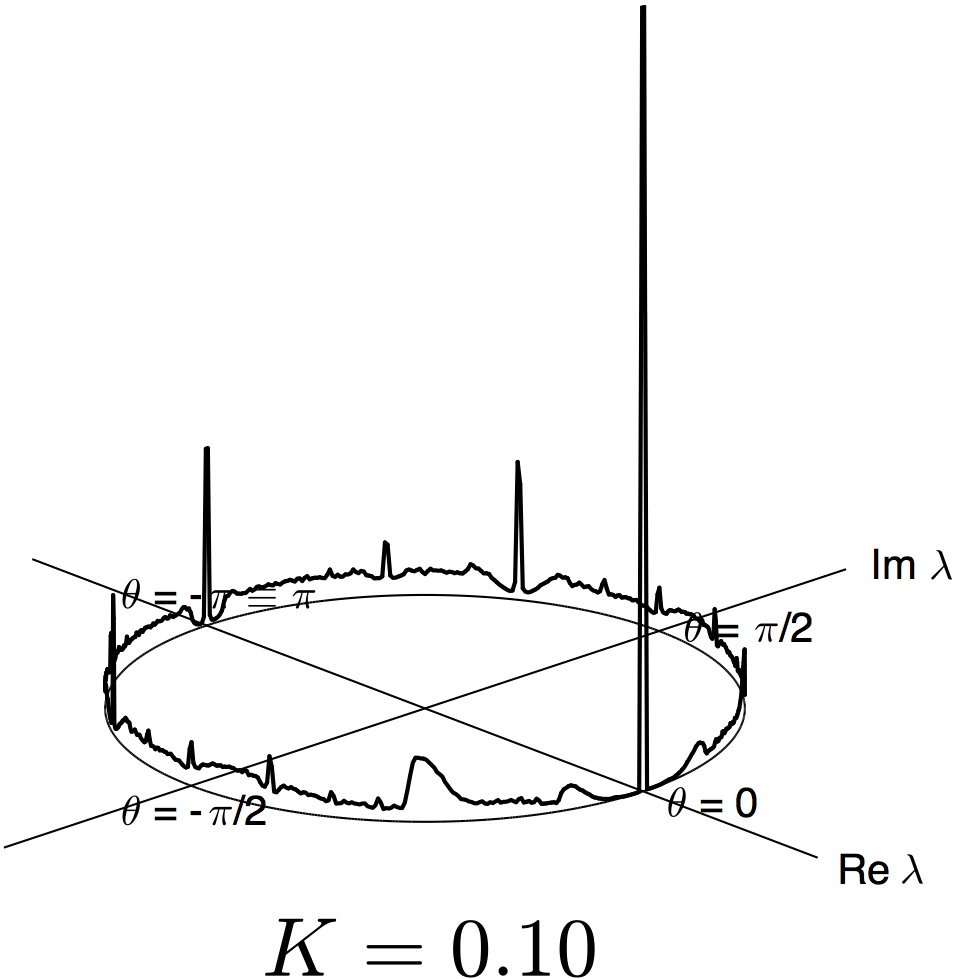} \\
			\includegraphics[width=.25\textwidth]{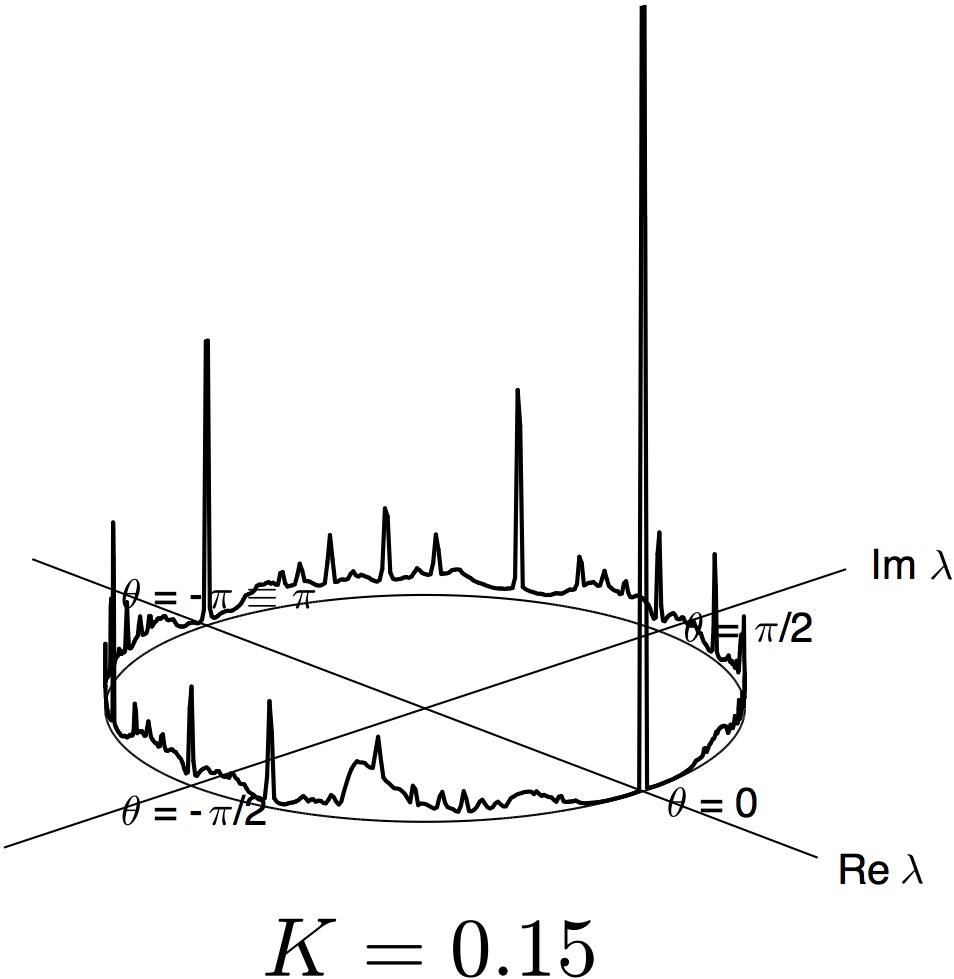} & \includegraphics[width=.25\textwidth]{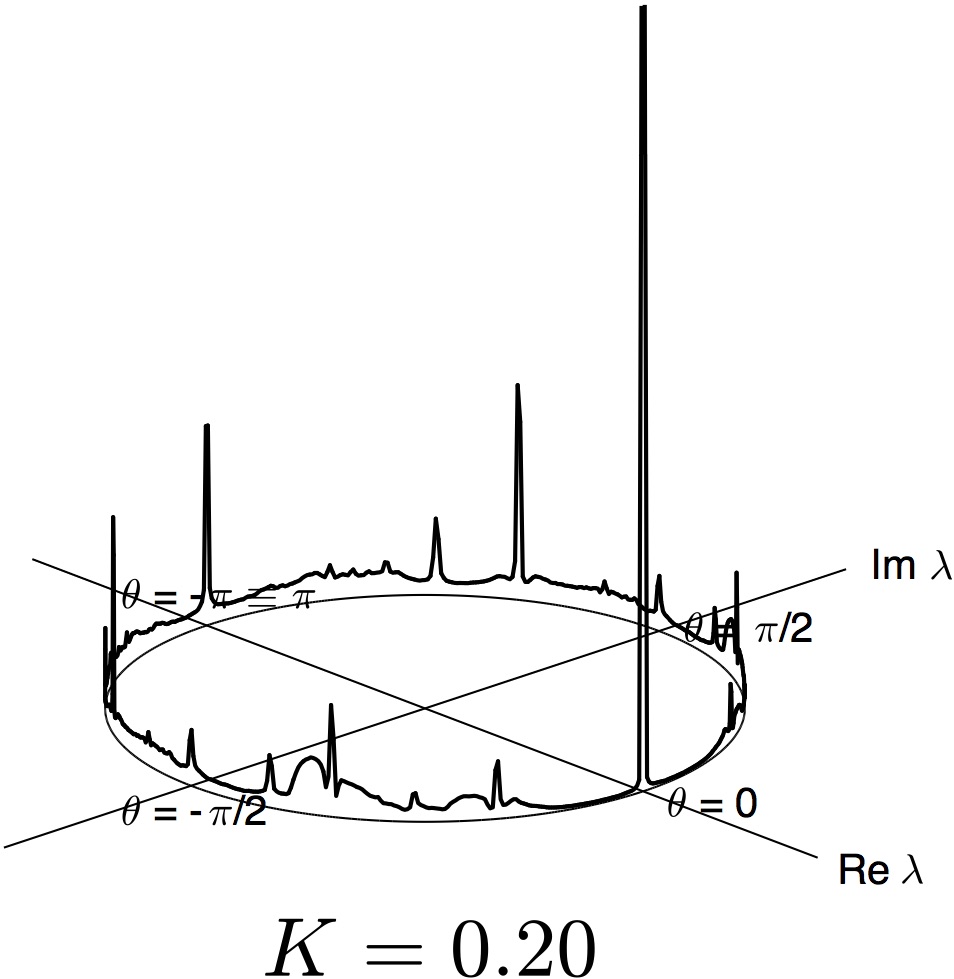} & \includegraphics[width=.25\textwidth]{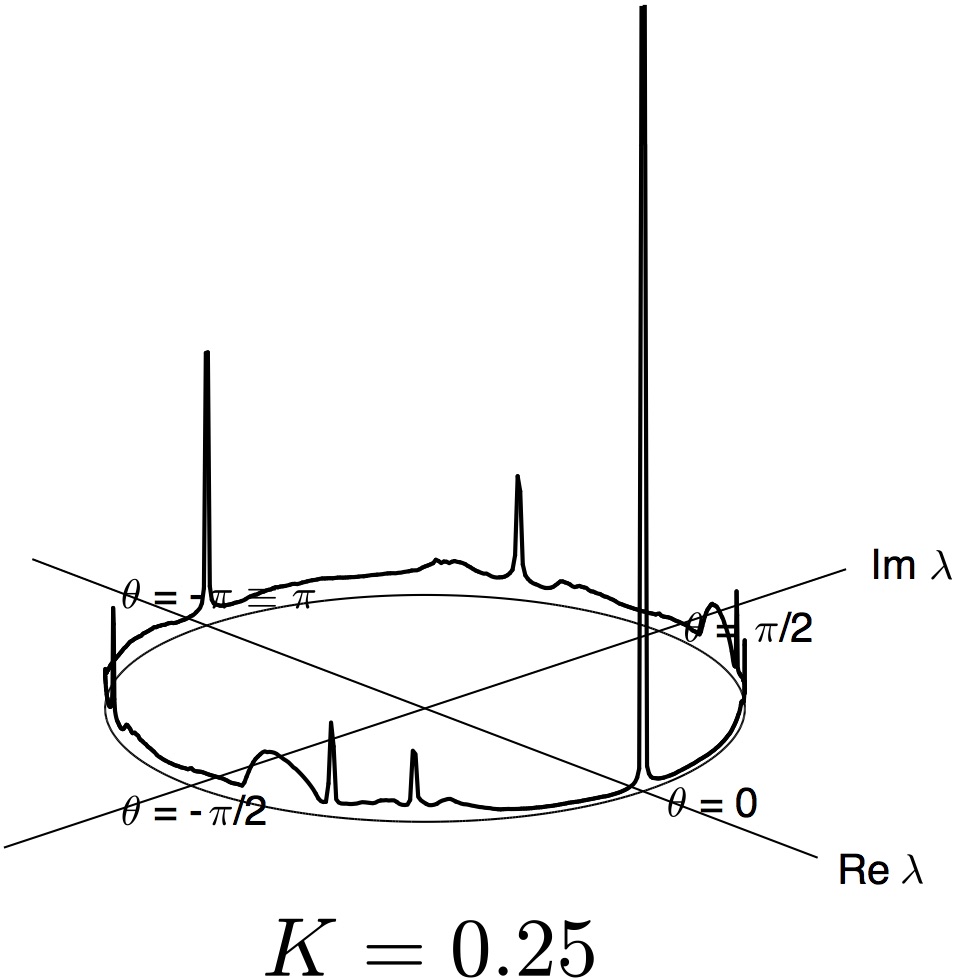} 
		\end{tabular}
	\end{center}
	\caption{Evolution of the spectral density function of the Chirikov map \eqref{eq:chirikovfamily} for the observable \eqref{eq:chirikovobs}.  Results for $\tilde{n}= 2000$ and  $\alpha = 2\pi / 500$.} \label{fig:chirikovdensity}
\end{figure}

\begin{figure}[h!]
	\begin{center}
		\begin{tabular}{cc}\includegraphics[width=.45\textwidth]{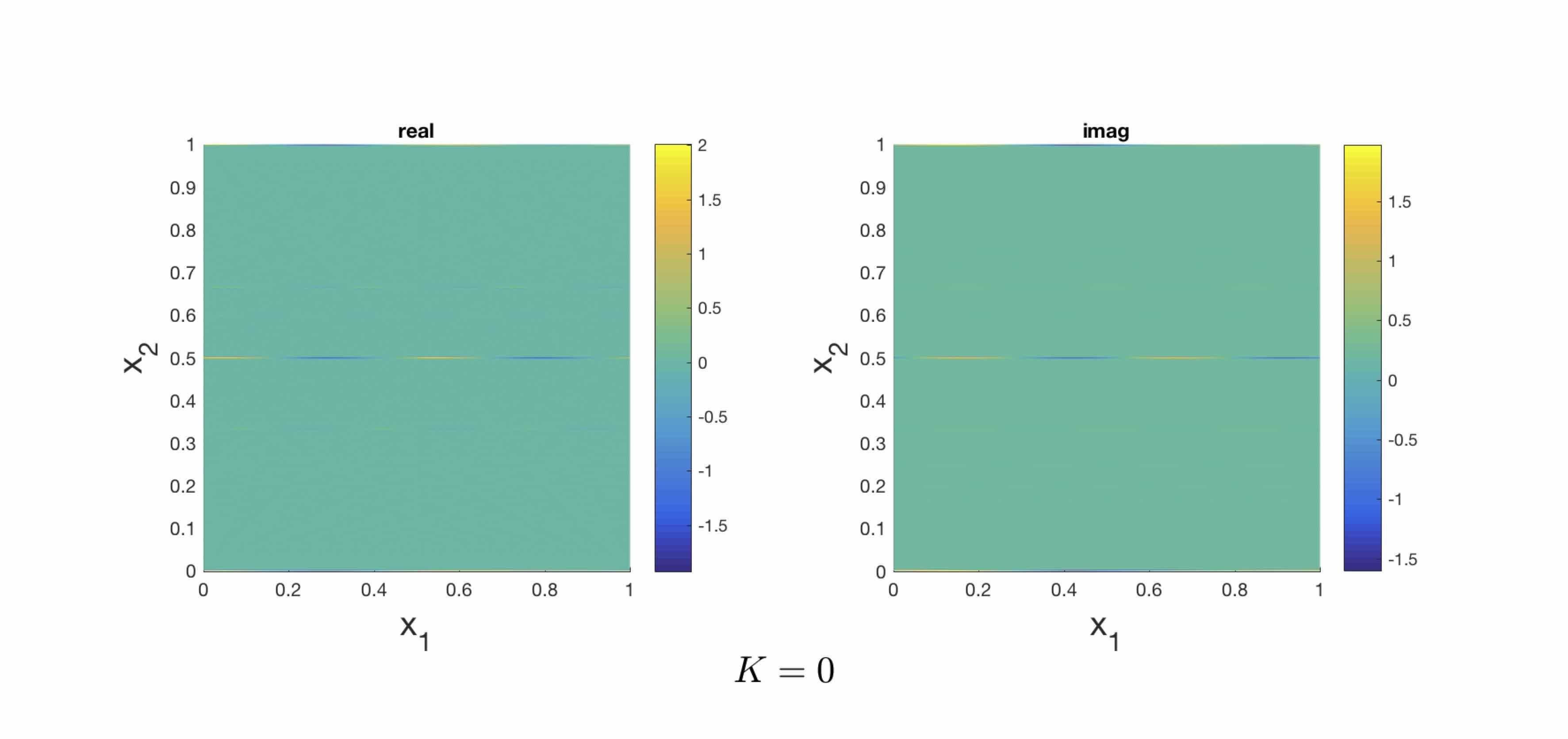} & \includegraphics[width=.45\textwidth]{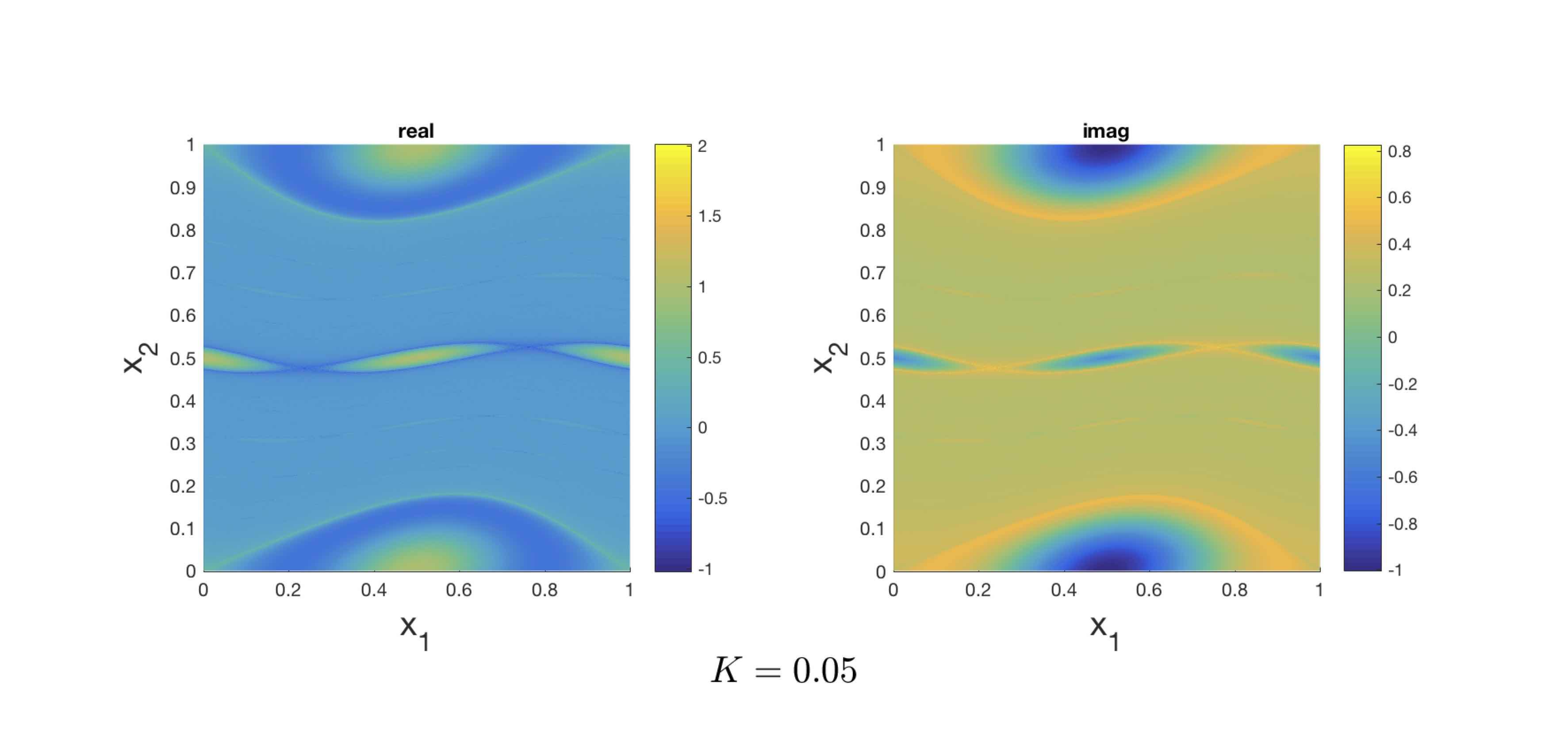} \\
			\includegraphics[width=.45\textwidth]{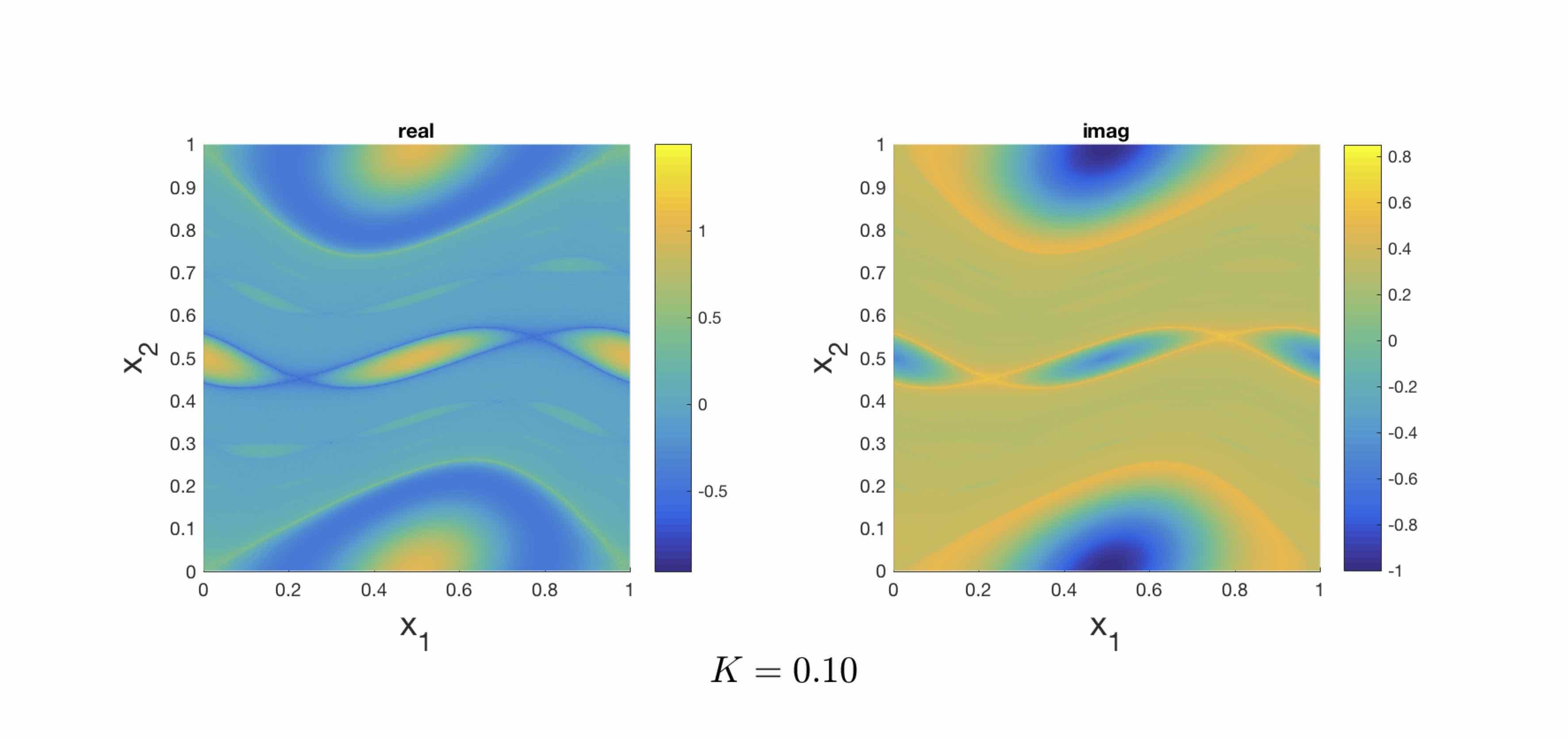} & \includegraphics[width=.45\textwidth]{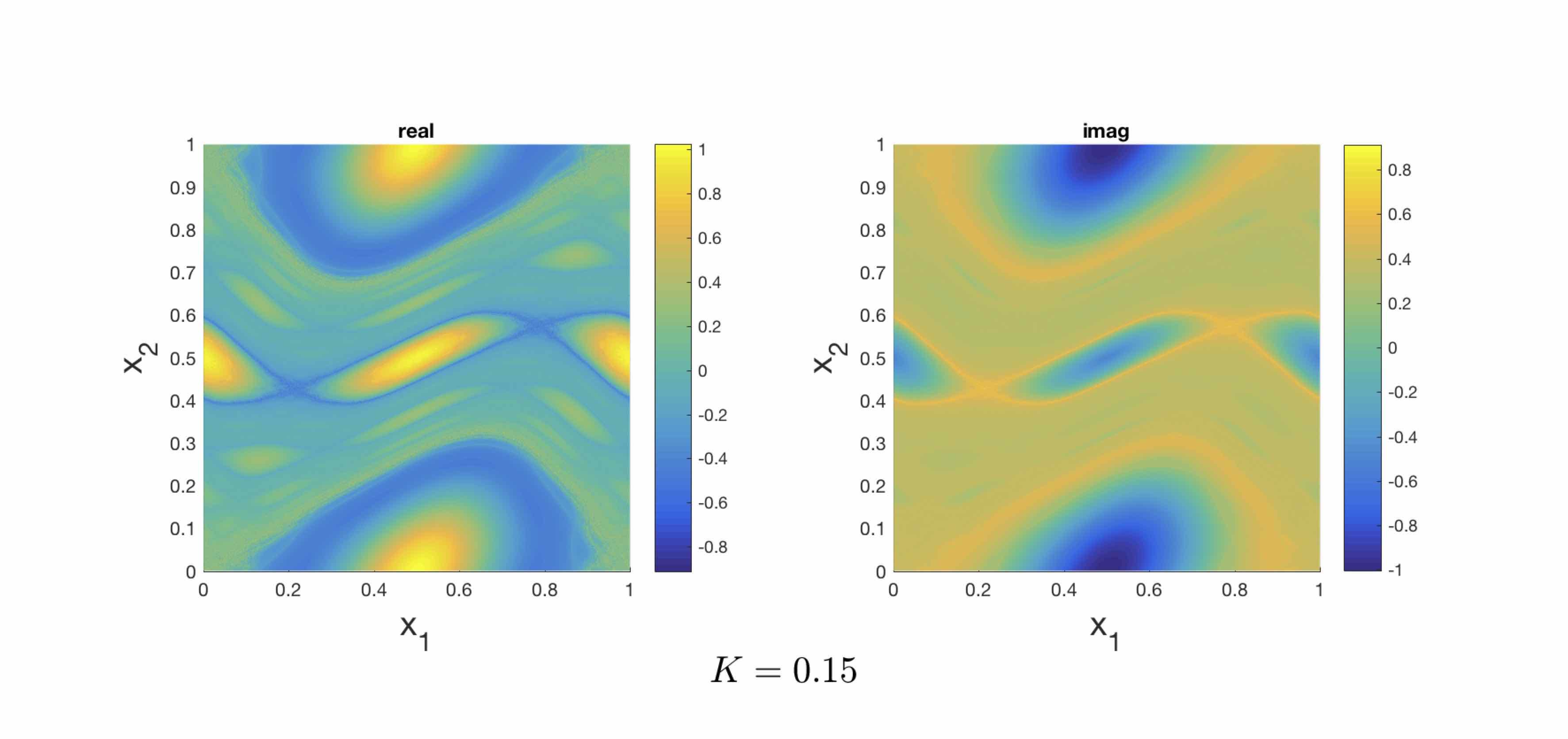}  \\
			\includegraphics[width=.45\textwidth]{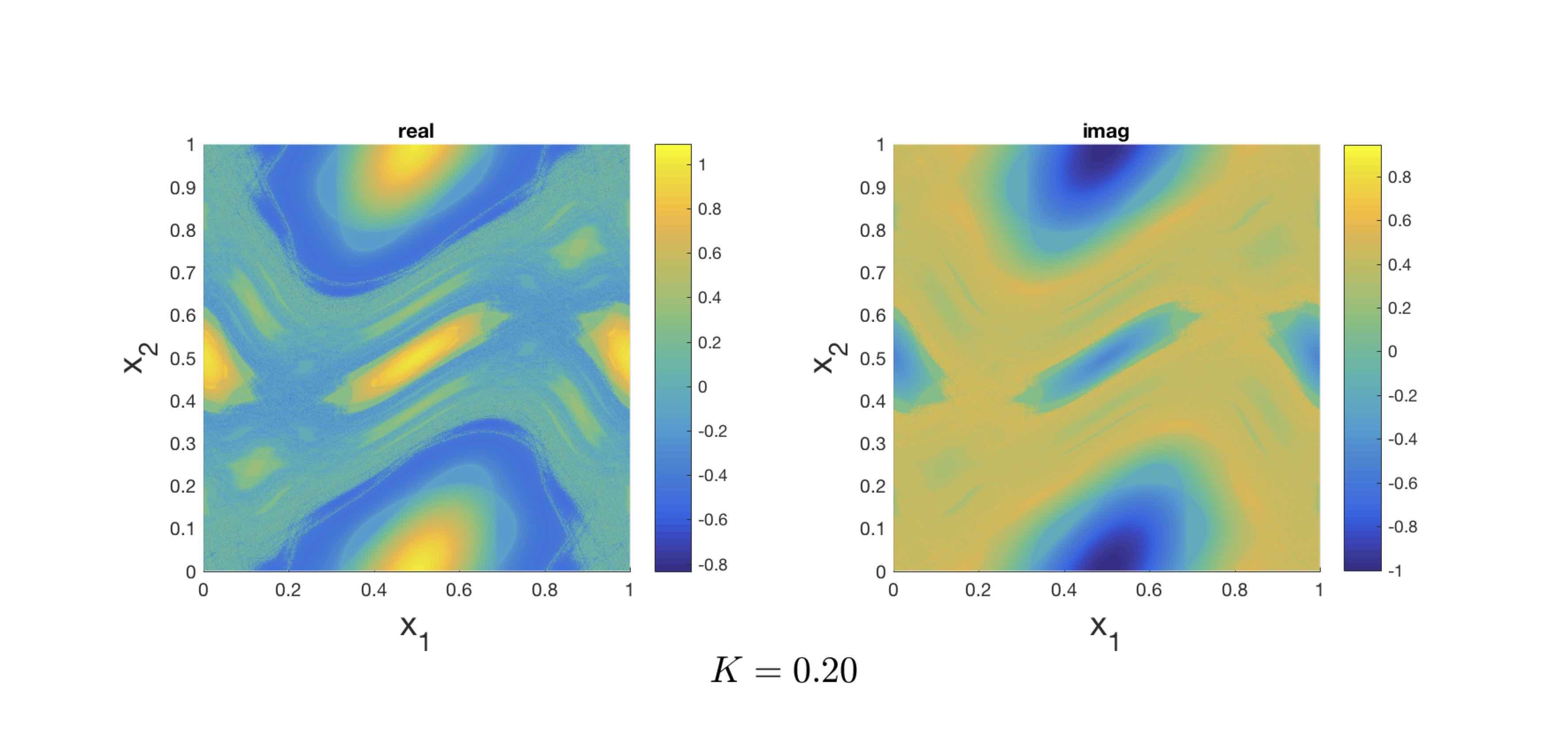}  & \includegraphics[width=.45\textwidth]{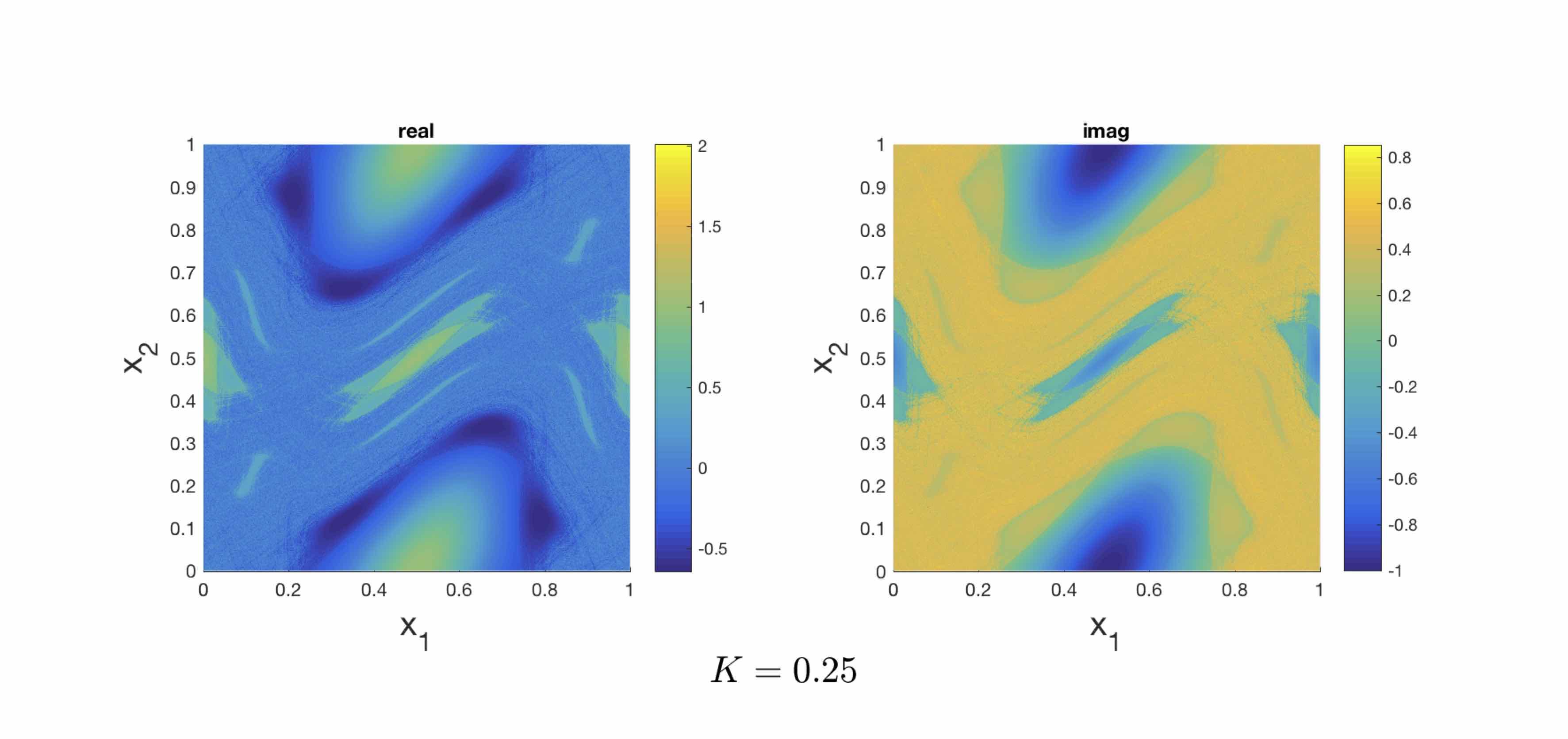} 	  \end{tabular}    
	\end{center}
	\caption{Spectral projections computed for the Chirikov map \eqref{eq:chirikovfamily} at the interval $D=[-0.02, 0.02]$ with $\tilde{n}=2000$.  The depicted projection approximates the eigenfuncions at $\theta=0$, which generates an invariant partition of the state-space.} \label{fig:1period}
	\begin{center}
		\begin{tabular}{cc}\includegraphics[width=.45\textwidth]{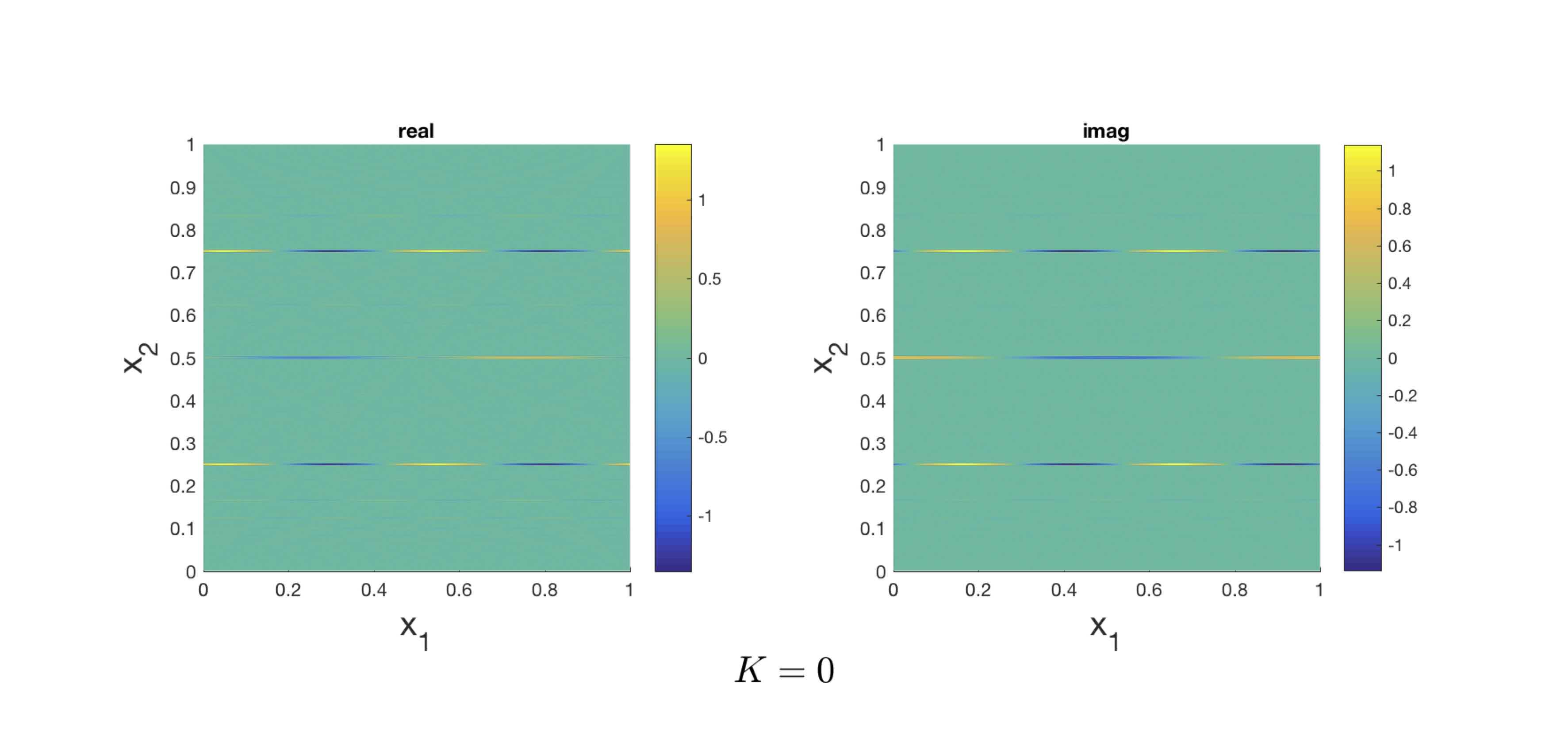} & \includegraphics[width=.45\textwidth]{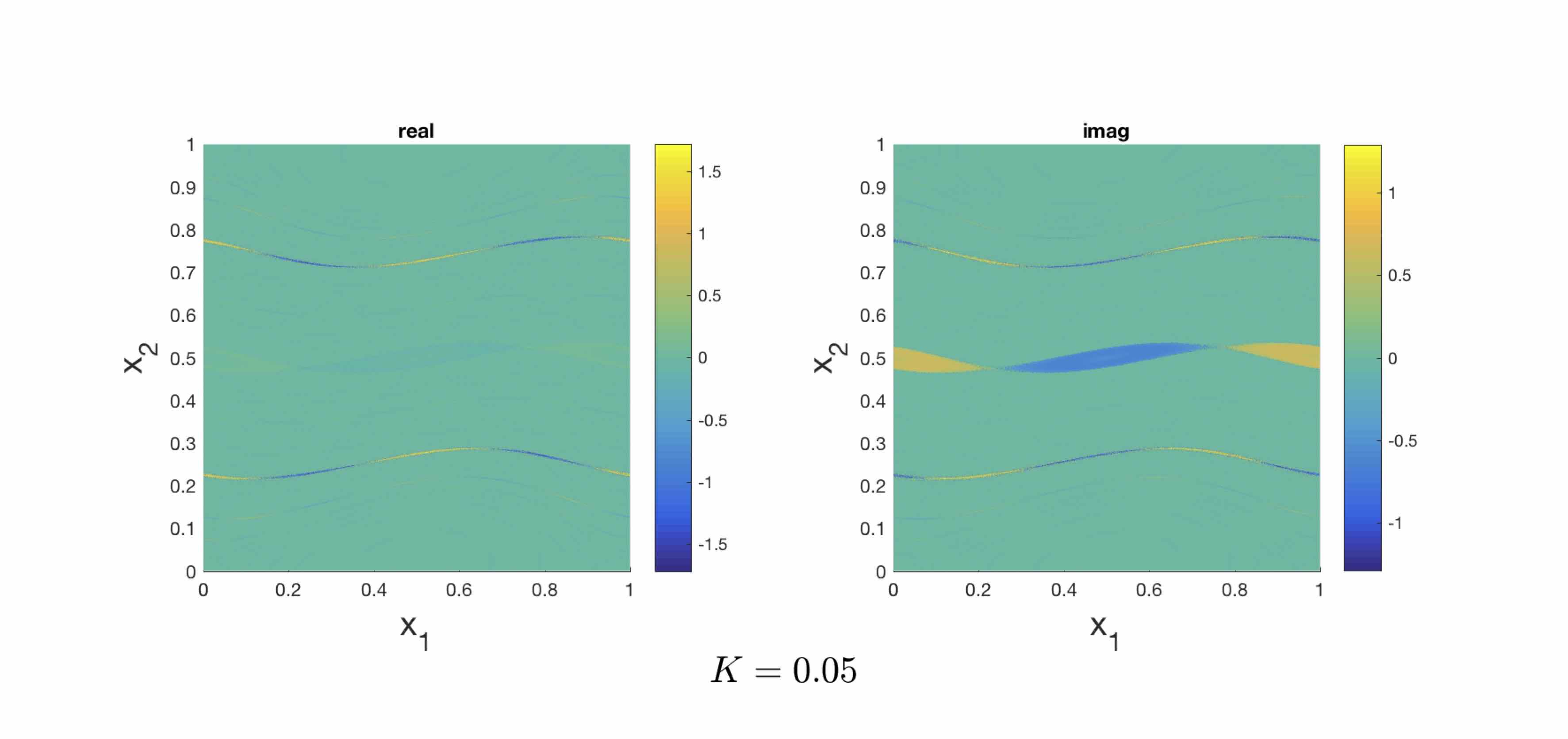} \\
			\includegraphics[width=.45\textwidth]{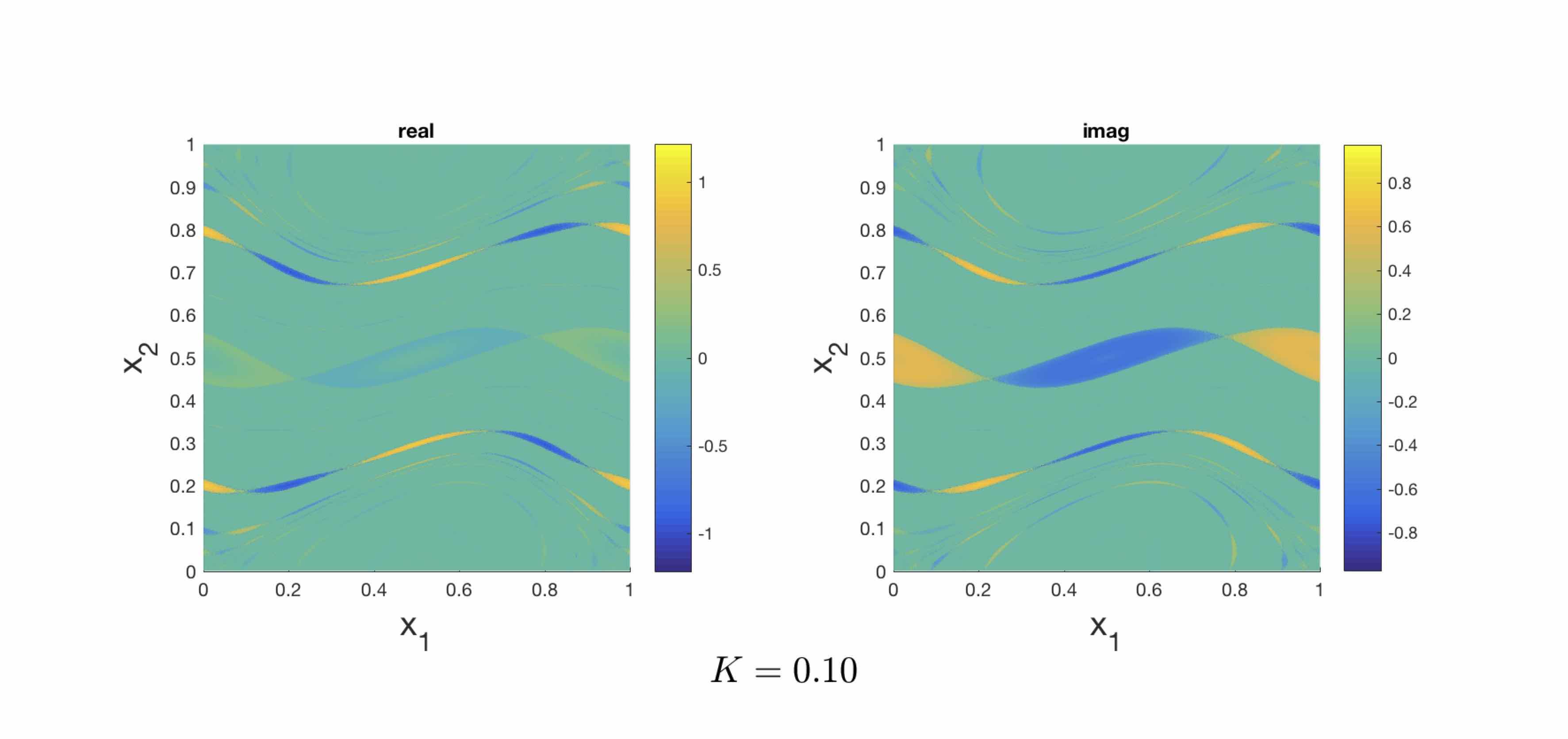} & \includegraphics[width=.45\textwidth]{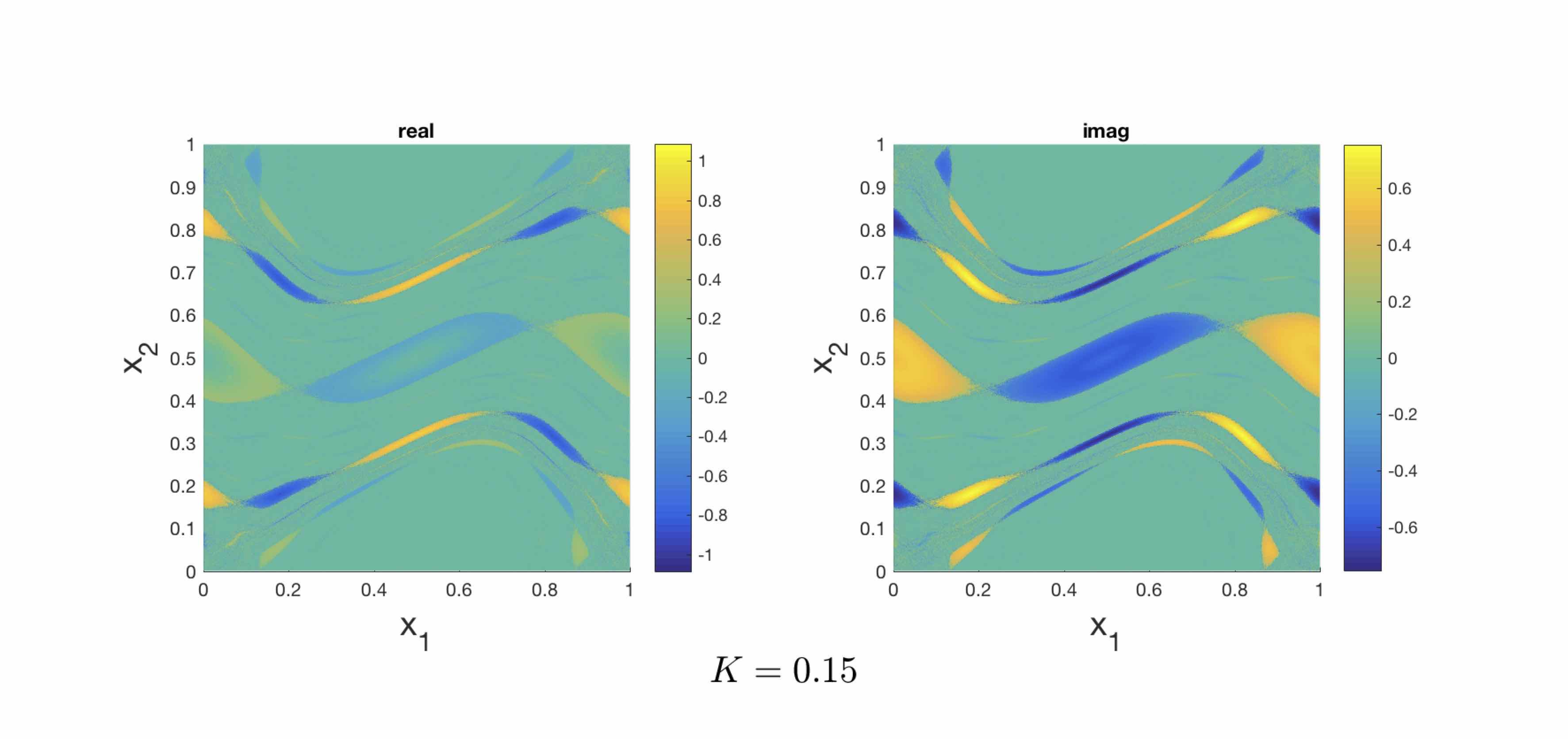}  \\
			\includegraphics[width=.45\textwidth]{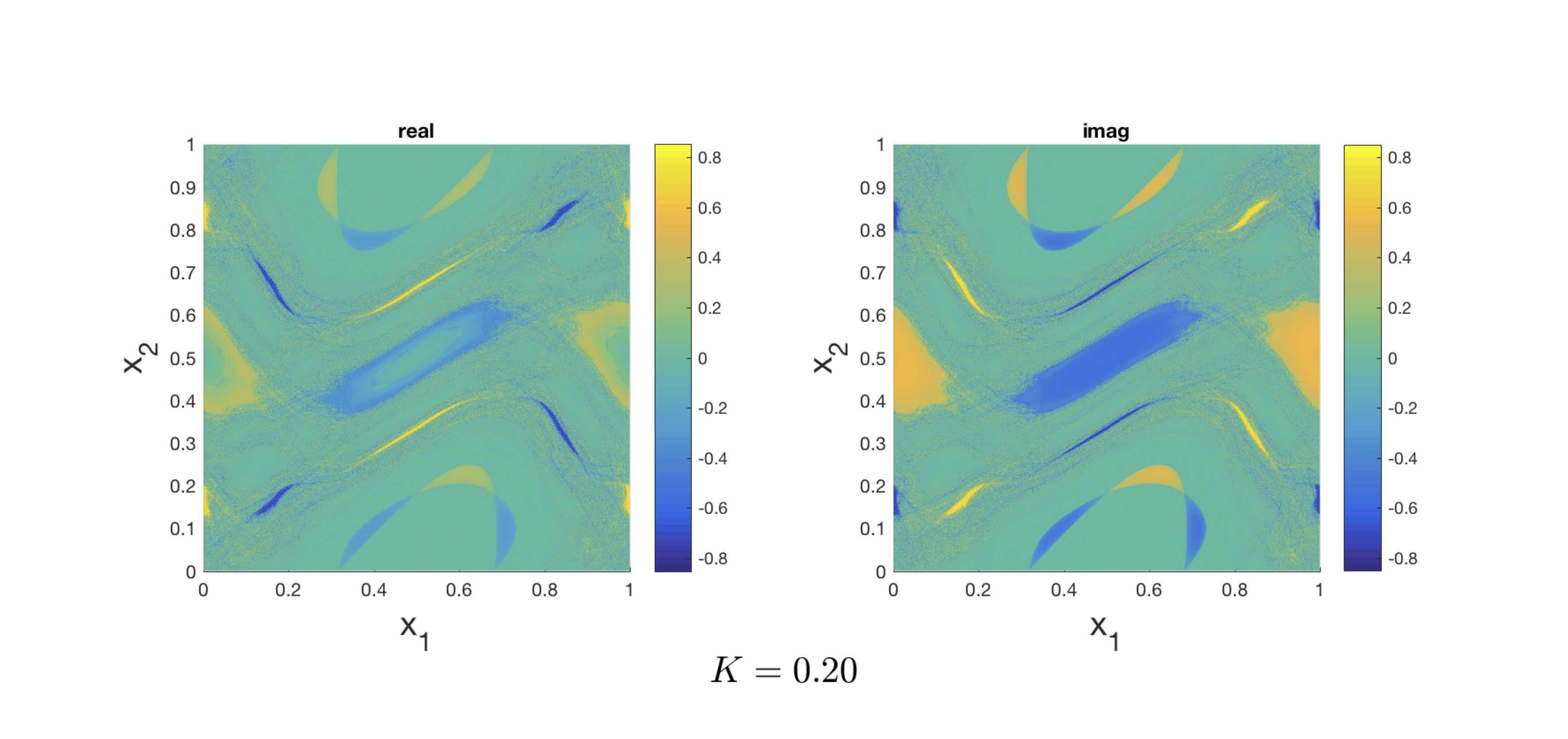}  & \includegraphics[width=.45\textwidth]{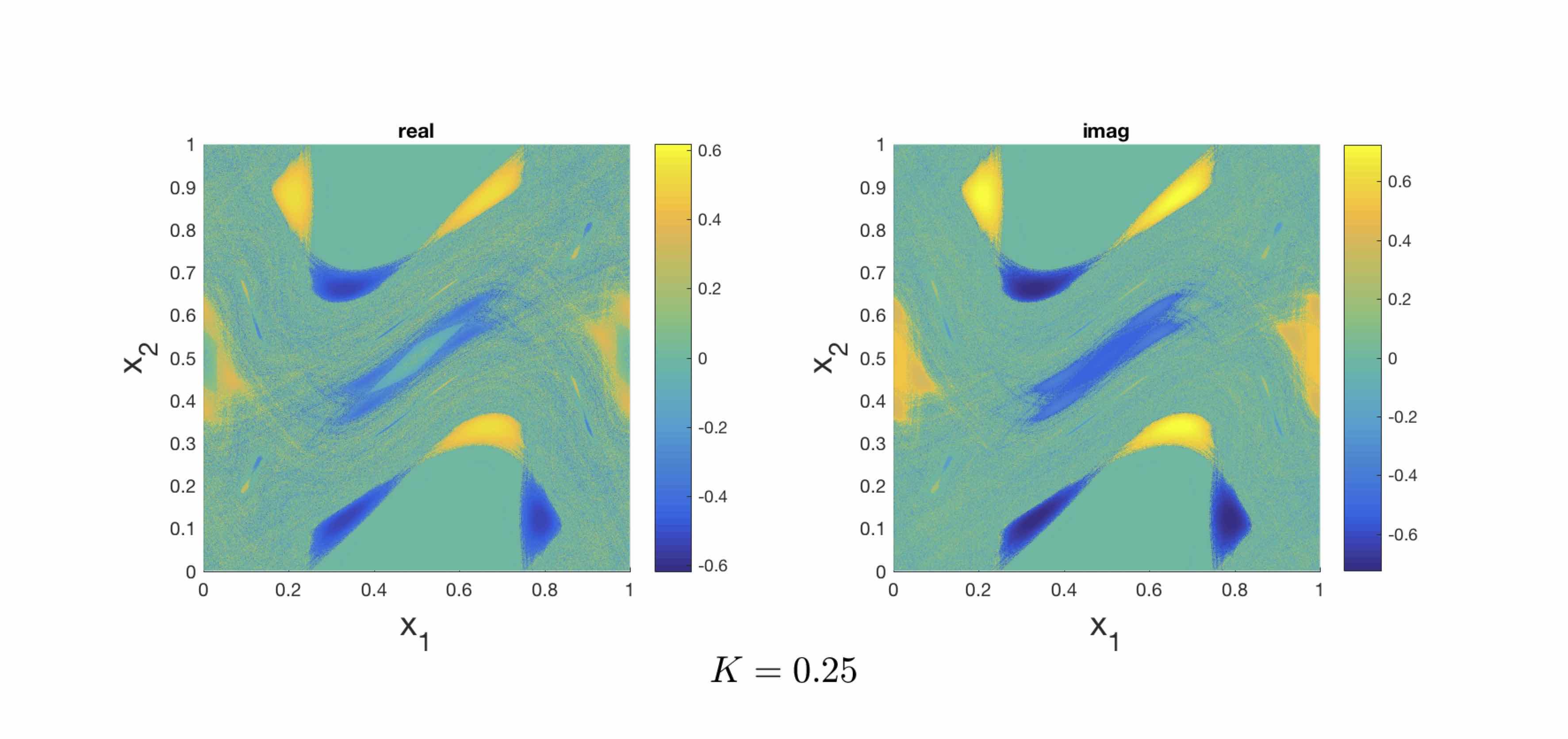} 	  \end{tabular}    
	\end{center}
	\caption{Spectral projections computed for the Chirikov map \eqref{eq:chirikovfamily} at the interval $D=[\pi-0.02, \pi+0.02]$ with $\tilde{n}=2000$.  The depicted projection approximates the eigenfuncions at $\theta=\pi$, which generates an period-2 partition of the state-space.} \label{fig:2period}
\end{figure}

\begin{figure}[h!]
	\begin{center}
		\begin{tabular}{cc}\includegraphics[width=.45\textwidth]{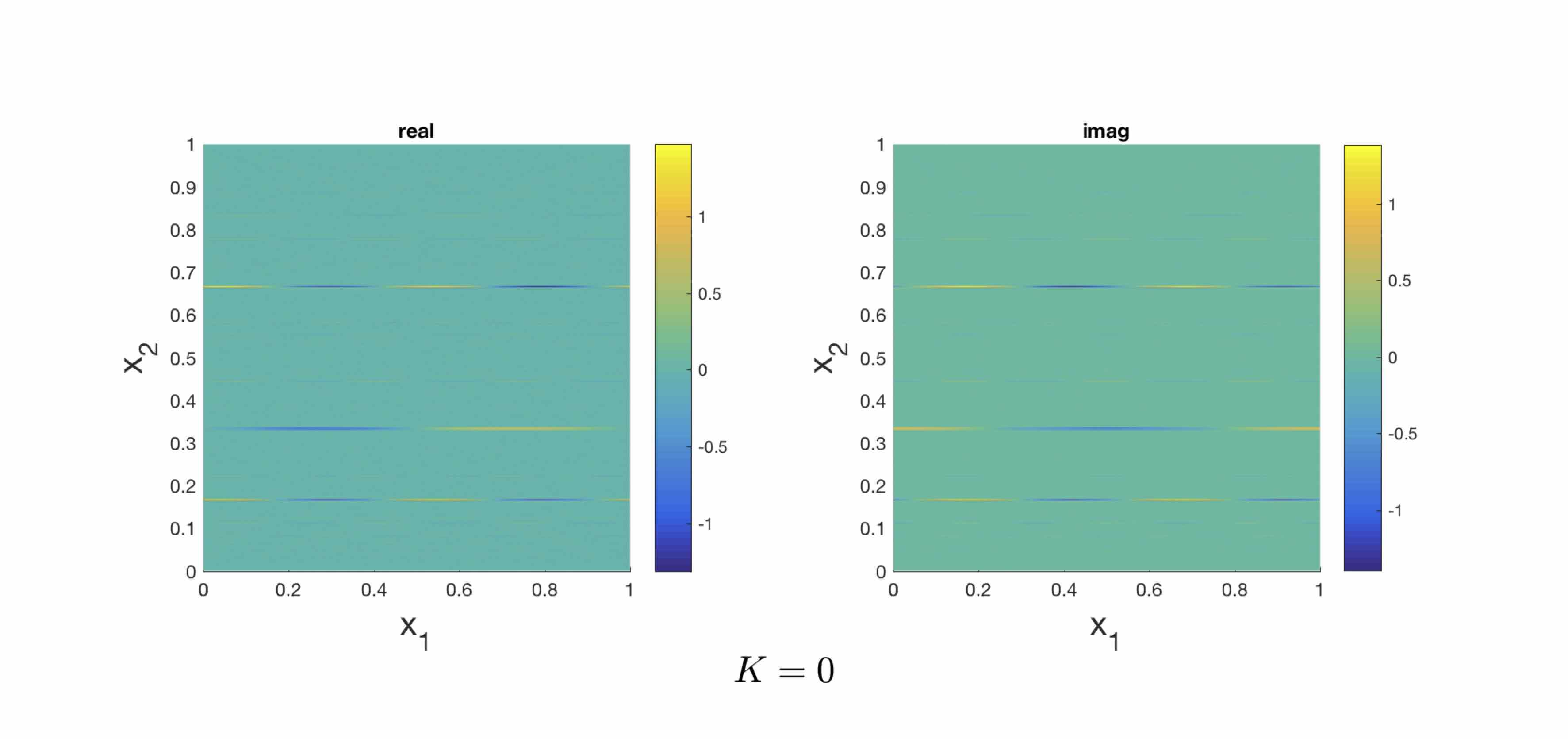} & \includegraphics[width=.45\textwidth]{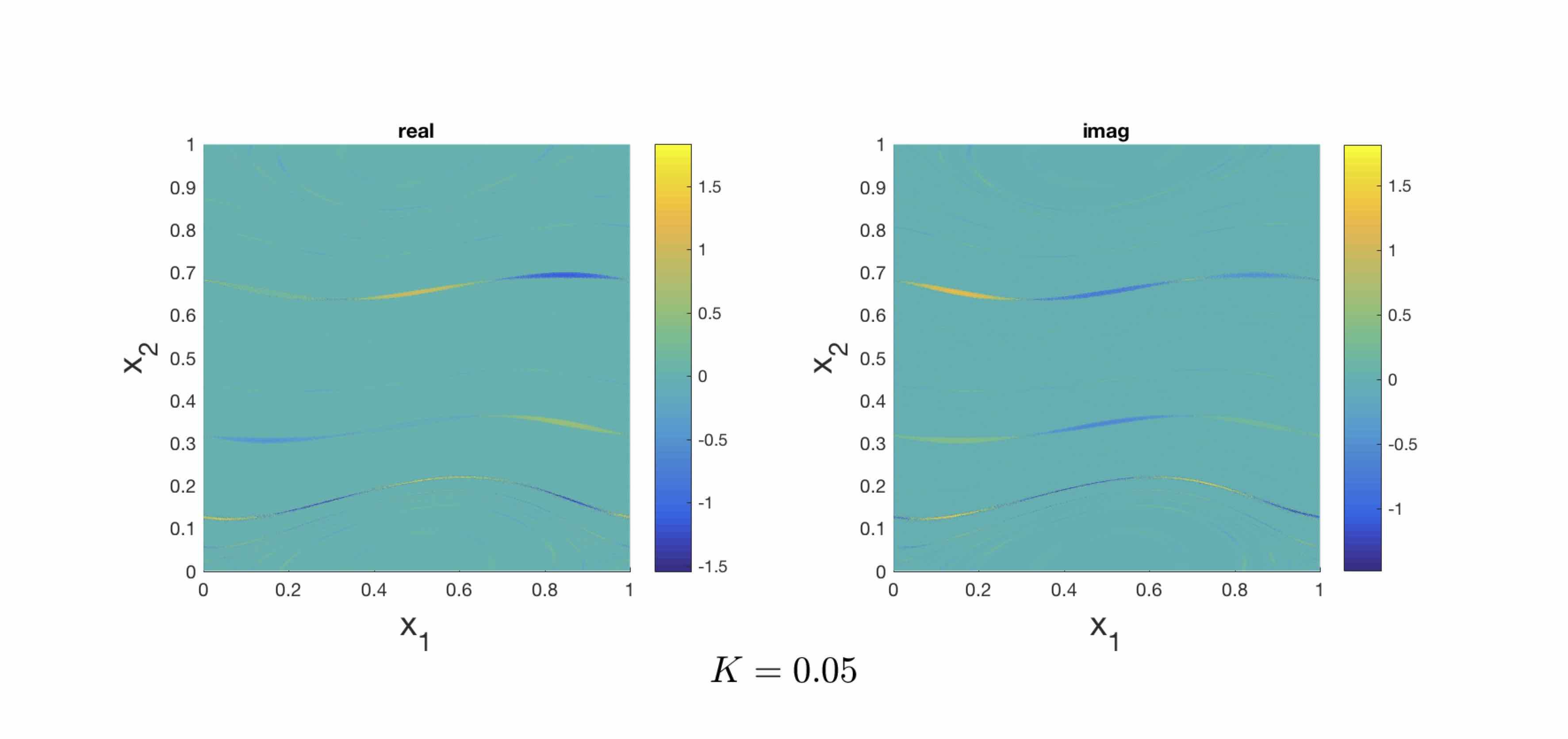} \\
			\includegraphics[width=.45\textwidth]{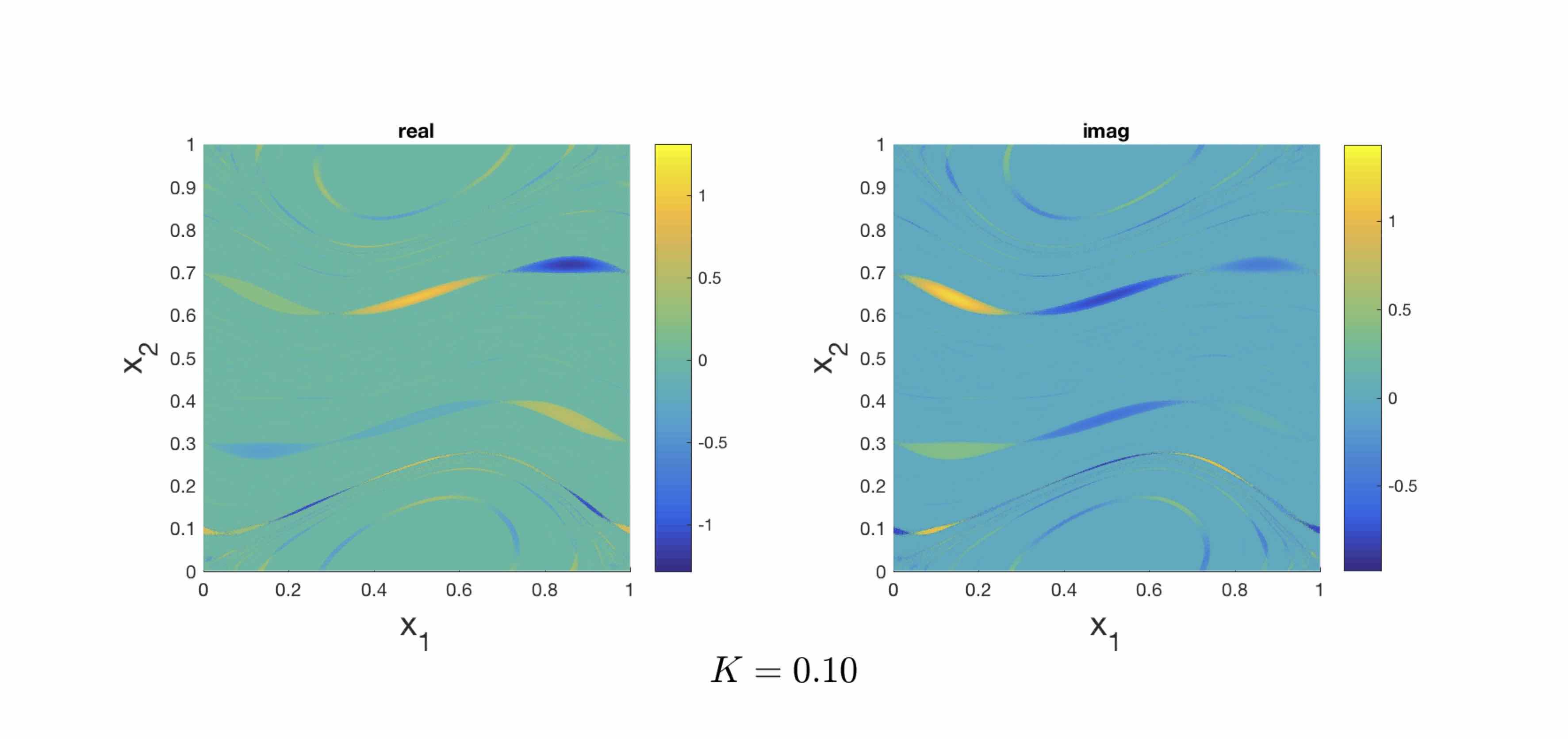} & \includegraphics[width=.45\textwidth]{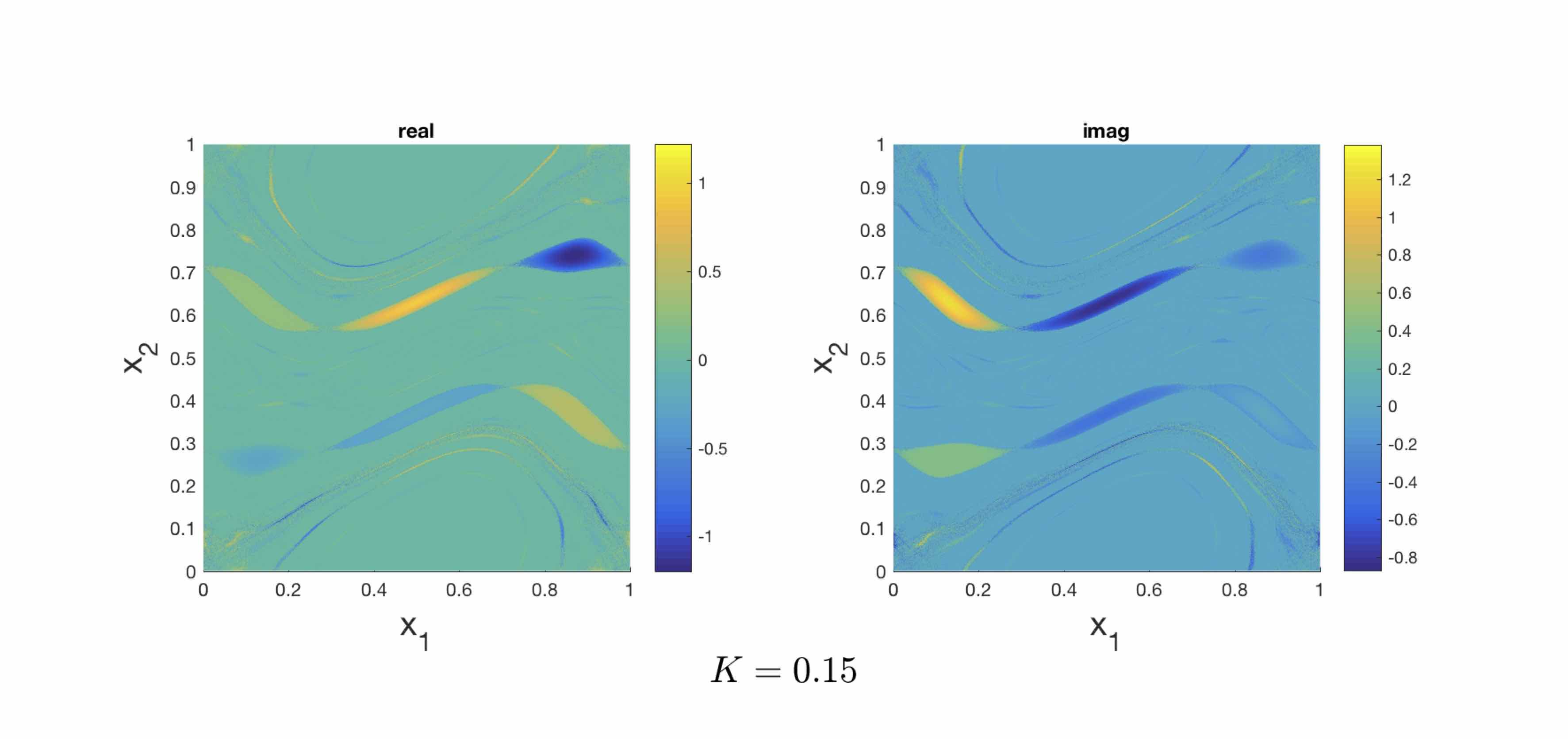}  \\
			\includegraphics[width=.45\textwidth]{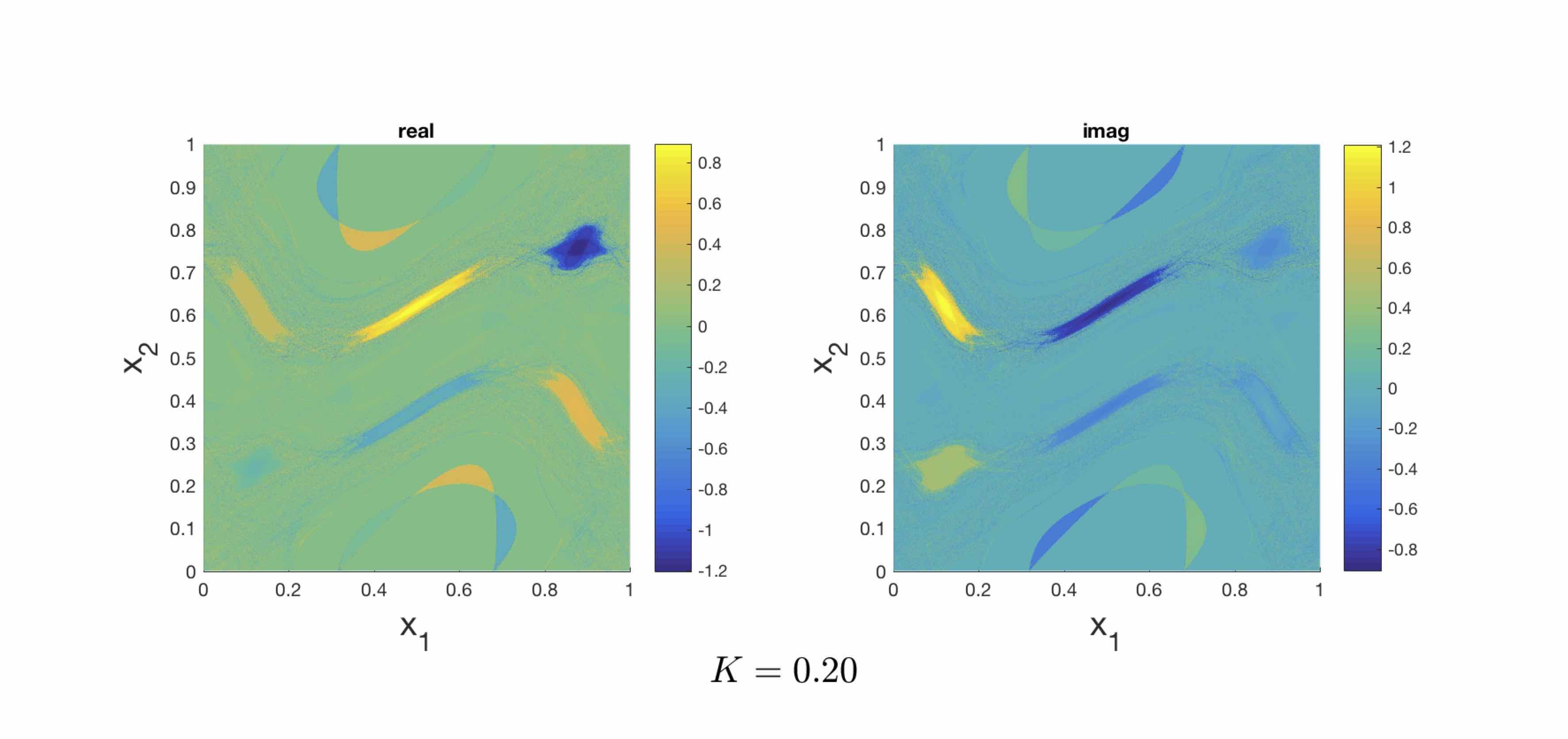}  & \includegraphics[width=.45\textwidth]{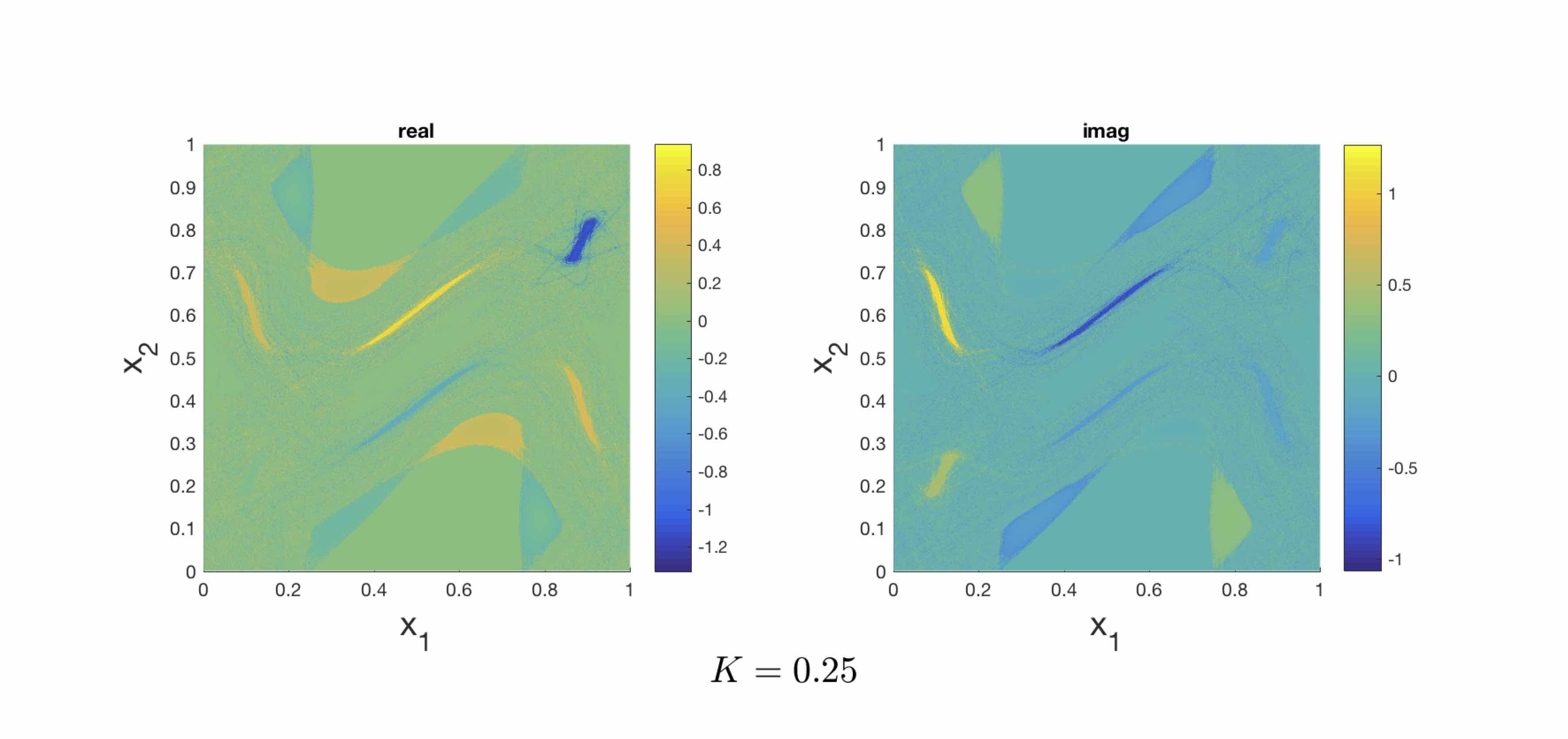} 	  \end{tabular}    
	\end{center}
	\caption{Spectral projections computed for the Chirikov map \eqref{eq:chirikovfamily} at the interval $D=[2\pi/3-0.02, 2\pi/3+0.02]$ with $\tilde{n}=2000$.  The depicted projection approximates the eigenfuncions at $\theta=2\pi/3$, which generates an period-3 partition of the state-space.} \label{fig:3period}
	\begin{center}
		\begin{tabular}{cc}\includegraphics[width=.45\textwidth]{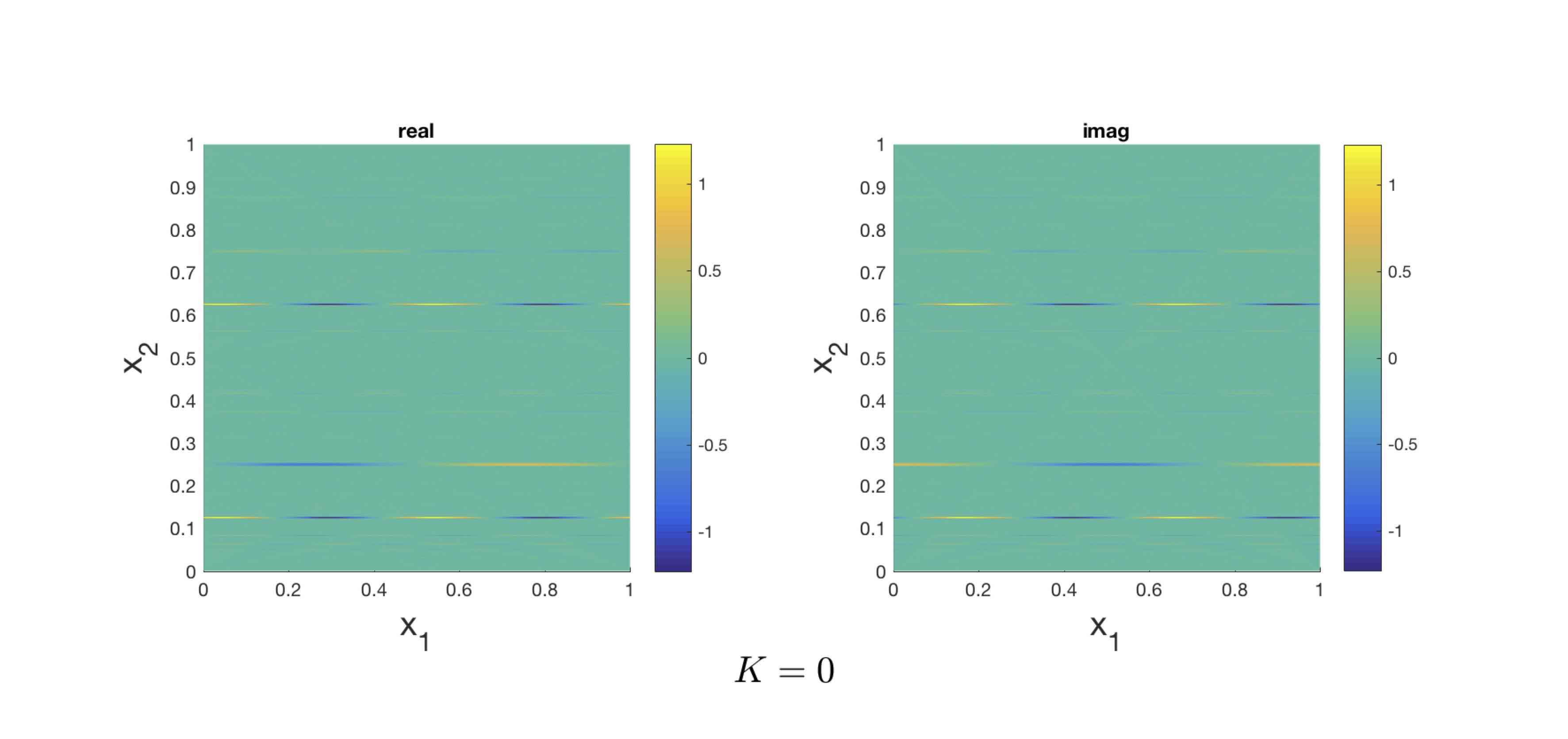} & \includegraphics[width=.45\textwidth]{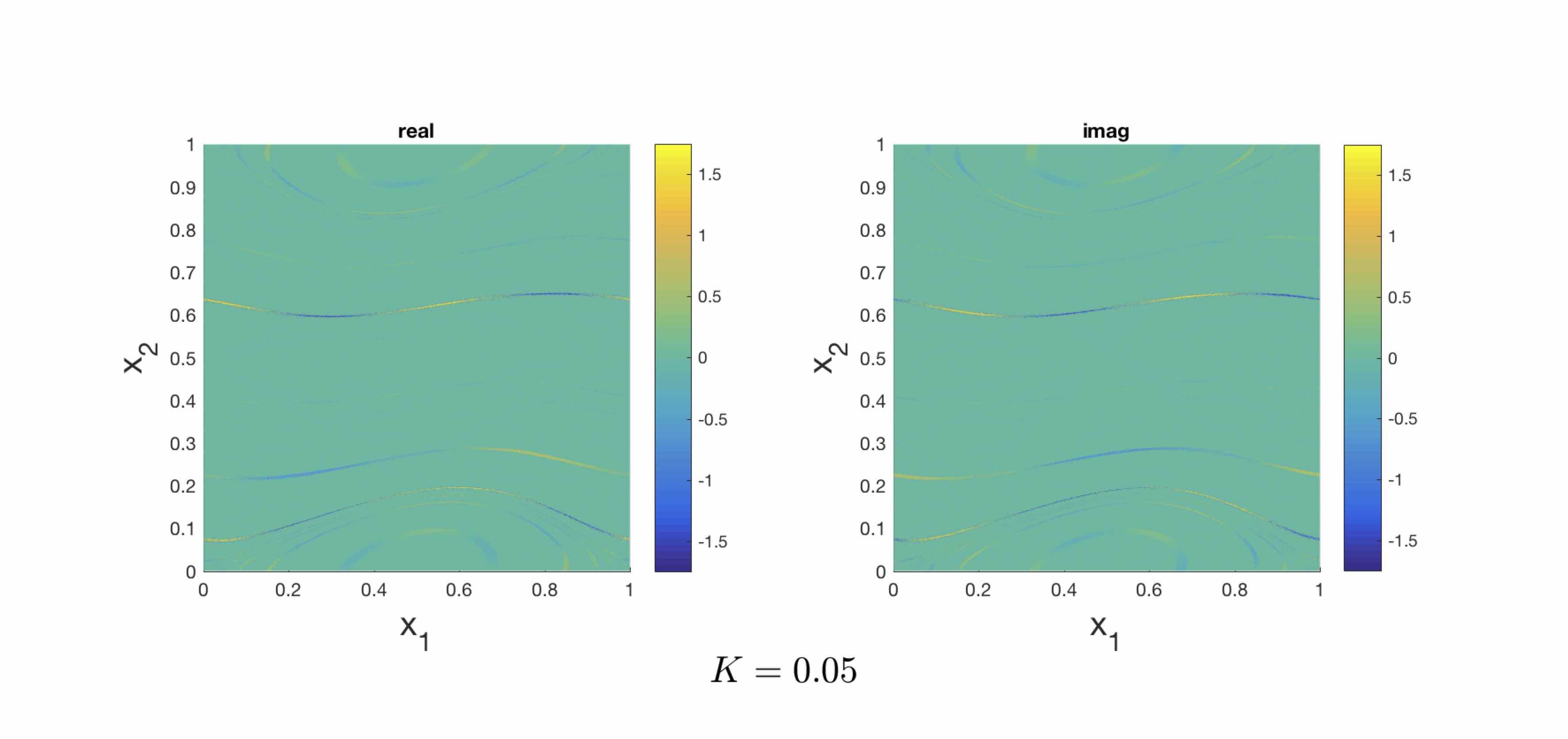} \\
			\includegraphics[width=.45\textwidth]{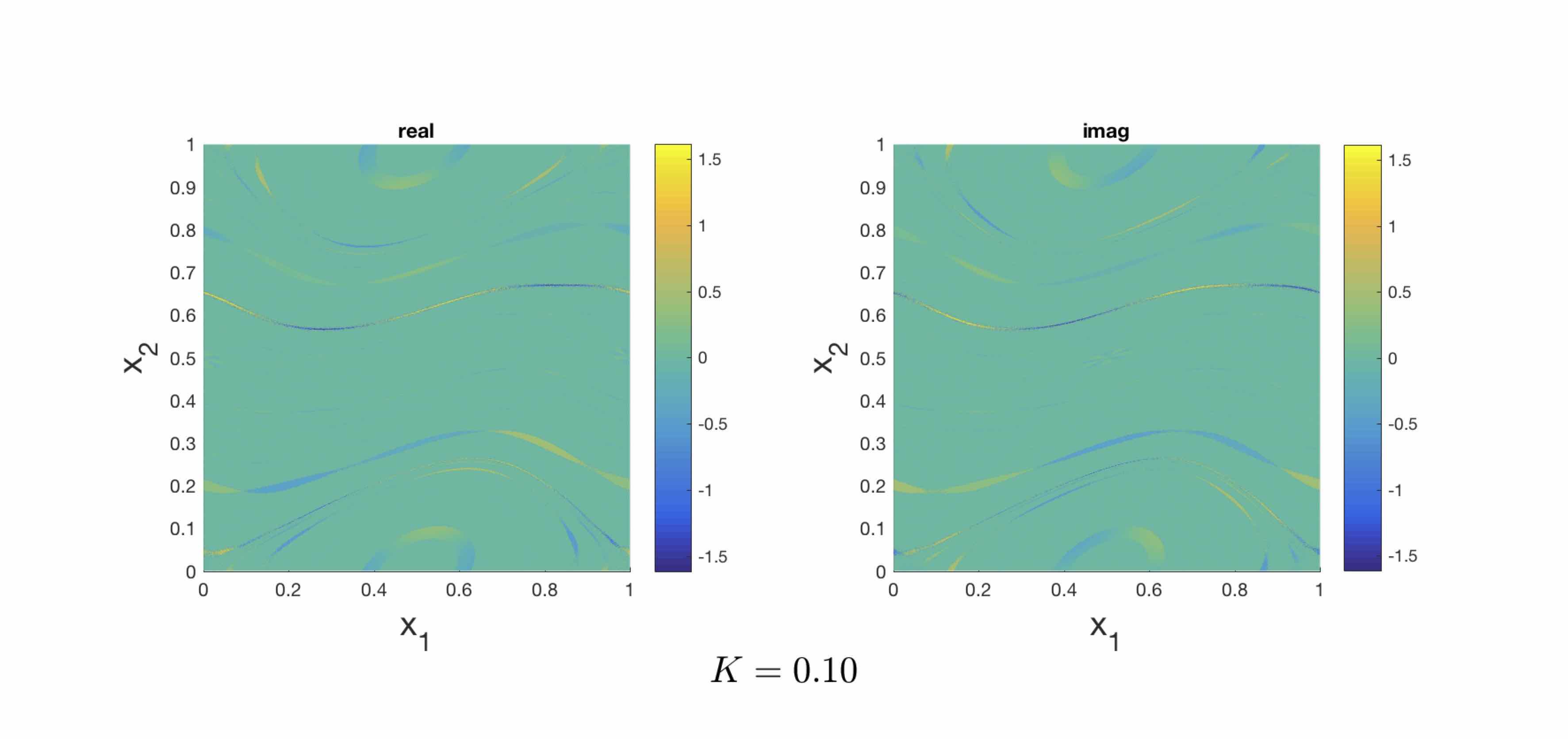} & \includegraphics[width=.45\textwidth]{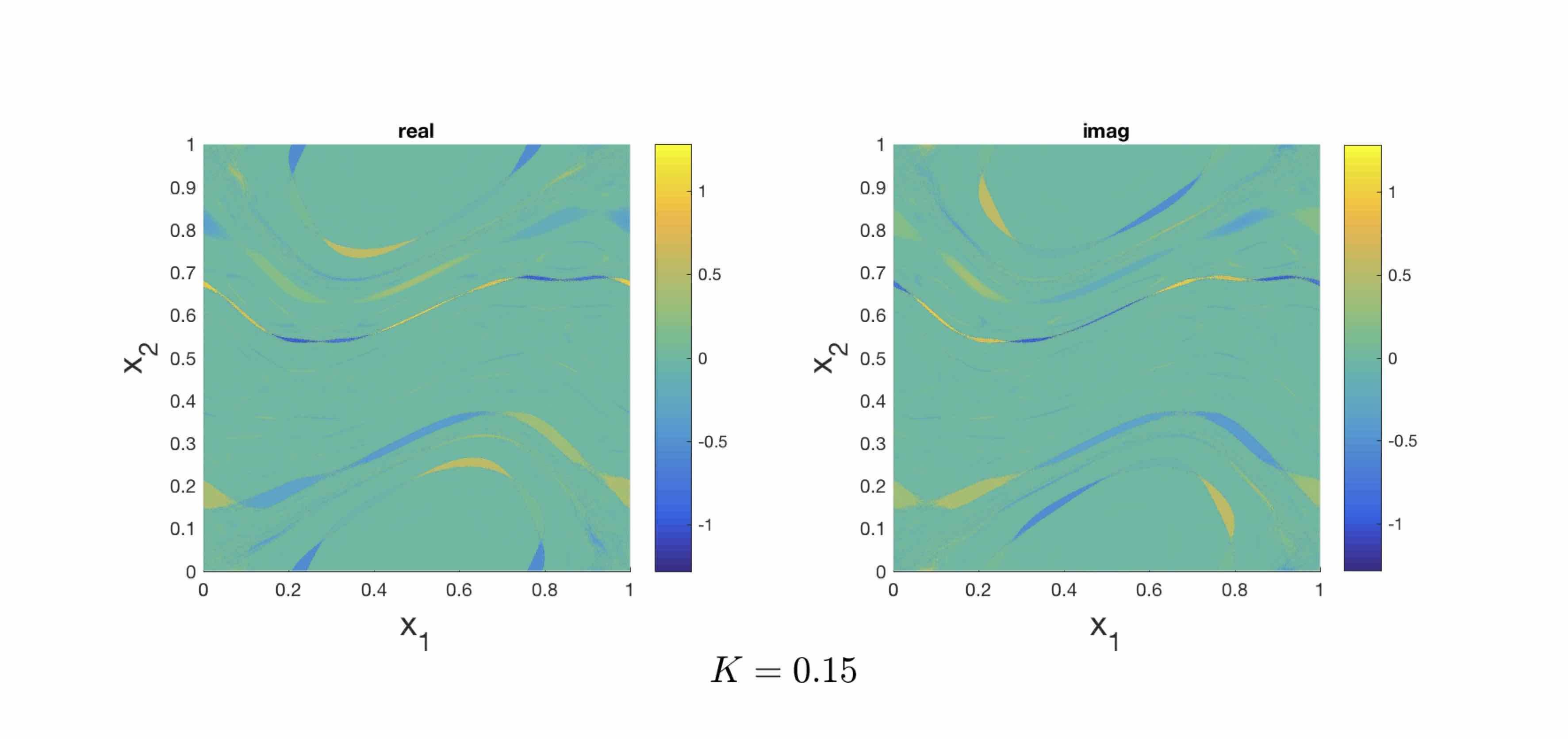}  \\
			\includegraphics[width=.45\textwidth]{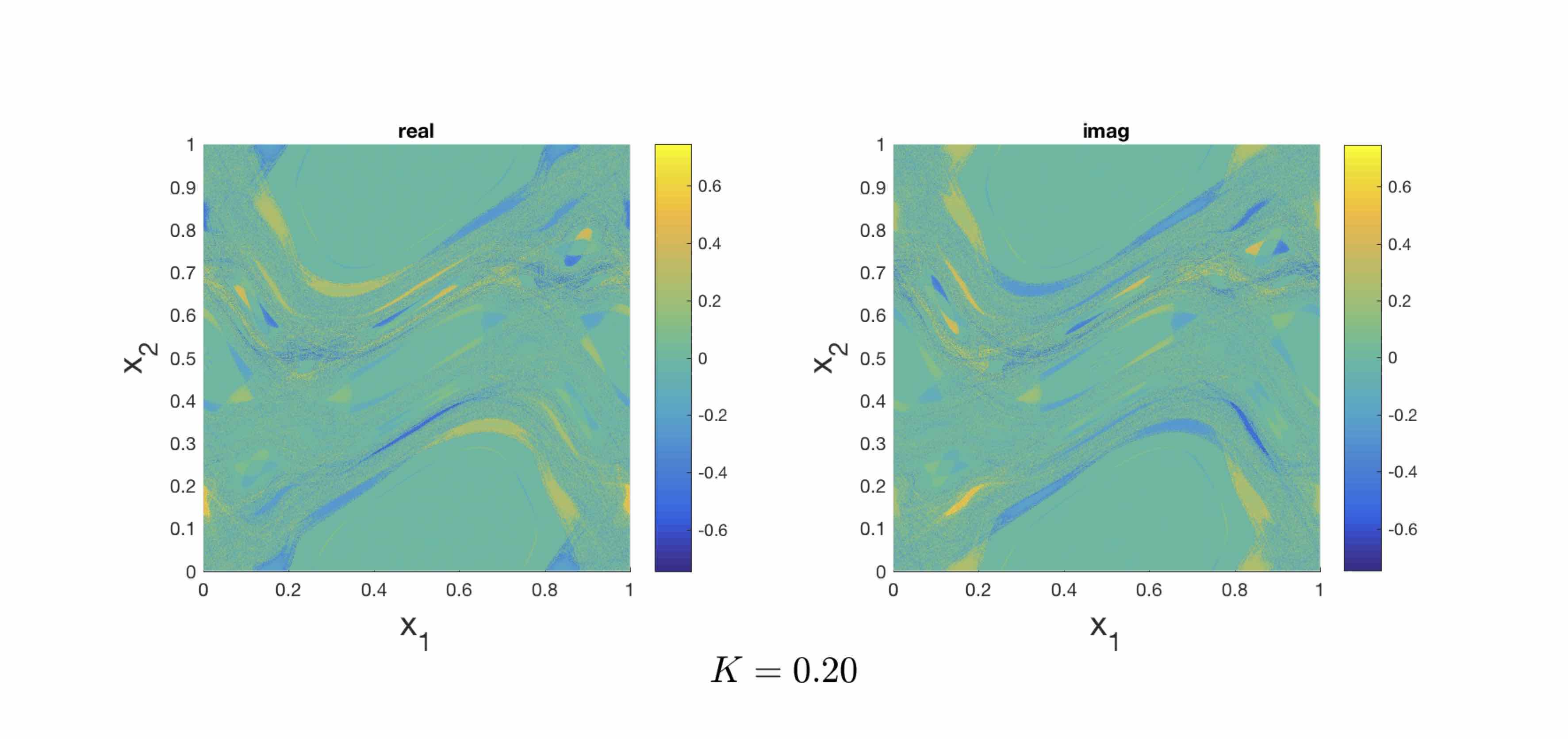}  & \includegraphics[width=.45\textwidth]{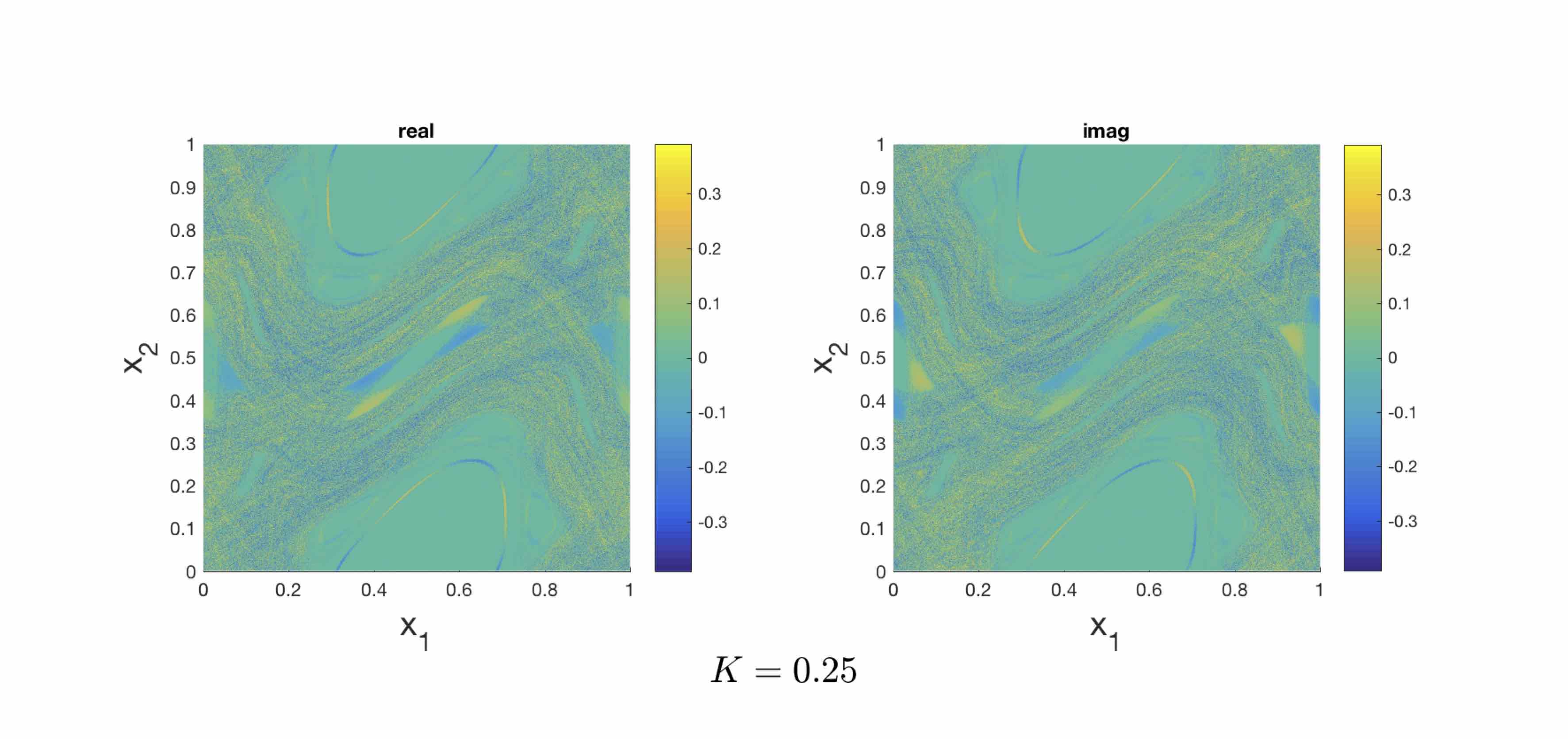} 	  \end{tabular}    
	\end{center}
	\caption{Spectral projections computed for the Chirikov map \eqref{eq:chirikovfamily} at the interval $D=[\pi/2-0.02, \pi/2+0.02]$ with $\tilde{n}=2000$.  The depicted projection approximates the eigenfuncions at $\theta=2\pi/3$, which generates an period-3 partition of the state-space.} \label{fig:4period}
\end{figure}

\section{Conclusions \& future work} \label{sec:conclusions}

In this paper, we introduced a procedure for discretizing the unitary Koopman operator associated with an invertible, measure-preserving transformation on a compact metric space. The method relies on the construction of a periodic approximation of the dynamics, and thereby, approximates the action of the Koopman operator with a permutation. By doing so, one preserves the unitary nature of the underlying operator, which in turn, enables one to approximate the spectral properties of the operator in a weak sense. Here, the phrase ``weak convergence'' implies that approximation of the spectral measure is restricted to only a cyclic subspace, and over a restricted set of test functions (\cref{thm:testfunction}). The results nevertheless do not involve any assumptions on the spectral type, and therefore effectively handles systems with continuous or mixed spectra. The discretization procedure is constructive, and a convergent stable numerical method is formulated to compute the spectral decomposition for Lebesgue measure-preserving transformations on the $m$-torus. The general method involves solving a bipartite matching problem to obtain the periodic approximation. One then effectively traverses the cycles of this periodic approximation to compute the spectral projections and density functions of observables.  Our method is closely related to taking harmonic averages of obervable traces \cite{korda2017data}. In fact, one effectively computes the harmonic averages of the discrete periodic dynamical system using the FFT algorithm. 

One of the major open questions is how these results can be generalized to handle systems that are not necessarily invariant with respect to a Lebesgue absolutely continuous measure (e.g. invariant measures defined on fractal domains which appear in chaotic attractors). It would also be nice to derive some explicit convergence rate results for the current algorithm. Additionaly, it would be interesting to investigate whether a similar discretization procecudure can be followed for dissipative dynamics, in which the periodic approximation is replaced with a many-to-one map. Generalizations of the procedure to flows will be covered in a upcoming paper.

\iffalse
The present method does suffer from the curse of dimensionality, however we believe that sampling ideas may be employed to alleviate some of these complexity issues. For example, in certain cases the spectra can be reconstructed by traversing only a random selection of cycles in the periodic approximation. Furthermore, the necessity of constructing an explicit periodic map can be circumvented by evaluating trajectories at random initial conditions and converting them into cycles once the trajectory arrive within an epsilon bound of the starting point. Future work will be geared towards exploring these ideas in more detail.
\fi

\section*{Acknowledgments}
This project was funded by the Army Research Office (ARO) through grant W911NF-11-1-0511 under the direction of program manager Dr. Samuel Stanton, and an ARO-MURI grant W911NF-17-1-0306 under the direction of program managers Dr. Samuel Stanton and Matthew Munson

\bibliographystyle{amsplain}
\bibliography{references}

\end{document}